\documentclass[12pt]{amsart}

\usepackage{amsmath}
\pdfoutput=1 
\usepackage{amsfonts}
\usepackage{amssymb}
\usepackage{amsthm}
\usepackage{url}
\usepackage{color}

\usepackage[foot]{amsaddr}

\makeatletter
\renewcommand{\email}[2][]{%
  \ifx\emails\@empty\relax\else{\g@addto@macro\emails{,\space}}\fi%
  \@ifnotempty{#1}{\g@addto@macro\emails{\textrm{(#1)}\space}}%
  \g@addto@macro\emails{#2}%
}
\makeatother

\allowdisplaybreaks

\usepackage{upgreek}
\usepackage{subcaption}
\usepackage{tikz-cd}
\usepackage{bigdelim}
\usepackage{adjustbox}
\usepackage{multirow}
\usepackage{hhline}
\usepackage{enumerate}
\usepackage{graphicx}
\usepackage{enumitem}
\usepackage[capitalize]{cleveref}
\usepackage{comment}
\usepackage{amsthm,verbatim}
\usepackage{latexsym}
\usepackage{fancyhdr}
\usepackage{import}
\usepackage[margin=1in]{geometry}

\DeclareMathOperator*{\Euc}{Euc}
\DeclareMathOperator*{\sca}{sc}
\DeclareMathOperator*{\injrad}{injrad}
\DeclareMathOperator*{\radius}{radius}
\DeclareMathOperator*{\cent}{center}
\DeclareMathOperator*{\interior}{interior}

\DeclareMathOperator*{\genus}{genus}

\DeclareMathOperator*{\length}{length}

\DeclareMathOperator*{\id}{id}

\DeclareMathOperator*{\Area}{area}

\DeclareMathOperator*{\corner}{cor}

\DeclareMathOperator*{\sys}{sys}

\DeclareMathOperator*{\notflat}{notflat}
\DeclareMathOperator*{\Mod}{Mod}
\DeclareMathOperator*{\hor}{hor}
\DeclareMathOperator*{\skv}{skv}
\DeclareMathOperator*{\modulus}{mod}

\DeclareMathOperator*{\arcsinh}{arcsinh}

\DeclareMathOperator*{\Ext}{Ext}
\DeclareMathOperator*{\comb}{Comb}
\DeclareMathOperator*{\lb}{lb}
\newcommand*{\Comb}[2]{\comb^{#2}(\mathcal{M}_{#1})}
\newcommand*{\Ntri}{N^{\mathcal{M}}}
\newcommand*{\Ntran}{N^{\mathcal{H}}}
\newcommand*{\Ntric}{N^{\mathcal{M}}_{\lb}}
\newcommand*{\Ntranc}{N^{\mathcal{H}}_{\lb}}
\newcommand*{\Ncomb}{N^{\mathcal{M}}}
\newcommand*{\Tri}[2]{\comb^{#2}(\mathcal{M}_{#1})}
\newcommand*{\Tric}[2]{\comb^{#2}_{\lb}(\mathcal{M}_{#1})}
\newcommand*{\Tran}[2]{\comb^{#2}(\mathcal{H}_{#1})}
\newcommand*{\Tranc}[2]{\comb^{#2}_{\lb}(\mathcal{H}_{#1})}

\renewcommand{\epsilon}{\varepsilon}
\renewcommand{\Lambda}{\Uplambda}

\newcommand{\nocontentsline}[3]{}
\newcommand{\tocless}[2]{\bgroup\let\addcontentsline=\nocontentsline#1{#1}\egroup}

\newtheorem{theorem}{Theorem}[section]
\newtheorem{lemma}[theorem]{Lemma}

\newtheorem{cor}[theorem]{Corollary}
\newtheorem{prop}[theorem]{Proposition}

\newtheorem{question}{Question}
\theoremstyle{definition}
\newtheorem{remark}{Remark}
\newtheorem{definition}{Definition}

\begin{document}

\title[Triangulated Surfaces in Moduli Space]{Large genus bounds for the distribution of triangulated surfaces in moduli space}
\date{\today}
\author[Vasudevan]{Sahana Vasudevan}
\address[Sahana Vasudevan]{Department of Mathematics, Massachusetts Institute of Technology, Cambridge, MA 02139, USA}
\email{sahanav@ias.edu}

\fancyhead[CE]{Vasudevan}

\begin{abstract} Triangulated surfaces are compact Riemann surfaces equipped with a conformal triangulation by equilateral triangles. In 2004, Brooks and Makover asked how triangulated surfaces are distributed in the moduli space of Riemann surfaces as the genus tends to infinity. Mirzakhani raised this question in her 2010 ICM address. We show that in the large genus case, triangulated surfaces are well distributed in moduli space in a fairly strong sense. We do this by proving upper and lower bounds for the number of triangulated surfaces lying in a Teichm\"{u}ller ball in moduli space. In particular, we show that the number of triangulated surfaces lying in a Teichm\"{u}ller unit ball is at most exponential in the number of triangles, independent of the genus.
\end{abstract} 

\maketitle

\setcounter{tocdepth}{1}

\tableofcontents

\section{Introduction} 

In this paper, we consider compact Riemann surfaces obtained by gluing together equilateral triangles. We call such surfaces triangulated surfaces. We give genus independent bounds for the distribution of triangulated surfaces in the moduli space of Riemann surfaces. 

Brooks and Makover started the study of triangulated surfaces in the context of hyperbolic geometry. In \cite{BM04}, they asked the following question. 

\begin{question} What are the geometric properties of a large genus random triangulated surface? 
\end{question}

Here, the triangulated surface is assumed to have genus at least $2$, and is equipped with the hyperbolic metric. Randomness is with respect to the counting measure on the discrete set of triangulated surfaces. Brooks and Makover studied $T$-triangle genus $g$ triangulated surfaces in the range $T\sim 4g$, and showed that the systole of a random triangulated surface is asymptotically almost surely bounded below by a constant. The Cheeger constant and first eigenvalue of the Laplacian are also asymptotically almost surely bounded below by a constant. The diameter is asymptotically almost surely bounded above by around the logarithm of the genus. Subsequently Guth, Parlier and Young \cite{GPY11} studied Question 1 with respect to the canonical flat metric on triangulated surfaces and showed that a random triangulated surface asymptotically almost surely has large total pants length. Budzinski and Louf \cite{BL19} also studied the flat metric on triangulated surfaces when $T\sim \theta g$ (for constants $\theta>4$). They showed that a random point on a random triangulated surface asymptotically almost surely does not lie close to a short loop of nontrivial homotopy class.

A related question is: 

\begin{question} What are the geometric properties of a large genus random hyperbolic surface?
\end{question}

Here, randomness is with respect to the Weil-Petersson measure on moduli space $\mathcal{M}_g$. Mirzakhani studied Question 2 in \cite{Mir13} (following \cite{Mir07a} and \cite{Mir07b}). The systole of a surface is bounded below by a constant with positive probability in the large genus limit. The Cheeger constant and first eigenvalue of the Laplacian are asymptotically almost surely bounded below by a constant. The diameter is asymptotically almost surely bounded above by around the logarithm of the genus. Mirzakhani also proved in \cite{Mir13} that a random point on a random hyperbolic surface almost surely does not lie close to a short loop of nontrivial homotopy class. Guth, Parlier and Young \cite{GPY11} studied Question 2 and showed that a random hyperbolic surface asymptotically almost surely has large total pants length. 

Further results on the geometry of random surfaces have been obtained in \cite{BCP19a}, \cite{BCP19b}, \cite{BCP21}, \cite{LW21}, \cite{MP19}, \cite{Mon20}, \cite{MT20}, \cite{NWX20}, \cite{WX18} and \cite{WX21}. Similarities between answers to Question 1 and Question 2 motivate the following question.

\begin{question} How are triangulated surfaces distributed in the moduli space of Riemann surfaces, quantitatively? 
\end{question}

Question 3 was first asked by Brooks and Makover in \cite{BM04}. Subsequently it has been raised in Mirzakhani's 2010 ICM address \cite{Mir10} as well as in \cite{BCP19a}, \cite{CP12}, \cite{GPY11} and \cite{Mir13}.

We answer Question 3 by proving well distribution results for triangulated surfaces. Our answer is with respect to the Teichm\"{u}ller metric (see \cref{metrics}). One consequence of our main results (stated in \cref{statement of main results}) is the following simplified answer to Question 3. 

\begin{theorem}\label{simplified upper bound} In a Teichm\"{u}ller $1$-ball in $\mathcal{M}_g$, there are at most $C^T$ number of $T$-triangle genus $g$ triangulated surfaces where $C$ is a constant independent of $g$ and $T$.
\end{theorem}

Henceforth, all generic universal constants will be denoted by $C$. In \cite{FKM13}, Fletcher, Kahn and Markovic showed that the number of Teichm\"{u}ller $1$-balls needed to cover the thick part of $\mathcal{M}_g$ is around $g^{2g}$. For $g=O(T)$, the number of $T$-triangle genus $g$ surfaces is around $g^{2g}C^T$, which was computed by Budzinski and Louf in \cite{BL19}. In this range, for triangulated surfaces to be well distributed in moduli space means there are around $C^T$ surfaces in each $1$-ball. \cref{simplified upper bound} gives such an upper bound. We also prove a lower bound (with different constants), which is stated in \cref{statement of main results}.

Note that the sphere admits a unique conformal structure, and the number of distinct $T$-triangle triangulations of the sphere is also around $C^T$. In this sense, \cref{simplified upper bound} is genus independent. When we fix the conformal class, the higher genus case behaves just like the genus $0$ case.

Triangulated surfaces are combinatorial objects, while moduli space parametrizes geometric objects. The difficulty in proving \cref{simplified upper bound} lies in the difficulty of distinguishing if two particular triangulated surfaces are close in moduli space. Moreover, given a triangulated surface, it is difficult to explicitly determine the hyperbolic metric on it. Given two marked hyperbolic surfaces, it is difficult to determine if they are close in moduli space. Indeed, the results of Mirzakhani on random hyperbolic surfaces show that the usual geometric quantities like systole, diameter and Cheeger constant fail to distinguish most hyperbolic surfaces. 

\subsection{Statement of main results}\label{statement of main results} We prove two results which describe the distribution of triangulated surfaces in $\mathcal{M}_g$. The following lower bound describes when a surface in moduli space can be approximated by a triangulated surface.

\begin{theorem}\label{lower bound} Let $g\geq 2$. Let $X\in \mathcal{M}_g$ equipped with the conformal hyperbolic metric. Let $$R=\sum_{\substack{\gamma \text{ \normalfont simple closed geodesic on } X\\\length(\gamma)<2\arcsinh(1)}}\length(\gamma)^{-1}.$$ Then for $r\in (0,1]$, there exists a $T(r)$-triangle triangulated surface inside a Teichm\"{u}ller $r$-ball around $X$, where $T(r)\leq C(R+g)r^{-2}$. Here, $C$ is a universal constant (independent of $g$, $T$, $r$ and $R$).
\end{theorem}

\begin{remark} In particular, if $X$ lies in the thick part of $\mathcal{M}_g$, meaning $\sys X\geq 2\arcsinh(1)$, then there exists a  $Cgr^{-2}$-triangle triangulated surface within Teichm\"{u}ller distance $r$ from $X$. 
\end{remark}

\begin{remark} The condition that there should not be too short geodesics on $X$ is necessary, as we shall see in \cref{tri systole}.
\end{remark}

We also have the following upper bound regarding the distribution of triangulated surfaces in $\mathcal{M}_g$. 

\begin{theorem}\label{upper bound} There exists at most $C^{T+rg}$ number of $T$-triangle triangulated surfaces in a Teichm\"{u}ller $r$-ball in $\mathcal{M}_g$. Here, $C$ is a universal constant (independent of $g$, $T$ and $r$).
\end{theorem}

\begin{remark} By Euler characteristic conditions, $g=O(T)$. So substituting $r=1$ in \cref{upper bound} we obtain \cref{simplified upper bound}. 
\end{remark}

\begin{remark}\label{error} In a Teichm\"{u}ller $1$-ball in the thick part of $\mathcal{M}_g$, the lower and upper bounds for the number of triangulated surfaces given by \cref{lower bound} and \cref{upper bound} differ by a multiple of $\exp(O(T))$ for $T/g$ sufficiently large.
\end{remark}

\begin{remark} It is useful to first ask the combinatorial question: what is the number of $T$-triangle genus $g$ triangulated surfaces, as a function of $T$ and $g$? The best bounds for this question in the linear range are proved in \cite[Theorem 3]{BL19}. This bound has a multiplicative error term of $\exp(o(T))$, in contrast to our slightly worse error of $\exp(O(T))$ as stated in \cref{error}. (The constants in both error terms are independent of the genus.)
\end{remark}

\begin{remark} \cref{upper bound} is most interesting in the range $T\sim \theta g$ where $\theta\geq 4$ is a constant. In this range, we obtain that the number of triangulated surfaces in a Teichm\"{u}ller $1$-ball grows roughly exponentially in $g$ as $g\to \infty$. This is similar to how integer points are distributed in high dimensional Euclidean space (which is a useful toy example). In $\mathbb{R}^n$, a radius $\sim n^{1/2}$ ball has volume $\sim 1$. Such a ball contains around $C^n$ integer points. It is not possible to give a better bound for the number of integer points as a small translation of the ball can change this number by an exponential multiplicative factor. 
\end{remark}

\begin{remark}\label{metrics} There are several possible choices of metrics on moduli space. In Teichm\"{u}ller theory, the Teichm\"{u}ller metric is a natural choice. Since we are also interested in the hyperbolic geometry of individual surfaces in moduli space, we may also consider the bi-Lipschitz metric on moduli spaces. In this metric, the distance between two surfaces measures how far apart their hyperbolic metrics are. It turns out that the bi-Lipschitz metric is comparable to the Teichm\"{u}ller metric, with genus independent constants (see \cref{qc to lipschitz}). So \cref{lower bound} and \cref{upper bound} hold with respect to the bi-Lipschitz metric also. Another natural choice of metric on moduli space is the Weil-Petersson metric. Note that the Weil-Petersson volume of $\mathcal{M}_g$ is around $g^{2g}$, as computed by Penner in \cite{Pen92} and Grushevsky in \cite{Gru01}, and subsequently improved by Mirzakhani and Zograf in \cite{MZ15}. Up to an exponential factor of $T$ this number is also the approximate number of $T$-triangle genus $g$ triangulated surfaces. So we may ask if an analogue of our main results hold for the Weil-Petersson metric as well. However, we do not know the answer to this question because we do not yet understand the large genus local geometry of the Weil-Petersson metric in a sufficiently detailed manner. 
\end{remark}

\subsection{Key ideas in the proof of \cref{lower bound}}

Our approach is based on a characterization of the Teichm\"{u}ller metric in \cite{Ker80} using extremal length and Jenkins-Strebel differentials, which we explain in Section 3. To show \cref{lower bound}, given a surface in moduli space we construct a certain nicely behaved triangulation of it, take the associated triangulated surface. We use the characterization of the Teichm\"{u}ller metric above to show that the triangulated surface is close to the original surface in moduli space.  We do this in Section 4.

\subsection{Key ideas in the proof of \cref{upper bound}} \label{key ideas upper bound}

There are roughly four parts to our proof:

\begin{enumerate}

\item Riemann surfaces equipped with a holomorphic $1$-form are called translation surfaces. We first consider the subset of triangulated surfaces that are actually translation surfaces where the associated holomorphic $1$-form is compatible with the triangulation. We call such surfaces combinatorial translation surfaces, a term we define precisely in Section 2. In this situation, the $1$-form gives us a cohomology class, which we then deal with using Hodge theory. Suppose two combinatorial translation surfaces are close together in moduli space and so are their cohomology classes. Then, we show that constraints coming from Hodge norms imply geometric constraints on how close or far apart vertices, edges, and faces of the two surfaces must be to each other. A combinatorial argument shows that these geometric constraints imply that the two triangulations are close except on a part of the surface with much smaller genus. So we reduce the counting problem for combinatorial translation surfaces to the counting problem for triangulated surfaces in a lower dimensional moduli space. As a result, we get bounds on combinatorial translation surfaces in terms of bounds on triangulated surfaces. We do this in Section 6.

\item Given any triangulated surface, there exists a degree six branched cover which is a combinatorial translation surface. We enumerate the number of possibilities for such covers and study the possibilities for where the branch points lie, to get bounds on triangulated surfaces given bounds on combinatorial translation surfaces. Combining with our previous bounds described in (1), we obtain recursive upper bounds for the number of triangulated surfaces lying in a ball in moduli space. We solve these recursive bounds to show \cref{upper bound}. We do this in Section 7.

\item For the technicalities in (1) and (2) to work, we require the use of bounded degree triangulations instead of arbitrary triangulations. In Section 5, we show that any triangulated surface may be approximated by a bounded degree triangulation in a way that increases the number of triangles by at most a constant factor.

\end{enumerate}

These three sections contain the key steps of the proof of \cref{upper bound}. Before that, in Section 2, we introduce various definitions and prove several preliminary results. In Section 3, we prove results related to quasiconformal maps and the large genus geometry of Teichm\"{u}ller space.   

\subsection{Comments and references} Non-quantitative versions of Question 3 have been studied in number theory. Belyi's theorem states that Riemann surfaces defined over the algebraic numbers $\overline{\mathbb{Q}}$ are exactly the Riemann surfaces which admit a branched cover (Belyi map) to $\mathbb{P}^1$ branched only at $0$, $1$ and $\infty$. Belyi maps give triangulations on the Belyi surface, and conversely, triangulations give Belyi maps. Note that Riemann surfaces defined over $\overline{\mathbb{Q}}$ are dense in $\mathcal{M}_g$, which is a non-quantitative answer to Question 3. Given this context, \cref{lower bound} may be interpreted as a quantitative version of this statement. It describes how well one can approximate an arbitrary surface in $\mathcal{M}_g$ by a Belyi surface with respect to the Teichm\"{u}ller metric, in terms of the degree of the Belyi map. Another approach to approximating arbitrary Riemann surfaces by Belyi surfaces is described by Bishop in \cite{Bis14} and \cite[Section 15]{Bis15}. 

The study of random triangulations is a central topic of research in probability theory. In the large genus setting, progress has been made in the study of local limits of triangulations. In the range $T\sim 4g$ (which is the probabilistically expected range, as shown by Gamburd in \cite{Gam06}), local limits do not exist since the expected surface has very few vertices and very high degrees of vertices. In the range $T\sim \theta g$ where $\theta>4$, convergence does occur. Planar stochastic hyperbolic triangulations introduced by Curien in \cite{Cur16} are a family of random triangulations in the plane. These were conjectured by Benjamini and Curien to be local limits of high genus random triangulations. This was proved by Budzinski and Louf in \cite{BL19}. However, global questions, such as scaling limits, are difficult to understand in the large genus case. See \cite{LeG07}, \cite{Mar18} and \cite{Mie13} for some results on scaling limits of random maps in the planar setting.

In \cite{BS94}, Buser and Sarnak related the homological systole of a hyperbolic Riemann surface to the systole of its Jacobian. In \cite{BPS12}, Balacheff, Parlier and Sabourau gave a way to find a minimal homology basis on a hyperbolic Riemann surface, and used that to deduce more general results about the geometry of its Jacobian. The original method to prove \cref{upper bound} involved bounding the number of ways to express a hyperbolic Riemann surface as a triangulated surface by studying the geometry of its Jacobian, using \cite{BPS12}. Then, it turned out that these ideas were not necessary to prove \cref{upper bound}, so they do not appear in the proofs henceforth.

\subsection*{Acknowledgments} I thank my advisor Larry Guth for many inspiring discussions and enormous help with this paper. I thank Curt McMullen for many comments on the geometry of Teichm\"{u}ller space. I thank Chris Bishop for correspondence on quasiconformal maps. I have learned a lot from conversations with Morris Ang, Scott Sheffield and Yilin Wang on random triangulations in probability theory. I thank the referee for communicating to me the proof of \cref{qc to lipschitz} due to Maxime Fortier Bourque. I thank Robert Burklund, Yilin Wang and the referee for helping me improve the writing. This research was supported by the National Science Foundation Graduate Research Fellowship Program (under Grant No. 1745302) and the Simons Foundation (under Larry Guth's Simons Investigator award).

\section{Preliminaries}

\subsection{Triangulated surfaces} 

Let $S_g$ be a genus $g$ surface, and $S_{g,b}$ a genus $g$ surface with $b$ boundary components. Let $S$ be a metrized simplicial $2$-complex wherein each $2$-simplex is an oriented unit equilateral triangle and gluings of faces preserve orientations. We call $S$ a triangulated surface if it is homeomorphic to $S_g$ for some $g\geq 0$ and a triangulated surface with boundary if it is homeomorphic to $S_{g,b}$ for some $g,b\geq 0$. If $S$ is a triangulated surface with or without boundary, we may consider each equilateral triangle as embedded in $\mathbb{C}$ with vertices at $0$, $1$ and $\frac{1}{2}+\frac{\sqrt{3}}{2}i$. The complex structure on each equilateral triangle of $S$ is preserved when edges are glued. Extending the complex structure over the interior vertices of $S$, we obtain a canonical complex structure on $S$. In this way, we consider $S$ as a Riemann surface.

\begin{remark} The construction of triangulated surfaces (without boundary) in \cite{BM04} is done slightly differently. Instead of gluing together equilateral triangles, hyperbolic ideal triangles are used, and the resulting surface with cusps is compactified. One may check that these two constructions result in the same Riemann surface.
\end{remark}

Given a triangulated surface $S$ (with or without boundary), we denote by $V(S)$ the vertices of $S$, $E(S)$ the edges of $S$, and $F(S)$ the triangular faces of $S$. Given $v\in V(S)$ the degree of $v$ is the number of edges emanating from $v$. We also denote by $V_{> 6}(S)$, $V_{\neq 6}(S)$ and $V_{<6}(S)$ the set of vertices of $S$ of degree strictly greater than $6$, not equal to $6$ and strictly less than $6$, respectively.

\subsection{Space of triangulated surfaces} 

\begin{definition} $\Tri{g}{T}$ is the set of all triangulated surfaces of genus $g$ with $T$ triangles, up to simplicial isomorphism.
\end{definition} 

We have a map $$\Phi: \textstyle\Comb{g}{T}\to \mathcal{M}_g$$ which takes a triangulated surface to its underlying Riemann surface in $\mathcal{M}_g$. We also have canonical biholomorphisms $\Phi_S:S\to \Phi(S)$ for all $S\in \Comb{g}{T}$. 

Note that $T\geq 2$ for $\Tri{g}{T}$ to be nonempty; moreover, Euler characteristic conditions imply that $T\geq 4g-4$. Together these imply $g/T\leq 1/2$. In the future we will implicitly assume this condition. 

Finally, we denote by $\mathcal{M}_{\leq g}$ the union of $\mathcal{M}_{g'}$ over all $g'\leq g$. Similarly, we denote by $\Tri{\leq g}{\leq T}$ the union of $\Tri{g'}{T'}$ over all $g'\leq g$, $T'\leq T$. We will use this type of notation with respect to other spaces we will define later as well.

\subsection{Translation surfaces and combinatorial translation surfaces}

\begin{definition} A translation surface is a pair $(X,\omega)$ where $X$ is a Riemann surface and $\omega$ is a holomorphic $1$-form on $X$. 
\end{definition} 

The metric $|\omega|$ is a flat metric on $X$ with singularities at zeros of $\omega$. See \cite{Wri15} for an introduction to translation surfaces. The following result gives an alternative definition of translation surfaces. 

\begin{prop}[\cite{Wri15}, Proposition 1.6 and Proposition 1.8]\label{equivalent definition of translation surfaces} Any translation surface $(X,\omega)$ can be expressed in the following manner: $X$ is the union of a collection of polygons $P_1,...,P_n$ in $\mathbb{C}$ together with a choice of identification of parallel boundary edges of equal length on opposite sides, and $\omega$ is the $1$-form $dz$ on each polygon. Similarly, any collection of polygons $P_1,...,P_n\subset \mathbb{C}$ with edge identifications as above defines a translation surface.
\end{prop}

We shall see that some triangulated surfaces are canonically translation surfaces as well. To this end, we define a combinatorial translation structure on a triangulated surface. Let $S$ be a triangulated surface. 

\begin{definition} \label{def of comb tran surface}  A combinatorial translation structure on $S$ is an assignment, to each edge $e\in E(S)$ and vertex $v\in V(S)$ such that $e$ emanates from $v$, a $6$th root of unity $\zeta(e,v)$ (called a directional weight). Directional weights must satisfy the following two properties:

\begin{enumerate}
\item if $e$ contains the vertices $v$ and $w$, then $\zeta(e,v)=-\zeta(e,w)$ and 

\item if $e_1$ and $e_2$ are two edges emanating from a vertex $v$ that lie on a triangle of $S$ such that $e_1$ lies counterclockwise from $e_2$ according to the orientation on $S$, then $\zeta(e_1,v)=e^{\pi i/3}\zeta(e_2,v)$.
\end{enumerate}
\end{definition}

Conditions 1 and 2 imply that for each triangle, there are only two possibilities for directional weights, which we label as Type A and Type B as seen in \cref{Type A} and \cref{Type B}, respectively.

\begin{figure}[ht] \begin{minipage}{.49\textwidth} \centering \hspace{.5in}%% Creator: Inkscape 1.0.2-2 (e86c870879, 2021-01-15), www.inkscape.org
%% PDF/EPS/PS + LaTeX output extension by Johan Engelen, 2010
%% Accompanies image file 'Fig241.pdf' (pdf, eps, ps)
%%
%% To include the image in your LaTeX document, write
%%   \input{<filename>.pdf_tex}
%%  instead of
%%   \includegraphics{<filename>.pdf}
%% To scale the image, write
%%   \def\svgwidth{<desired width>}
%%   \input{<filename>.pdf_tex}
%%  instead of
%%   \includegraphics[width=<desired width>]{<filename>.pdf}
%%
%% Images with a different path to the parent latex file can
%% be accessed with the `import' package (which may need to be
%% installed) using
%%   \usepackage{import}
%% in the preamble, and then including the image with
%%   \import{<path to file>}{<filename>.pdf_tex}
%% Alternatively, one can specify
%%   \graphicspath{{<path to file>/}}
%% 
%% For more information, please see info/svg-inkscape on CTAN:
%%   http://tug.ctan.org/tex-archive/info/svg-inkscape
%%
\begingroup%
  \makeatletter%
  \providecommand\color[2][]{%
    \errmessage{(Inkscape) Color is used for the text in Inkscape, but the package 'color.sty' is not loaded}%
    \renewcommand\color[2][]{}%
  }%
  \providecommand\transparent[1]{%
    \errmessage{(Inkscape) Transparency is used (non-zero) for the text in Inkscape, but the package 'transparent.sty' is not loaded}%
    \renewcommand\transparent[1]{}%
  }%
  \providecommand\rotatebox[2]{#2}%
  \newcommand*\fsize{\dimexpr\f@size pt\relax}%
  \newcommand*\lineheight[1]{\fontsize{\fsize}{#1\fsize}\selectfont}%
  \ifx\svgwidth\undefined%
    \setlength{\unitlength}{154.9383566bp}%
    \ifx\svgscale\undefined%
      \relax%
    \else%
      \setlength{\unitlength}{\unitlength * \real{\svgscale}}%
    \fi%
  \else%
    \setlength{\unitlength}{\svgwidth}%
  \fi%
  \global\let\svgwidth\undefined%
  \global\let\svgscale\undefined%
  \makeatother%
  \begin{picture}(1,0.61697289)%
    \lineheight{1}%
    \setlength\tabcolsep{0pt}%
    \put(0,0){\includegraphics[width=\unitlength,page=1]{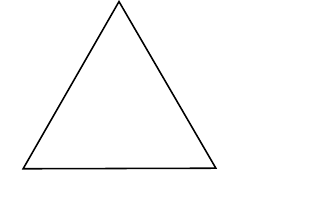}}%
    \put(0.18157797,0.00701422){\makebox(0,0)[lt]{\lineheight{1.25}\smash{\begin{tabular}[t]{l}$1$\end{tabular}}}}%
    \put(0.46994159,0.0139294){\makebox(0,0)[lt]{\lineheight{1.25}\smash{\begin{tabular}[t]{l}$e^{\pi i}$\end{tabular}}}}%
    \put(0.63313993,0.22484279){\makebox(0,0)[lt]{\lineheight{1.25}\smash{\begin{tabular}[t]{l}$e^{2\pi i/3}$\end{tabular}}}}%
    \put(0.48031436,0.4786303){\makebox(0,0)[lt]{\lineheight{1.25}\smash{\begin{tabular}[t]{l}$e^{5\pi i/3}$\end{tabular}}}}%
    \put(0.12591065,0.48208796){\makebox(0,0)[lt]{\lineheight{1.25}\smash{\begin{tabular}[t]{l}$e^{4\pi i/3}$\end{tabular}}}}%
    \put(-0.00340356,0.23175797){\makebox(0,0)[lt]{\lineheight{1.25}\smash{\begin{tabular}[t]{l}$e^{\pi i/3}$\end{tabular}}}}%
  \end{picture}%
\endgroup%
\caption{Type A triangle}\label{Type A} \end{minipage}  \begin{minipage}{0.49\textwidth} \centering \vspace{.05in}\hspace{.5in}%% Creator: Inkscape 1.0.2-2 (e86c870879, 2021-01-15), www.inkscape.org
%% PDF/EPS/PS + LaTeX output extension by Johan Engelen, 2010
%% Accompanies image file 'Fig242.pdf' (pdf, eps, ps)
%%
%% To include the image in your LaTeX document, write
%%   \input{<filename>.pdf_tex}
%%  instead of
%%   \includegraphics{<filename>.pdf}
%% To scale the image, write
%%   \def\svgwidth{<desired width>}
%%   \input{<filename>.pdf_tex}
%%  instead of
%%   \includegraphics[width=<desired width>]{<filename>.pdf}
%%
%% Images with a different path to the parent latex file can
%% be accessed with the `import' package (which may need to be
%% installed) using
%%   \usepackage{import}
%% in the preamble, and then including the image with
%%   \import{<path to file>}{<filename>.pdf_tex}
%% Alternatively, one can specify
%%   \graphicspath{{<path to file>/}}
%% 
%% For more information, please see info/svg-inkscape on CTAN:
%%   http://tug.ctan.org/tex-archive/info/svg-inkscape
%%
\begingroup%
  \makeatletter%
  \providecommand\color[2][]{%
    \errmessage{(Inkscape) Color is used for the text in Inkscape, but the package 'color.sty' is not loaded}%
    \renewcommand\color[2][]{}%
  }%
  \providecommand\transparent[1]{%
    \errmessage{(Inkscape) Transparency is used (non-zero) for the text in Inkscape, but the package 'transparent.sty' is not loaded}%
    \renewcommand\transparent[1]{}%
  }%
  \providecommand\rotatebox[2]{#2}%
  \newcommand*\fsize{\dimexpr\f@size pt\relax}%
  \newcommand*\lineheight[1]{\fontsize{\fsize}{#1\fsize}\selectfont}%
  \ifx\svgwidth\undefined%
    \setlength{\unitlength}{156.70621467bp}%
    \ifx\svgscale\undefined%
      \relax%
    \else%
      \setlength{\unitlength}{\unitlength * \real{\svgscale}}%
    \fi%
  \else%
    \setlength{\unitlength}{\svgwidth}%
  \fi%
  \global\let\svgwidth\undefined%
  \global\let\svgscale\undefined%
  \makeatother%
  \begin{picture}(1,0.61240557)%
    \lineheight{1}%
    \setlength\tabcolsep{0pt}%
    \put(0,0){\includegraphics[width=\unitlength,page=1]{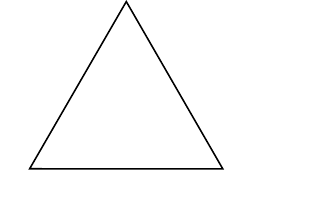}}%
    \put(0.52070479,0.01377226){\makebox(0,0)[lt]{\lineheight{1.25}\smash{\begin{tabular}[t]{l}$1$\end{tabular}}}}%
    \put(0.20756209,0.01377226){\makebox(0,0)[lt]{\lineheight{1.25}\smash{\begin{tabular}[t]{l}$e^{\pi i}$\end{tabular}}}}%
    \put(-0.00336517,0.22640839){\makebox(0,0)[lt]{\lineheight{1.25}\smash{\begin{tabular}[t]{l}$e^{4\pi i/3}$\end{tabular}}}}%
    \put(0.15730887,0.48519567){\makebox(0,0)[lt]{\lineheight{1.25}\smash{\begin{tabular}[t]{l}$e^{\pi i/3}$\end{tabular}}}}%
    \put(0.50087698,0.48348645){\makebox(0,0)[lt]{\lineheight{1.25}\smash{\begin{tabular}[t]{l}$e^{2\pi i/3}$\end{tabular}}}}%
    \put(0.63727861,0.22469917){\makebox(0,0)[lt]{\lineheight{1.25}\smash{\begin{tabular}[t]{l}$e^{5\pi i/3}$\end{tabular}}}}%
  \end{picture}%
\endgroup%
 \caption{Type B triangle}\label{Type B} \end{minipage}\end{figure}

Then, we have the following proposition. 

\begin{prop}\label{combinatorial translation structure gives translation surface} Each triangulated surface $S$ that admits a combinatorial translation structure is a translation surface wherein the associated flat metric agrees with the flat metric coming from the triangulation. Moreover, the associated holomorphic $1$-form is canonical in the sense that it only depends on $S$ and the combinatorial translation structure.
\end{prop}

\begin{proof} Rotating as necessary, we identify Type A triangles with the equilateral triangle in $\mathbb{C}$ having vertices at $0$, $1$ and $\frac{1}{2}+\frac{\sqrt{3}}{2}i$, and Type B triangles with the equilateral triangle in $\mathbb{C}$ having vertices at $0$, $1$ and $\frac{1}{2}-\frac{\sqrt{3}}{2}i$. Conditions 1 and 2 in the definition of combinatorial translation surface imply that all edge identifications must be of a Type A triangle with a Type B triangle on opposite sides along parallel edges. By \cref{equivalent definition of translation surfaces}, $S$ is a translation surface. Under the identification of Type A triangles and Type B triangles with triangles in $\mathbb{C}$ described above, the $1$-forms $dz$ on each triangle glue to give a $1$-form $\phi$ on $S$. The $1$-form $\phi$ only depends on $S$ and the combinatorial translation structure. Finally $|dz|$ is simply the Euclidean metric on each triangle.
\end{proof}

Note that given a combinatorial translation structure on $S$, then for $0\leq i\leq 6$ we have another combinatorial translation structure on $S$ obtained by multiplying each directional weight by $e^{\pi i/3}$. These six structures are the only valid combinatorial translation structures on $S$:  once we assign directional weights to one triangle on $S$, there is only one choice for all other directional weights.

\begin{definition} A combinatorial translation surface is a triangulated surface equipped with a combinatorial translation structure. 
\end{definition}

\begin{definition} $\Tran{g}{T}$ is the set of combinatorial translation surfaces of genus $g$ with $T$ triangles.
\end{definition}

By \cref{combinatorial translation structure gives translation surface}, any $S\in \Tran{g}{T}$ determines a canonical holomorphic $1$-form on $S$ that we denote by $\phi_S$. We call the flat metric on $S$ the $S$-metric. Its length element is $ds_{S}=|\phi_S|$, and area element is $|\phi_S|^2$. Distances in this metric shall be denoted by $d_{S}(\cdot,\cdot)$. As in the case of triangulated surfaces, we denote by $$\Phi:\textstyle\Tran{g}{T}\to \mathcal{T}_g$$ the map which sends a  combinatorial translation surface to the underlying Riemann surface.

\subsection{Extremal length on annuli} Let $A$ be a Riemann surface that is topologically an annulus. By the uniformization theorem $A$ is biholomorphic to $$A(r)=\{z\in \mathbb{C}|1<|z|<r\}$$ for some $r>1$. The modulus of $A$, denoted $\modulus(A)$, is the quantity $(1/2\pi)\log r$. 

Now, denote by $\gamma$ a generator of $H_1(A,\mathbb{Z})$. 

\begin{definition} Given a Riemannian metric $\rho$ on $A$, the quantity $$\textstyle\length_\rho(\gamma)$$ is defined to be the infimum of lengths in the $\rho$-metric over all curves representing $\gamma$. 
\end{definition}

\begin{definition} The extremal length of $\gamma$ on $A$ is $$\textstyle\Ext_A(\gamma)=\displaystyle\sup_\rho\frac{\length_\rho(\gamma)^2}{\Area_\rho(A)}$$ where the supremum is taken over all conformal metrics $\rho$ on $A$.
\end{definition} The following result is proved in \cite[Section 1.D]{Ahl66}.  

\begin{prop}\label{extremal length of annulus} We have, $\Ext_A(\gamma)=\modulus(A)^{-1}$.
\end{prop}

\subsection{Hyperbolic metric on a triangulated surface} In this section we prove two lemmas about the hyperbolic metric on a triangulated surface. The first lemma is about short geodesics.

\begin{lemma}\label{tri systole} Let $S$ be a $T$-triangle triangulated surface of genus $g$, and $X=\Phi(S)$. Denote by $\rho_X$ the hyperbolic metric on $X$. Then $$\sum_{\substack{\gamma \text{ \normalfont simple closed geodesic on } X\\\length_{\rho_X}(\gamma)<2\arcsinh(1)}}\textstyle\length_{\rho_X}(\gamma)^{-1}\leq CT$$ where $C$ is a universal constant. 
\end{lemma} 

\begin{proof} Let $\gamma$ be a simple closed geodesic on $X$ with $\length_{\rho_X}(\gamma)<2\arcsinh(1)$. By the collar theorem \cite[Theorem 4.1.1]{Bus10}, $\gamma$ has an associated hyperbolic annular collar $A_\gamma$ of width $$\textstyle C\log (\length_{\rho_X}(\gamma)^{-1}).$$ Moreover, the collars $A_\gamma$ associated to the short geodesics are all mutually disjoint. The modulus of $A_\gamma$ is $$\textstyle C\length_{\rho_X}(\gamma)^{-1}.$$ By \cref{extremal length of annulus}, $$\frac{\sys_{S}(A_\gamma)^2}{\Area_{S}(A_\gamma)}\leq C\textstyle\length_{\rho_X}(\gamma).$$ Since $\sys_{S}(A_\gamma)\geq 1$, we have $$\textstyle\Area_{S}(A_\gamma)\geq C\textstyle\length_{\rho_X}(\gamma)^{-1}.$$ Since the $A_\gamma$ are all disjoint, 
\begin{align*}\sum_{\substack{\gamma \text{ \normalfont simple closed geodesic on } X\\\length_{\rho_X}(\gamma)<2\arcsinh(1)}}\textstyle\length_{\rho_X}(\gamma)^{-1} &\leq C\displaystyle\sum_{\substack{\gamma \text{ \normalfont simple closed geodesic on } X\\\length_{\rho_X}(\gamma)<2\arcsinh(1)}}\textstyle\Area_{S}(A_\gamma)\\&\leq CT 
\end{align*}
as desired.
\end{proof}

Next, we show that the hyperbolic metric on a triangulated surface admits a nicely behaved covering by hyperbolic balls.

\begin{lemma}\label{covering lemma} Let $X$ be a hyperbolic surface, with metric $\rho_X$. Suppose $$\sum_{\substack{\gamma \text{ \normalfont simple closed geodesic on } X\\\length_{\rho_X}(\gamma)<2\arcsinh(1)}}\textstyle\length_{\rho_X}(\gamma)^{-1}\leq R.$$ Then, there exist hyperbolic disks $U_1,...,U_N$, $V_1,...,V_N$ and $W_1,...,W_N$ on $X$ such that the following conditions are satisfied:

\begin{enumerate}

\item The $\{W_i\}$ together cover $X$.

\item $$\cent(U_i)=\cent(V_i)=\cent(W_i)$$ and $$\radius(U_i)=2\radius(V_i)=4\radius(W_i)\leq \arcsinh(1)/2.$$

\item If $U_i$ nontrivially intersects $U_j$, then $\radius(U_i)\leq C\radius(U_j)$.

\item Any point $x\in X$ is contained in at most $C$ of the $U_i$. 

\item Any $U_i$ nontrivially intersects at most $C$ of the $U_j$.

\item $N\leq C(R+g)$.

\end{enumerate}
Here, $C$ is a universal constant.

\end{lemma}

\begin{remark} Once the $U_i$ and $W_i$ are constructed, the existence of the $V_i$ is clear. \cref{covering lemma} will be used to prove \cref{lower bound} in Section 5, and in that proof the $V_i$ will become relevant. \cref{covering lemma} will also be used in \cref{8.5}.
\end{remark}

\begin{proof} We divide $X$ into the thick part and thin part. The thick part has injectivity radius at least $\arcsinh(1)$. The thin part is the union of disjoint annuli around each short geodesic.

The thick part has area at most $Cg$. On this part, we take a maximal $\arcsinh(1)/16$-separated set. This set has around $Cg$ points. Let the $W_i$ be radius $\arcsinh(1)/8$ disks around the points in the separated set, the $V_i$ radius $\arcsinh(1)/4$ disks around these points and the $U_i$ radius $\arcsinh(1)/2$ disks around these points.

Next, we construct the disks on the thin part. Let $\gamma$ be a geodesic on $X$ with length less than $2\arcsinh(1)$. Let $A_\gamma$ be the annulus associated to $\gamma$, which is a connected component of the thin part. By the collar theorem, each annulus $A_\gamma$ has width approximately $$C\textstyle\log(\length_{\rho_X}(\gamma)^{-1}).$$ We have coordinates $(r,\theta)$ on $A_j$ where $r=0$ on $\gamma$ and $r\in [w(A_\gamma)/2,w(A_\gamma)/2]$ such that the hyperbolic metric on $A_\gamma$ is $$d\rho_X^2=dr^2+\cosh^2 rd\theta^2.$$ The injectivity radius of $A_\gamma$ at any point $(r,\theta)$ (denoted $\injrad_{(r,\theta)}(A_\gamma)$) is around $$C\textstyle\length_{\rho_X}(\gamma)\cosh r.$$ 

Now, for $s\in \mathbb{N}\cap [-w(A_\gamma)/2,w(A_\gamma)/2]$, let $$A_\gamma^s=A_\gamma\cap \{r\in [s-1,s+1]\}.$$ Denote by $\injrad(A_\gamma^s)$ the injectivity radius of $A_\gamma^s$. Choose a $\injrad(A_\gamma^s)/16$-separated set on $A_\gamma^s$ to be the centers of $W_i$, so that the $W_i$ are radius $\injrad(A_\gamma^s)/8$ disks around these centers, the $V_i$ are radius $\injrad(A_\gamma^s)/4$ disks, and the $U_i$ are radius $\injrad(A_\gamma^s)/2$ disks. In each $A_\gamma^s$ there are at most $$(C \textstyle\length_{\rho_X}(\gamma)\cosh s)^{-1}$$ centers. In total, there are at most 
\begin{align*}C\sum_{s=1}^{\lceil w(A_\gamma)\rceil} (\textstyle\length_{\rho_X}(\gamma)\cosh s)^{-1}&\leq C\textstyle\length_{\rho_X}(\gamma)^{-1}\displaystyle\int_{s=0}^\infty (\cosh s)^{-1}\\&\leq C\textstyle\length_{\rho_X}(\gamma)^{-1}
\end{align*} centers in $A_\gamma$.

We simply take all the $U_i$, $V_i$ and $W_i$ constructed on the thick part and thin part. By construction, conditions 1, 2, 3, 4 and 5 are satisfied. The total number of $U_i$ (or $V_i$ or $W_i$) is at most $$C\left(g+\sum_{\substack{\gamma \text{ \normalfont simple closed geodesic on } X\\\length_{\rho_X}(\gamma)<2\arcsinh(1)}}\textstyle\length_{\rho_X}(\gamma)^{-1}\right)\leq C(R+g)$$ which shows condition 6.
\end{proof}

\subsection{Conformal doubles and triangulated surfaces} Denote by $S_{g,b}$ a surface of genus $g$ with $b$ boundary components. Let $X$ be a Riemann surface of genus $g$ with $b$ boundary components, meaning that $X$ is homeomorphic to $S_{g,b}$. 

\begin{definition}\label{conformal mirror} The conformal mirror $X^{-1}$ of $X$ is the Riemann surface whose local coordinates are obtained by composing each local coordinate $z$ on $X$ with the anti-holomorphic map $z\to \overline{z}$. 
\end{definition} 

We may construct another Riemann surface, $X^d$, by gluing $X$ and $X^{-1}$ along their boundaries $\partial X$ and $\partial X^{-1}$ which are canonically identified. The complex structure on $X^d$ in a neighborhood $U$ of a point $p\in \partial X$ is given as follows. If $z$ is a local coordinate on $X$ under which $U\cap X$ is conformally a half-disk around $p$, then the local coordinate on $U$ which is $z$ on $U\cap X$ and $\overline{z}$ on $U\cap X^{-1}$ identifies $U$ with a disk. This gives a complex structure on $U$ which can be checked to be independent of the choice of local coordinate $z$. In this way, we obtain a complex structure on $X^d$. 

\begin{definition}\label{conformal double}
The conformal double of $X$ is the surface $X^d$ constructed above.
\end{definition} 

Then $X^d$ is a compact genus $2g+b-1$ surface which has an anti-holomorphic involution fixing $\partial X=\partial X^{-1}\subset X^d$. We have the following statement about conformal doubles of triangulated surfaces.

\begin{prop}\label{double of triangulated surface with boundary} Suppose $S$ is a triangulated surface with boundary. Then $S^d$ is canonically a triangulated surface without boundary.
\end{prop}

\begin{proof} Denote by $S^{-1}$ the triangulated surface wherein each triangle of $S$ is equipped with the opposite orientation. Since $\partial S$ and $\partial S^{-1}$ are naturally identified, gluing $S$ and $S^{-1}$ under this identification produces a triangulated surface $S^d$ which can be checked to be the conformal double of $S$.
\end{proof}

\subsection{Triangulated surfaces with boundedness properties} 

In this section we define several subsets of $\Tri{g}{T}$ that shall be useful in the proof of \cref{upper bound}.

Let $E$ be an equilateral triangle in $\mathbb{C}$ of side length $\ell$. 

\begin{definition} A $k$-subdivision of $E$ is the unique triangulation by $k^2$ equilateral triangles of side length $\ell/k$.
\end{definition}

Let $T$ be a triangulation of a surface (possibly with boundary).

\begin{definition} The $k$-subdivision of $T$ is the new triangulation constructed by replacing each triangle in $T$ with its $k$-subdivision.
\end{definition}

\begin{figure}[ht]

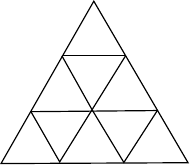

\caption{A 3-subdivision of an equilateral triangle.}

\end{figure}

\begin{definition}\label{lb tri surface} A locally bounded triangulated surface is a triangulated surface $S$ which satisfies the following two properties. 

\begin{enumerate}

\item The maximum degree of any vertex of $S$ is $7$ and

\item There exists a triangulation $S_{\lb}$ of the surface $S$ into equilateral triangles of side length $5$ such that the triangulation $S$ is a $5$-subdivision $S_{\lb}$.

\end{enumerate}

The set of locally bounded triangulated surfaces with $T$ triangles and genus $g$ (resp. $\leq T$ triangles and genus $\leq g$) is denoted $\Tric{g}{T}$ (resp. $\Tric{\leq g}{\leq T}$).

\end{definition}

Similarly, we also define locally bounded combinatorial translation surfaces. First, we have a preliminary lemma that motivates the definition. 

\begin{lemma}\label{periods of tri tran surface} Let $S$ be a combinatorial translation surface. Let $V_{>6}(S)$ be the set of vertices of $S$ with degree greater than $6$. If $\gamma$ is an arc representing an element of $H_1(S, V_{>6}(S), \mathbb{Z})$, then $$\int_\gamma\phi_S\in \mathbb{Z}+e^{\pi i/3}\mathbb{Z}.$$
\end{lemma}

\begin{proof} Note that any arc $\gamma$ on $S$ with $\partial \gamma\subset V_{>6}(S)$ is homotopic to a piecewise smooth arc wherein each piece is an edge of $S$. Since $dz$ integrates to an element of $\mathbb{Z}+e^{\pi i/3}\mathbb{Z}$ along sides of the equilateral triangle with vertices $0,1,\frac{1}{2}+\frac{\sqrt{3}}{2}i$ in $\mathbb{C}$, $$\int_\gamma\phi_S\in \mathbb{Z}+e^{\pi i/3}\mathbb{Z},$$ as desired.
\end{proof}

\begin{definition} \label{lb comb tran surface} A locally bounded combinatorial translation surface is a combinatorial translation surface $S$ that satisfies the following two properties.

\begin{enumerate} 

\item The maximum degree of any vertex of $S$ is $42$ and

\item Given any two vertices $x,y\in V_{> 6}(S)$ (not necessarily distinct) and $\gamma$ any arc from $x$ to $y$, $$\int_\gamma \phi_S\in 5\mathbb{Z}+5e^{\pi i/3}\mathbb{Z}.$$

\end{enumerate}

The set of locally bounded combinatorial translation surfaces with $T$ triangles and genus $g$ (resp. $\leq T$ triangles and genus $\leq g$) is denoted $\Tranc{g}{T}$ (resp. $\Tranc{\leq g}{\leq T}$). 

\end{definition}

\begin{remark} Given \cref{periods of tri tran surface}, one would expect condition 2 in \cref{lb comb tran surface} to be the translation surface version of condition 2 in \cref{lb tri surface}. More rigorously, we shall see in \cref{8.2} that triangulated surfaces have a canonical branched $6$-cover which is a combinatorial translation surface. Locally bounded combinatorial translation surfaces are defined so that the canonical branched $6$-covers of locally bounded triangulated surfaces are locally bounded combinatorial translation surfaces. See \cref{covers}.
\end{remark}

It is also useful to consider triangulated surfaces for which the number of vertices of degree other than $6$ is bounded. (Note that combinatorial translation surfaces of genus $g$ automatically satisfy the property that the number of vertices of degree other than $6$ is bounded by $Cg$.)

\begin{definition} The set $\Tri{\leq g}{\leq T,\leq m}$ consists of $S\in \Tri{\leq g}{\leq T}$ which satisfy the property that $|V_{\neq 6}(S)|\leq m$. 
\end{definition} 

\begin{definition} The set $\Tric{\leq g}{\leq T,\leq m}$ consists of $S\in \Tric{\leq g}{\leq T}$ which satisfy the property that $|V_{\neq 6}(S)|\leq m$.
\end{definition} 

Euler characteristic conditions imply that $|V(S)|\leq T/2$ when $g\geq 1$, so in the future, we will implicitly assume the condition $m/T\leq 1/2$ when $g\geq 1$. 

\subsection{Counting functions for the image of $\Phi$ and roadmap to prove \cref{upper bound}}

\begin{definition}\label{counting functions} Let $$\Ntri(T,g,r)=\sup \{\#(\{S\in \textstyle\Tri{\leq g}{\leq T}\displaystyle|\Phi(S)\in B_{d_T}(X,r)\})|X\in \mathcal{M}_{\leq g}\},$$ $$\Ntric(T,g,r)= \sup \{\#(\{S\in \textstyle\Tric{\leq g}{\leq T}\displaystyle|\Phi(S)\in B_{d_T}(X,r)\})|X\in \mathcal{M}_{\leq g}\},$$ $$ \Ntri(T,g,m,r)= \sup \{\#(\{S\in \textstyle\Tri{\leq g}{\leq T,\leq m}\displaystyle|\Phi(S)\in B_{d_T}(X,r)\})|X\in \mathcal{M}_{\leq g}\}  ,$$ $$ \Ntric(T,g,m,r)= \sup \{\#(\{S\in \textstyle\Tric{\leq g}{\leq T,\leq m}\displaystyle|\Phi(S)\in B_{d_T}(X,r)\})|X\in \mathcal{M}_{\leq g}\}$$ and $$\Ntranc(T,g,r)=\sup \{\#(\{S\in \textstyle\Tranc{\leq g}{\leq T}\displaystyle|\Phi(S)\in B_{d_T}(X,r)\})|X\in \mathcal{M}_{\leq g}\}.$$
\end{definition}

\begin{remark}\label{counting function notation} Here, $d_T$ denotes the Teichm\"{u}ller metric on moduli space $\mathcal{M}_g$. The Teichm\"{u}ller metric is only defined for $g\geq 2$. For the purpose of notation, in our counting functions when we count genus $g$ triangulated surfaces for $g=0,1$ we simply count all the surfaces and omit the radius variable.
\end{remark}

To show \cref{upper bound}, it is necessary to find an upper bound for $\Ncomb(T,g,r)$. We do by proving several bounds related to the other counting functions introduced in \cref{counting functions}. In Section 5, we bound $\Ntri$ in terms of $\Ntric$. In Section 6, we bound $\Ntranc$ in terms of $\Ntri$. In Section 7, we bound $\Ntric$ in terms of $\Ntranc$. Meanwhile, in Section 3 we prove bounds about the geometry of Teichm\"{u}ller balls in $\mathcal{M}_g$ which shall be useful for the bounds in Section 6 and Section 7 as well. Finally, in Section 7, we combine all these bounds and use a recursive argument to prove \cref{upper bound}.

\section{Geometry of Teichm\"{u}ller space}

In this section we give a brief review of definitions in Teichm\"{u}ller theory and describe the large genus geometry of Teichm\"{u}ller space. Detailed exposition can be found in \cite{Ahl66}, \cite{Gar87} and \cite{Hub06}. We also prove some results about quasiconformal maps.

\subsection{Quasiconformal maps} Let $U$ and $V$ be Riemann surfaces and $K\geq 1$. 

\begin{definition} An orientation preserving diffeomorphism $f:U\to V$ is $K$-quasiconformal if it satisfies $$\left|\frac{\partial f}{\partial \overline{z}}\right|\leq k\displaystyle\left|\frac{\partial f}{\partial z}\right|$$ where $k=(K-1)/(K+1)$ and $z$ is a holomorphic coordinate on $U$.
\end{definition}

Given conformal metrics on $U$ and $V$, this means locally there exist oriented orthonormal bases on $U$ and $V$ with respect to which $df$ has singular values $\lambda_1$ and $\lambda_2$ satisfying $K^{-1}\leq \lambda_1/\lambda_2\leq K$. The smallest such quantity $K$ is called the quasiconformal dilatation. 

\begin{remark} Quasiconformal maps are generally required to be homeomorphisms only. See \cite[Sections 4.1 and 4.5]{Hub06} for definitions in this setting. In future sections, we explicitly use only the definition we have given above, but we implicitly use the definition in the homeomorphism setting in \cref{extending quasiconformal maps} and \cref{bi-Lipschitz metric}.
\end{remark}

\subsection{Hodge norm on $H^1(X,\mathbb{C})$} Let $X$ be a compact Riemann surface of genus $g\geq 2$ and $\rho_X$ the hyperbolic metric on $X$. In this section we record certain constructions from Hodge theory on Riemann surfaces. See \cite{Dem09} for a detailed exposition. Let $\mathcal{E}^1(X,\mathbb{C})$ denote the space of complex valued differential $1$-forms on $X$; also denote closed and exact $1$-forms by $Z^1(X,\mathbb{C})$ and $B^1(X,\mathbb{C})$, respectively. The Hodge star $$*:\mathcal{E}^1(X,\mathbb{C})\to \mathcal{E}^1(X,\mathbb{C})$$ is an $\mathbb{C}$-linear map which fixes the space of harmonic $1$-forms $\mathcal{H}^1(X,\mathbb{C})\subset \mathcal{E}^1(X,\mathbb{C})$. 

\begin{definition} The Hodge inner product (a Hermitian inner product) on $\mathcal{E}^1(X,\mathbb{C})$ is defined as $$\langle \omega_1,\omega_2\rangle_X =\int_X \omega_1\wedge *\overline{\omega_2}=\int_X \langle \omega_1,\omega_2\rangle_{\rho_X}d\rho_X^2.$$ The Hodge norm on $\mathcal{E}^1(X,\mathbb{C})$ is defined as $$\|\omega\|^2_X=\int_X \omega \wedge *\overline{\omega}$$ for $\omega, \omega_1,\omega_2\in \mathcal{E}^1(X,\mathbb{C})$.
\end{definition} 

By the Hodge decomposition theorem, the space of closed complex valued differential $1$-forms on $X$ splits as $$\mathcal{Z}^1(X,\mathbb{C})=\mathcal{H}^1(X,\mathbb{C})\oplus \mathcal{B}^1(X,\mathbb{C}),$$ and the splitting is orthogonal with respect to the Hodge inner product. 

\begin{definition}\label{Hodge norm} The Hodge norm of a cohomology class $$u\in \mathcal{Z}^1(X,\mathbb{C})/\mathcal{B}^1(X,\mathbb{C})\simeq H^1(X,\mathbb{C}),$$ is $$\|u\|_X=\inf_{[\omega]= u}\|\omega\|_X.$$
\end{definition}

The infimum in \cref{Hodge norm} is attained by the unique harmonic representative of $u$.

\subsection{Quasiconformal maps and Hodge norm} In this section we study how the Hodge norm behaves under a quasiconformal map. All quasiconformal maps in this section are assumed to be diffeomorphisms. Let $X$ and $Y$ be compact Riemann surfaces of genus $g\geq 2$, with hyperbolic metrics $\rho_X$ and $\rho_Y$, respectively.  

\begin{lemma}\label{qc map Hodge norm diff form} If $f:X\to Y$ is a $K$-quasiconformal map and $\omega$ is a complex valued differential $1$-form on $Y$, then $$ (1/K)^{1/2}\|\omega\|_Y\leq \|f^*\omega\|_X\leq K^{1/2}\|\omega\|_Y.$$
\end{lemma}

\begin{proof} We have, \begin{align*}
\|f^*\omega\|^2_X&=\int_X\langle f^*\omega,f^*\omega\rangle_{\rho_X}\rho^2_X\\&\leq \int_X \|df\|^2f^*(\langle \omega,\omega\rangle_{\rho_Y})|\det(df)|^{-1}f^*\rho^2_Y
\\&\leq K\int_Y \langle \omega,\omega\rangle_{\rho_Y}\rho^2_Y\\&=K\|\omega\|^2_Y.
\end{align*}

Applying an analogous argument to $f^{-1}$, which is also $K$-quasiconformal, we obtain $\|\omega\|_Y\leq K^{1/2}\|f^*\omega\|_X$.
\end{proof}

As a corollary, we have:

\begin{cor}\label{qc map Hodge norm cohomology} Let $f:X\to Y$ be a $K$-quasiconformal map. Let $u\in H^1(Y,\mathbb{C})$. Then $$(1/K)^{1/2}\|u\|_Y\leq \|f^*u\|_X\leq K^{1/2}\|u\|_Y.$$
\end{cor}

\begin{proof} Let $\omega$ be the harmonic form on $Y$ representing $u$. Then by \cref{qc map Hodge norm diff form}, 
\begin{align*}\|f^*u\|_X &\leq \|f^*\omega\|_X\\&\leq K^{1/2}\|\omega\|_Y\\&= K^{1/2}\|u\|_Y.
\end{align*} 
The analogous argument applied to $f^{-1}$, gives $\|u\|_Y\leq K^{1/2}\|f^*u\|_X$. 
\end{proof}

Finally, we show that the pullback of a harmonic form under a quasiconformal map is close to its harmonic representative. 

\begin{lemma}\label{pullback of harmonic form close to harmonic representative} Let $f:X\to Y$ be a $K$-quasiconformal map. Let $\omega$ be a harmonic $1$-form on $Y$. Denote by $(f^*\omega)^h$ the unique harmonic $1$-form on $X$ cohomologous to $f^*\omega$. Then $$\|f^*\omega-(f^*\omega)^h\|_X\leq ((K^2-1)/K)^{1/2}\|\omega\|_Y.$$
\end{lemma}

\begin{proof} We have,
\begin{align*} \|f^*\omega-(f^*\omega)^h\|^2_X&=\langle f^*\omega-(f^*\omega)^h, f^*\omega-(f^*\omega)^h \rangle_X\\&=\langle f^*\omega,f^*\omega\rangle_X-\langle (f^*\omega)^h, (f^*\omega)^h\rangle_X-2\langle f^*\omega-(f^*\omega)^h,(f^*\omega)^h\rangle_X.
\end{align*}
By \cref{qc map Hodge norm diff form} and \cref{qc map Hodge norm cohomology}, $$\langle f^*\omega,f^*\omega\rangle_X-\langle (f^*\omega)^h, (f^*\omega)^h\rangle_X\leq ((K^2-1)/K)\|\omega\|^2_Y.$$ Exact forms are orthogonal to harmonic forms, so since $f^*\omega-(f^*\omega)^h$ is exact and $(f^*\omega)^h$ is harmonic, $$2\langle f^*\omega-(f^*\omega)^h,(f^*\omega)^h\rangle_X=0.$$ The lemma follows.
\end{proof}

\subsection{Quasisymmetric maps and quasicircles}\label{extending quasiconformal maps} In this section, we introduce some definitions related to quasiconformal maps.
 
Let $(U,d_U)$ and $(V,d_V)$ be metric spaces bi-Lipschitz to domains in $\mathbb{C}$. 

\begin{definition} An embedding $f:U\to V$ is called $K$-weakly-quasisymmetric if for all $x,y,z\in U$, $$d_V(f(x), f(y))\leq K d_V(f(x),f(z))$$ if $$d_U(x,y)\leq d_U(x,z).$$\end{definition} 

When $U$ and $V$ are oriented topological manifolds, we will assume $f$ is orientation preserving. When $U$ and $V$ are simply domains in $\mathbb{C}$, the metrics on $U$ and $V$ are taken to be the restrictions of the Euclidean metric to $U$ and $V$, respectively. Then weakly-quasisymmetric maps of the unit circle may be extended to quasiconformal maps of the unit disk:

\begin{prop}[\cite{Ahl66}, Theorem IV.B.2] \label{extending bi-Lipschitz map 1} Suppose $f':S^1\to S^1$ is $K$-weakly-quasisymmetric. Then $f'$ extends to a $C(K)$-quasiconformal map $f:\overline{\mathbb{D}}\to \overline{\mathbb{D}}$ where $C(K)$ only depends on $K$.\end{prop}

\begin{definition} A simple closed curve (resp. simple arc) in $\mathbb{C}$ is a $K$-quasicircle (resp. $K$-quasiarc) is the image of the unit circle (resp. unit line segment) under a $K$-quasiconformal homeomorphism of the plane. \end{definition}

Let $(U,d_U)$ be a metric space bi-Lipschitz and homeomorphic to a domain in $\mathbb{C}$. When $U$ is simply a domain in $\mathbb{C}$, $d_U$ is taken to be the restriction of the Euclidean metric to $U$. 

\begin{definition} A simple arc $\gamma$ in $U$ is a $K$-bounded-turning curve if it satisfies the following condition: for all $x,y,z\in \gamma$ lying in order, $$d_U(x,y)+d_U(y,z)\leq K d_U(x,z).$$ A simple closed curve $\gamma$ in $U$ is a $K$-bounded-turning curve if every point of $\gamma$ lies in the interior of a subarc of $\gamma$ that is a $K$-bounded-turning curve.\end{definition} 

In the rest of this section, we note some results about quasicircles, quasiarcs, bounded-turning curves and quasisymmetric maps that shall be useful in later sections.

The following statement relates quasicircles and quasiarcs to the bounded-turning property. A proof can be found in \cite[Sections 2.8.7, 2.8.8 and 2.8.9]{LV73}.

\begin{prop}\label{bounded turning curve is quasicircle}Suppose a simple closed curve (resp. simple arc) $\gamma$ in $\mathbb{C}$ is a $K$-bounded-turning curve. Then $\gamma$ is a $C(K)$-quasicircle (resp. $C(K)$-quasiarc) where $C(K)$ is a constant only depending on $K$. Similarly, suppose a simple closed curve (resp. simple arc) $\gamma$ is a $K$-quasicircle (resp. $K$-quasiarc). Then it is a $C(K)$-bounded-turning curve.\end{prop}

We have the following extension result about quasisymmetric maps.

\begin{lemma}\label{extending bi-Lipschitz map 2} Let $U_1\subset \mathbb{C}$ and $U_2\subset \mathbb{C}$ be simply connected open domains such that $\partial U_1$ is a $K_1$-quasicircle and $\partial U_2$ is a $K_2$-quasicircle. Suppose $$f':\partial U_1\to \partial U_2$$ is a $K_3$-weakly-quasisymmetric homeomorphism. Then $f'$ extends to a map $$f:\overline{U_1}\to \overline{U_2}$$ that is $C(K_1,K_2,K_3)$-quasiconformal. Here, $C(K_1,K_2,K_3)$ is a constant that depends only on $K_1$, $K_2$ and $K_3$.\end{lemma}
\begin{proof} Let $$f_1:\mathbb{C}\to \mathbb{C}$$ be a $K_1$-quasiconformal mapping taking $\overline{U_1}$ to $\overline{\mathbb{D}}$. Then $f_1$ is $C(K_1)$-weakly-quasisymmetric for some $C(K_1)$ only depending on $K_1$ (see \cite[Theorem 18.1]{Vai71}). Similarly, let $$f_2:\mathbb{C}\to \mathbb{C}$$ be a $K_2$-quasiconformal mapping taking $\overline{U_2}$ to $\overline{\mathbb{D}}$. Then $f_2$ is $C(K_2)$-weakly-quasisymmetric. Then identifying $\partial \mathbb{D}\simeq S^1$, the map $$f_2\circ f' \circ f_1^{-1}:S^1\to S^1$$ is $C(K_1,K_2,K_3)$-weakly-quasisymmetric (note that compositions and inverses of weakly-quasisymmetric maps of the plane are weakly-quasisymmetric by Theorem 10.6 and Theorem 10.19 in \cite{Hei01}). By \cref{extending bi-Lipschitz map 1}, $$f_2\circ f' \circ f_1^{-1}$$ extends to a $C(K_1,K_2,K_3)$-quasiconformal map $$f'':\overline{\mathbb{D}}\to\overline{\mathbb{D}}.$$ Taking $$f=f_2^{-1}\circ f''\circ f_1$$ gives a $C(K_1,K_2,K_3)$-quasiconformal map from $\overline{U_1}$ to $\overline{U_2}$ that is an extension of $f'$.\end{proof}

\subsection{Teichm\"{u}ller space and Teichm\"{u}ller metric}

Let $g\geq 2$ and $S_g$ a smooth oriented genus $g$ surface. A marked Riemann surface is a Riemann surface $X$ along with a diffeomorphism $S_g\to X$.

Let $X$ and $Y$ be marked Riemann surfaces of genus $g$, with markings $$f_X:S_g\to X$$ and $$f_Y:S_g\to Y,$$ respectively. The marked surfaces $X$ and $Y$ are considered equivalent if there is a biholomorphism  $$f:X\to Y$$ satisfying the property that $f_Y^{-1}\circ f\circ f_X$ is isotopic to the identity.

\begin{definition}\label{Teichmuller space} Teichm\"{u}ller space $\mathcal{T}_g$ is the set of equivalence classes of marked Riemann surfaces. The Teichm\"{u}ller distance, denoted $d_T$, is given by $$d_{T}(X,Y)=\inf\left\lbrace\frac{1}{2}\log K|f:X\to Y \text{ is } K \text{-quasiconformal} \right\rbrace$$ where $f$ satisfies the property that $f_Y^{-1}\circ f\circ f_X$ is isotopic to the identity. 
\end{definition}

Teichm\"{u}ller's theorem asserts that the infimum in \cref{Teichmuller space} is attained by a homeomorphism and is unique. The unique map that attains the infimum is called a Teichm\"{u}ller map.

The space $\mathcal{T}_g$ admits a natural complex structure under which it is a complex manifold of dimension $3g-3$. The cotangent space at a point $X$ may be naturally identified with $Q(X)$, the space of holomorphic quadratic differentials on $X$. Under this identification, the Teichm\"{u}ller metric is the $L^1$-norm on $Q(X)$, given by $$\|\phi\|=\int_{X}|\phi|$$ for a holomorphic quadratic differential $\phi$ on $X$.  

Moduli space $\mathcal{M}_g$ can be obtained from $\mathcal{T}_g$ by quotienting by the action of the mapping class group $\Mod_g$. The Teichm\"{u}ller metric also descends to $\mathcal{M}_g$.

\subsection{Extremal length and Teichm\"{u}ller metric} The Teichm\"{u}ller metric has a description in terms of extremal length, as we shall explain now. Let $X$ be a Riemann surface of genus $g$. Let $\gamma$ a free homotopy class of a simple closed curve on $X$. 

\begin{definition} Given a Riemannian metric $\rho$ on $X$, the quantity $$\textstyle\length_\rho(\gamma)$$ is defined to be the infimum of lengths in the $\rho$-metric over all curves representing $\gamma$.
\end{definition}

\begin{definition}\label{extremal length} The extremal length of $\gamma$ on $X$ is defined to be $$\textstyle\Ext_X(\gamma)=\displaystyle\sup_\rho\frac{\length_\rho(\gamma)^2}{\Area_\rho(X)}$$ where the supremum is taken over all conformal metrics $\rho$ on $X$. 
\end{definition}

The next theorem describes when the supremum in \cref{extremal length} is achieved. 

\begin{theorem} [Jenkins \cite{Jen57}, Strebel \cite{Str66}, see also Theorem 3.1 in \cite{Ker80}] \label{Jenkins-Strebel} The supremum of $$\sup_\rho\frac{\length_\rho(\gamma)^2}{\Area_\rho(X)}$$ is achieved when $\rho$ is the flat metric $|\phi|$ associated to a holomorphic quadratic differential $$\phi\in Q(X).$$
\end{theorem}

The Teichm\"{u}ller distance has the following description due to Kerckhoff \cite[Theorem 4]{Ker80}. 

\begin{theorem}\label{Teichmuller metric via extremal length} For $X,Y\in \mathcal{T}_g$, $$d_T(X,Y)=\frac{1}{2}\log \sup_{\gamma}\frac{\Ext_Y(\gamma)}{\Ext_X(\gamma)},$$ where the supremum is taken over all free homotopy classes of a simple closed curve on $X$. 
\end{theorem}

\begin{remark} Since $X$ and $Y$ are marked surfaces, $\gamma$ (which was initially defined as a free homotopy class on $X$) is also automatically a free homotopy class on $Y$.
\end{remark}

\subsection{Bi-Lipschitz metric}\label{bi-Lipschitz metric} Let $g\geq 2$. 
\begin{definition} The bi-Lipschitz metric $d_L$ on $\mathcal{M}_g$ is $$d_L(X,Y)=\inf\left\{\log K|f:X\to Y \text{ is a K-bi-Lipschitz diffeomorphism}\right\}.$$ Here, the bi-Lipschitz constant is measured with respect to the unique hyperbolic metrics on $X$ and $Y$.
\end{definition} 

The following result establishes a comparison between $d_T$ and $d_L$.

\begin{prop} \label{qc to lipschitz} There exists a universal constant $C$ such that for all $g$, $$d_T\leq d_L\leq Cd_T$$ on $\mathcal{M}_g$. 
\end{prop}

\begin{proof} The inequality $d_T\leq d_L$ follows from the fact that any $K$-bi-Lipschitz map is automatically $K^2$-quasiconformal. 

In the other direction, we must show that $d_L\leq Cd_T$. The following proof is due to Maxime Fortier Bourque. Fix $K_0>0$. Let $X,Y\in \mathcal{M}_g$, such that $$d_T(X,Y)=\log K.$$ Cut the Teichm\"{u}ller geodesic from $X$ to $Y$ into around $$n=\left\lceil \frac{\log K}{\log K_0}\right\rceil$$ segments of length around $$\frac{\log K}{n}.$$ Let $Z_0=X,Z_1,...,Z_n=Y$ be the endpoints of these segments. Since $$d_T(Z_i,Z_{i+1})\leq \log K_0,$$ there is a quasiconformal map $$f_i:Z_i\to Z_{i+1}$$ with dilatation at most $K_0^2$. Using the Douady-Earle extension \cite[Theorem 5.2]{DE86}, $f_i$ may be replaced with a map $g_i$ that is $C(K_0^2)$-bi-Lipschitz between the hyperbolic metrics on $Z_i$ and $Z_{i+1}$. Take $$g=g_{n-1}\circ...\circ g_0:X\to Y.$$ Then $g$ is $C(K_0^2)^n$-bi-Lipschitz between the hyperbolic metrics on $X$ and $Y$. Therefore, 
\begin{align*}d_L(X,Y)&\leq C\frac{\log C(K_0^2)}{\log K_0}\log K\\&\leq C\log K
\end{align*} for a universal constant $C$.
\end{proof}

\subsection{Kobayashi metric} Let $M$ be a complex manifold of arbitrary dimension. Roughly speaking, the Kobayashi pseudometric is the largest pseudometric on $M$ such that all holomorphic maps into $M$ are distance decreasing. See \cite[Section 4.1]{Kob05} for a rigorous construction. In general, the Kobayashi pseudometric may not be a metric (i.e. may not separate points). However, we shall see that in all cases of $M$ relevant to us, the Kobayashi metric exists. In particular, the Kobayashi metric exists when $M$ is a bounded domain in $\mathbb{C}^n$ \cite[Corollary 4.4.6]{Kob05}. The Kobayashi metric satisfies the following important property.

\begin{prop}[\cite{Kob05}, Proposition 4.1.1] \label{Kobayashi comparison} Let $M$ and $N$ be complex manifolds. Suppose they admit Kobayashi metrics $d_M$ and $d_N$, respectively. Let $f:M\to N$ be a holomorphic map. Then for all $x,y\in M$, $$d_M(x,y)\geq d_N(f(x),f(y)).$$ If $f$ is a biholomorphism, then equality holds. 
\end{prop}

We also have the following theorem due to Royden \cite[Theorem 3]{Roy71}.

\begin{theorem} \label{Royden's theorem} On $\mathcal{T}_g$, the Kobayashi metric exists and is the Teichm\"{u}ller metric $d_T$.
\end{theorem}

\subsection{Bers embedding of $\mathcal{T}_g$}

Let $g\geq 2$ and fix $X\in \mathcal{T}_g$. Recall from \cref{conformal mirror} that $X^{-1}$ denotes the conformal mirror of $X$. Let $Q^{\infty}(X^{-1})$ denote the space of quadratic differentials on $X^{-1}$, equipped with the norm $$\|\phi\|_\infty=\sup_{x\in X^{-1}}\frac{|\phi(x)|}{|\rho_{X^{-1}}|^2}.$$ (Here, $\rho_{X^{-1}}$ is the hyperbolic metric on $X^{-1}$.) The Bers embedding \cite[Section 5.4]{Gar87} is a holomorphic embedding of $\mathcal{T}_g$ into $Q^{\infty}(X^{-1})$ sending $X\in \mathcal{T}_g$ to the origin in $Q^{\infty}(X^{-1})$. 

\begin{theorem}[\cite{Gar87}, Theorem 5.4.1] \label{Bers embedding} The Bers embedding $$\beta_X:\mathcal{T}_g\to Q^{\infty}(X^{-1})$$ satisfies $$B_{\infty}(0,1/2)\subset \beta_X(\mathcal{T}_g)\subset B_{\infty}(0,3/2)$$ where $B_{\infty}(0,r)$ denotes the norm ball of radius $r$ in the space $Q^{\infty}(X^{-1})$.
\end{theorem}

\begin{remark} The constants in \cite[Theorem 4.5.1]{Gar87} are different because the normalization of the Teichm\"{u}ller metric is different. See \cite[Theorem 2.2]{McM00} for normalization and constants that agree with ours.
\end{remark}

\subsection{Asymptotic geometry of Teichm\"{u}ller space} 

In this section, we use the Bers embedding to obtain bounds on the geometry of $d_T$ on $\mathcal{T}_g$. This technique has been used to prove \cite[Theorem 8.2]{McM00} and \cite[Theorem 1.5]{FKM13}. Our bounds are variations of the bounds in the latter.

Fix $X\in \mathcal{T}_g$. Denote by $Q=B_{\infty}(0,1)$ the open unit norm ball in $Q^{\infty}(X^{-1})$. Since $Q^{\infty}(X^{-1})$ is $3g-3$-dimensional, $Q$ is an open $3g-3$ complex manifold. Let $d_Q$ denote the Kobayashi metric on $Q$. Putting together \cref{Kobayashi comparison}, \cref{Royden's theorem} and \cref{Bers embedding}, we have the following two metric comparison results. 

\begin{lemma}\label{3/2} Let $Y,Z\in \mathcal{T}_g$. Then $$\|\beta_X(Y)\|_\infty\leq 3/2$$ and $$\|\beta_X(Z)\|_\infty\leq 3/2.$$ Moreover, $$d_Q((2/3)\beta_X(Y),(2/3)\beta_X(Z))\leq d_T(Y,X).$$
\end{lemma} 

\begin{lemma}\label{1/2} There exists a holomorphic map $$\beta_X^{-1}:B_{\infty}(0,1/2)\to \mathcal{T}_g$$ such that $$\beta_X\circ \beta_X^{-1}=\id$$ on $B_{\infty}(0,1/2)$. Moreover, for $\phi_1,\phi_2\in B_{\infty}(0,1/2)$, $$d_T(\beta_X^{-1}(\phi_1),\beta_X^{-1}(\phi_2))\leq d_Q(2\phi_1,2\phi_2).$$
\end{lemma}

The following lemma gives us an approximation for $d_Q$.

\begin{lemma}\label{Kobayashi metric on norm ball} For $\phi_1,\phi_2\in B_{\infty}(0,1/8)$, $$C_1\|\phi_1-\phi_2\|_\infty \leq d_Q(\phi_1,\phi_2)\leq C_2\|\phi_1-\phi_2\|_\infty$$ for universal constants $C_1$ and $C_2$. 
\end{lemma}  

\begin{proof} First, we show that $$d_Q(\phi_1,\phi_2)\geq C_1\|\phi_1-\phi_2\|_\infty.$$ To see this, let $x\in X^{-1}$ (the conformal mirror of $X$), and let $v$ be a unit tangent vector at $x$ with respect to the hyperbolic metric $\rho_{X^{-1}}$ on $X^{-1}$. We define a map $$I_x: B_{\infty}(0,1)\to \mathbb{D}$$ that sends $\phi$ to the evaluation of $\phi$ at $v\otimes v$. Note that for all $\phi\in B_{\infty}(0,1)$, 
\begin{align*}
|I_x(\phi)|&= \displaystyle\frac{|\phi|}{|\rho_{X^{-1}}^2|}(x)\\&\leq 1.
\end{align*} Also, for all $\phi\in B_{\infty}(0,1/2)$, $$|I_x(\phi)|\leq 1/2.$$ Furthermore, $I_x$ is affine, thus holomorphic. Therefore, by \cref{Kobayashi comparison}, $I_x$ is distance decreasing with respect to $d_Q$ on $B_{\infty}(0,1)$ and the Kobayashi metric on $\mathbb{D}$ which is the Poincare metric. So for $\phi_1,\phi_2\in B_{\infty}(0,1/2)$,  
\begin{align*}d_Q(\phi_1,\phi_2)&\geq \rho_{\mathbb{D}}(I_x(\phi_1),I_x(\phi_2))\\&\geq C_1|I_x(\phi_1)-I_x(\phi_2)|\\&= C_1\frac{|\phi_1-\phi_2|}{|\rho_X|^2}(x).
\end{align*} Since this is true for all $x\in X^{-1}$, $$d_Q(\phi_1,\phi_2)\geq C_1\|\phi-\phi_2\|_\infty$$ as desired. 

Next, we show that for $\phi_1,\phi_2\in B_{\infty}(0,1/8)$, $$d_Q(\phi_1,\phi_2)\leq C_2\|\phi_1-\phi_2\|_\infty.$$ To do this, consider the map $I:\mathbb{D}\to B_{\infty}(0,1)$ sending $\zeta\in \mathbb{D}$ to $$\phi_1+\zeta\frac{\phi_2-\phi_1}{2\|\phi_2-\phi_1\|_\infty}.$$ The image of $I$ is contained in $B_{\infty}(0,1)$ because 
\begin{align*}\left\|\phi_1+\zeta\frac{\phi_2-\phi_1}{2\|\phi_2-\phi_1\|_\infty}\right\|_\infty &\leq \|\phi_1\|_\infty+\left|\frac{\zeta}{2}\right|\\&\leq 1.
\end{align*} Note that $I$ is affine and therefore holomorphic. Moreover, $I(0)=\phi_1$ and $$I(2\|\phi_1-\phi_2\|_\infty)=\phi_2.$$ By \cref{Kobayashi comparison}, $I$ must be distance decreasing with respect to the Poincare metric on $\mathbb{D}$. Since $\phi_1,\phi_2\in B_{\infty}(0,1/8)$, $$2\|\phi_1-\phi_2\|_\infty\leq 1/2.$$ Therefore, 
\begin{align*}d_Q(\phi_1,\phi_2)&\leq d_{\rho_{\mathbb{D}}}(0,2|\phi_1-\phi_2\|_\infty)\\&\leq C_2\|\phi_1-\phi_2\|_\infty,
\end{align*} as desired.
\end{proof}

\cref{3/2}, \cref{1/2} and \cref{Kobayashi metric on norm ball} together imply:

\begin{lemma} \label{Teichmuller bound 1} Let $$\phi_1,\phi_2\in B_{\infty}(0,1/16).$$ Then $$C_3\|\phi_1-\phi_2\|_\infty \leq d_T(\beta_X^{-1}(\phi_1),\beta_X^{-1}(\phi_2))\leq C_4\|\phi_1-\phi_2\|_\infty.$$
\end{lemma}

\begin{proof} By \cref{1/2}, $$d_T(\beta_X^{-1}(\phi_1),\beta_X^{-1}(\phi_2))\leq d_Q(2\phi_1,2\phi_2).$$ Since $$2\phi_1,2\phi_2\in B_{\infty}(0,1/8),$$ by \cref{Kobayashi metric on norm ball}, $$d_Q(2\phi_1,2\phi_2)\leq C_4\|2\phi_1-2\phi_2\|_\infty\leq C_4\|\phi_1-\phi_2\|_\infty.$$ This gives one direction of the comparison in the lemma statement. To obtain the other direction, by \cref{3/2}, $$d_T(\beta_X^{-1}(\phi_1),\beta_X^{-1}(\phi_2))\geq d_Q((2/3)\phi_1,(2/3)\phi_2).$$ Since $$(2/3)\phi_1,(2/3)\phi_2\in B_{\infty}(0,1/8),$$ \cref{Kobayashi metric on norm ball} implies $$d_Q((2/3)\phi_1,(2/3)\phi_2)\geq C_3\|(2/3)\phi_1-(2/3)\phi_2\|_\infty\geq  C_3\|\phi_1-\phi_2\|_\infty,$$ as desired.
\end{proof}

Finally, we have: 

\begin{lemma} \label{Teichmuller bound 2} There exist universal constants $C_3$ and $C_4$ such that for $$Y,Z\in B_{d_T}(X,C_3/20),$$
$$C_3\|\beta_X(Y)-\beta_X(Z)\|_\infty \leq d_T(Y,Z)\leq C_4\|\beta_X(Y)-\beta_X(Z)\|_\infty.$$
\end{lemma}

\begin{proof} Consider the map $$\beta_X^{-1}:B_{\infty}(0,1/20)\to \mathcal{T}_g.$$ This is a biholomorphism onto its image. Note that $$C_3\|\phi_1-\phi_2\|_\infty \leq d_T(\beta_X^{-1}(\phi_1),\beta_X^{-1}(\phi_2))$$ for $$\phi_1,\phi_2\in B_{\infty}(0,1/16)$$ by \cref{Teichmuller bound 1}. In particular, $$C_3\|\phi \|_\infty \leq d_T(X,\beta_X^{-1}(\phi))$$ for $$\phi\in B_{\infty}(0,1/16).$$ Thus $\beta_X^{-1}(B_{\infty}(0,1/20))$ contains the point $X\in \mathcal{T}_g$, but $\partial \beta_X^{-1}(B_{\infty}(0,1/20))$ does not intersect $B_{d_T}(X,C_3/20)$. This means $\beta_X^{-1}(B_{\infty}(0,1/20))$ contains $B_{d_T}(X,C_3/20)$. Now \cref{Teichmuller bound 1} gives the desired inequalities. 
\end{proof}

\begin{lemma}\label{Teichmuller bound 2.5} Let $C_3$ be as in \cref{Teichmuller bound 2} and suppose $r_1,r_2\in \mathbb{R}_+$ such that $r_1\leq C_3/20$ and $r_1\leq r_2$. Let $X_1,...,X_N\in B_{d_T}(X,r_2)\subset \mathcal{T}_g$ such that $$d_T(X_i,X_j)\geq r_1$$ for all $i,j\in \{1,...,N\}$ distinct. If $r_2\leq C_3/20$, then  $$N\leq (C/r_1)^{6g-6}.$$ For all $r_2\geq C_3/20$, $$N\leq (C/r_1)^{Cr_2(6g-6)}.$$ Here, $C$ is a universal constant.
\end{lemma}

\begin{proof} First we treat the case where $r_2\leq C_3/20$. For all $i,j\in \{1,...,N\}$ distinct, $$r_1/C_4\leq \|\beta_X(X_i)-\beta_X(X_j)\|_\infty$$ by \cref{Teichmuller bound 2}. Therefore norm balls of radius $r_1/(2C_4)$ around the $\beta_X(X_i)$ are disjoint. Again by \cref{Teichmuller bound 2}, $$\|\beta_X(X_i)\|_\infty\leq r_2/C_3$$ for all $i\in \{1,...,N\}$. Now, $Q^{\infty}(X^{-1})\simeq \mathbb{C}^{3g-3}$ as a vector space. Hence it admits a Euclidean volume. Denote by $V(r)$ the Euclidean volume of a norm ball of radius $r$ in $Q^{\infty}(X^{-1})$. Then $$V(r_1/(2C_4))\geq ((4C_4r_2)/(C_3r_1))^{-(6g-6)}V(2r_2/C_3)$$ since a ball of radius $2r_2/C_3$ is simply a scaled copy of a ball of radius $r_1/(2C_4)$. Note that $$B_{\infty}(\beta_X(X_i),r_1/(2C_4))\subset B_{\infty}(0,2r_2/C_3)$$ because $\beta_X(X_i)\in \overline{B_{\infty}(0,r_2/C_3)}$ for all $i\in \{1,...,N\}$ and $r_1\leq r_2$, $C_3\leq C_4$. The former balls are also mutually disjoint, so $$V(2r_2/C_3)\geq N\cdot V(r_1/(2C_4)).$$ Therefore $$N\leq (C/r_1)^{6g-6}.$$ This proves the first part of the lemma. 

Now, to prove the second part, let us assume that $r_2\geq C_3/20$. We have that at most $(C/r_1)^{6g-6}$ of the $X_i$ are contained in $B_{d_T}(X,C_3/20)$. Also, at most $(C/r_1)^{6g-6}$ of the $X_j$ are contained in $B_{d_T}(X_i,C_3/20)$ for all $X_i$. By induction we obtain the desired result.
\end{proof}

As a corollary, we have:

\begin{cor}\label{Teichmuller bound 3} Let $C_3$ be as in \cref{Teichmuller bound 2} and suppose $r_1,r_2\in \mathbb{R}_+$ such that $r_1\leq C_3/20$ and $r_1\leq r_2$. If $$r_2\leq C_3/20,$$ then any $r_2$-radius ball in $\mathcal{M}_g$ in the Teichm\"{u}ller metric can be covered by $$(C/r_1)^{6g-6}$$ number of $r_1$-radius balls. If $$r_2\geq C_3/20,$$ then any $r_2$-radius ball in $\mathcal{M}_g$ can be covered by $$(C/r_1)^{Cr_2(6g-6)}$$ number of $r_1$-radius balls. Here, $C$ is a universal constant.
\end{cor}

\begin{proof} By \cref{Teichmuller bound 2.5}, these statements are true on $\mathcal{T}_g$. Thus they are true on $\mathcal{M}_g$ also.
\end{proof}

\section{Lower bounds}

In this section, we prove \cref{lower bound}. 

Let $X\in \mathcal{M}_g$. Let $\rho_X$ be the conformal hyperbolic metric on $X$. Assume $$\sum_{\substack{\gamma \text{ \normalfont simple closed geodesic on } X\\\length_{\rho_X}(\gamma)<2\arcsinh(1)}}\textstyle\length_{\rho_X}(\gamma)^{-1}\leq R.$$
 
Choose $U_1,...,U_N$, $V_1,...,V_N$ and $W_1,...,W_N$ as in \cref{covering lemma}. Here, $N\leq C(R+g)$. Recall that the $\{W_i\}$ each cover $X$ by condition 1 of \cref{covering lemma}. The circles $\partial W_i$ divide $X$ into $M$ connected components that we label $P_1,...,P_M$. By construction, each $P_j$ is contained in some $W_i$. Moreover, each $\partial P_j$ is a union of hyperbolic circular arcs. There are at most $C$ such arcs in $\partial P_j$: $P_j$ is contained in some $W_i$, which in turn is contained in $U_i$, which intersects at most $C$ of the other $U_1,...,U_N$ (by condition 5 of \cref{covering lemma}). For the same reason, each $W_i$ is a union of at most $C$ of the $P_1,...,P_M$. 

Let $$\mathbb{D}_{\tanh(\radius(W_i)/2)}=\{z\in \mathbb{D}||z|<{\tanh(\radius(W_i)/2)}\}.$$ Let $$f_i:W_i\to \mathbb{D}_{\tanh(\radius(W_i)/2)}$$ be a conformal map (unique up to rotation) that sends $\cent(W_i)$ to $0$, that is $C$-bi-Lipschitz between $W_i$ with $\rho_X$ and $\mathbb{D}_{\tanh(\radius(W_i)/2)}$ with the Euclidean metric.

Now, we construct triangulations $T_r$ of $X$ by triangulating each $P_j$ as follows. Fix $r>0$ sufficiently small.

\begin{enumerate}

\item For each $P_j$, choose an arbitrary $W_{s(j)}$ such that $P_j\subset W_{s(j)}$. 

\item Divide $f_{s(j)}(P_j)$ into $$P^B_j=B_{\Euc}(\partial (f_{s(j)}(P_j)),r\cdot \radius(W_{s(j)}))$$ and $$P^I_j=f_{s(j)}(P_j)\setminus P^B_j.$$ We triangulate each $P_j$ separately.

\item Triangulate a subset of $f_{s(j)}(P_j)$ by the standard hexagonal triangulation with side length $$r\cdot \radius(W_{s(j)})$$ such that $P^I_j$ is covered by the triangulation.

\item Pullback by $f_{s(j)}$ to obtain a partial triangulation of a subset $X_{E}\subset X$. 

\begin{figure}[ht]

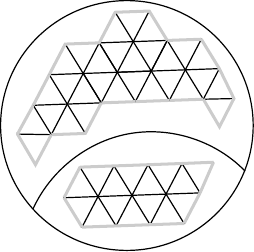

\caption{Two of the $P_j$ contained in the same $W_i$, with the partial triangulation constructed in Step 3. The gray edges are replaced in Step 5.}

\end{figure}

\item In the triangulation of $X_E$, replace all border edges between two vertices by the shortest hyperbolic geodesic connecting these two vertices.

\item Choose a set of vertices on $X\setminus X_E$ that is a $C_1r\cdot\radius(W_i)$-separated set and a $C_2r\cdot\radius(W_i)$-net on each $W_i$. Using these chosen vertices, complete the triangulation of $X\setminus X_E$ by hyperbolic triangles such that each triangle is contained in some $V_i$, and within $V_i$ is contained in a $C_3r\cdot \radius(W_i)$-radius hyperbolic ball.

\end{enumerate} 

We have the following lemma regarding Step 5 of the triangulation process described.

\begin{lemma}\label{se tri angle} There exists a universal constant $r_0$ such that for all $r<r_0$, the following statements hold. Let $x,y\in W_i$ such that $f_i^*d_{\Euc}(x,y)\leq r\cdot \radius(W_i)$. Let $\gamma_{\Euc}$ be the geodesic in $W_i$ connecting $x$ and $y$ in the metric $f_i^*d_{\Euc}$, and $\gamma_X$ be the geodesic in $W_i$ connecting $x$ and $y$ in the metric $\rho_X$. The angles between $f_i(\gamma_X)$ and $f_i(\gamma_{\Euc})$ are at most $\pi/7$. Moreover, the curvatures of $f_i(\gamma_X)$ and $f_i(\gamma_{\Euc})$ are bounded above by a universal constant.
\end{lemma}

\begin{proof} The map $f_i$ maps $W_i$ conformally to the disk $\mathbb{D}_{\tanh(\radius(W_i)/2)}$ (sending $\cent(W_i)$ to $0$) such that the restriction of the Poincare metric $\rho_{\mathbb{D}}$ on $\mathbb{D}$ is $\rho_X$ on $W_i$. Geodesics with respect to $d_{\Euc}$ correspond to straight lines in $\mathbb{D}_{\tanh(\radius(W_i)/2)}$, whereas geodesics with respect to $\rho_{\mathbb{D}}$ are circular arcs which meet $\partial \mathbb{D}$ at a right angle. Since $\radius(W_i)\leq \arcsinh(1)/8$ (see condition 2 of \cref{covering lemma}), $\tanh(\radius(W_i)/2)$ is bounded away from $1$. So a geodesic of $\rho_{\mathbb{D}}$ passing through $\mathbb{D}_{\tanh(\radius(W_i)/2)}$ must have curvature bounded above by a universal constant $C$. If two points in $\mathbb{D}$ are sufficiently close, then the line between them and constant curvature arc (of curvature at most $C$) between them meet at an angle of at most $\pi/7$.  The lemma follows. 
\end{proof}

Thus we obtain, for each $0<r<r_0$, a triangulation of $X$ that we label $T_r$. There are two types of edges in $T_r$: those constructed in Step 3 and left unchanged (which we call flat edges), and those constructed in Step 5 or 6 (which we call hyperbolic edges). There are three types of triangles: those constructed in Step 3 and left unchanged (which we call equilateral triangles), those constructed in Step 5 (which we call semi-equilateral triangles), and those constructed in Step 6 (which we call hyperbolic triangles). We denote by $T_E$ the union of the equilateral triangles of $T_r$, $T_{SE}$ the union of the semi-equilateral triangles of $T_r$, and $T_H$ the union of the hyperbolic triangles of $T_r$. 

Denote by $T_r^1$ the $1$-skeleton of the triangulation $T_r$. We define a metric $d_{T_{r}^1}$ on $T_r^1$ as follows. On a hyperbolic edge, we define $d_{T_{r}^1}$ to the restriction of $d_{\rho_X}$ (the hyperbolic distance on $X$) to the hyperbolic edge. On a flat edge that was constructed by triangulating $f_{s(j)}(P_j)$ in Step 3, we define $d_{T_{r}^1}$ to be the restriction of the pullback under $f_{s(j)}$ of the Euclidean distance $d_{\Euc}$ on $f_{s(j)}(P_j)\subset \mathbb{D}$ to the flat edge. 

We let $S_r$ be the unique triangulated surface with the same underlying triangulation as $T_r$. Note that $T_r$ has at most $Cr^{-2}(R+g)$ triangles, since the number of equilateral and semi-equilateral triangles is bounded by $CNr^{-2}$ and the number of hyperbolic triangles is bounded by $$Cr^{-1}\sum_{i}\textstyle{\length_{\rho_X}(L_i)\radius(W_i)^{-1}}\leq Nr^{-1}.$$ Hence $$S_r\in \textstyle\Tri{g}{\leq Cr^{-2}(R+g)}.$$ Let $$X_r=\Phi(S_r)\in \mathcal{M}_g.$$ 

We will now show: 

\begin{theorem}\label{approximation theorem} For $0<r<r_0$, $d_T(X,X_r)<Cr$. Here, $C$ is a universal constant.
\end{theorem}

This directly implies \cref{lower bound}. First, we have a preliminary lemma. 

\begin{lemma}\label{hyp tri vs tri} Let $Q$ be a hyperbolic triangle contained in a hyperbolic ball of radius $r<1/4$. Then there exists a $Cr^2\Area(Q)^{-1}$-quasiconformal map from $Q$ (equipped with the hyperbolic metric) to the unit equilateral triangle $T$ (equipped with the Euclidean metric) which sends vertices of $Q$ to vertices of $T$ and is a scaling map on each of the three boundary edges.
\end{lemma}

\begin{proof} The hyperbolic triangle $Q$ admits a $C$-quasiconformal map to some flat triangle $Q_1$, as follows. Consider $Q$ as contained in $\mathbb{D}$ (equipped with the Poincare metric $\rho_{\mathbb{D}}$) with one of the vertices the origin. 

Let $$H=\{x,y,z\in \mathbb{R}|x^2+y^2+(z-1/2)^2=1/4, z\leq 1/2\}$$ be a hemisphere in $\mathbb{R}^3$. View $\mathbb{D}$ as a unit disk in the $x,y$-plane in $\mathbb{R}^3$, centered at the origin. Taking the inverse stereographic projection of $Q$ to $H$, then the orthographic projection back to the $x,y$-plane gives a map from $Q$ to a Euclidean triangle $Q_1$.

Away from $\partial \mathbb{D}$ and $\partial H$, the stereographic projection is $C$-bi-Lipschitz between the spherical metric on $H$ and the hyperbolic metric on $\mathbb{D}$. The orthographic projection is $C$-bi-Lipschitz between the spherical metric on $H$ and the Euclidean metric on its image. Since $Q$ is contained in a hyperbolic ball of radius $r<1/4$, we have a $C$-bi-Lipschitz map $F_1$ from $Q$ (with the hyperbolic metric) to $Q_1$ (with the Euclidean metric). Let $Q_2$ be $Q_1$ scaled by $r^{-2}$; then we have a map $F_2:Q_1\to Q_2$ that is conformal. The map $F_2$ satisfies the property that $F_2$ is $r^{-2}$-Lipschitz and $F_2^{-1}$ is $r^2$-Lipschitz. Finally, note that $Q_2$ is contained in a $C$-radius ball in the Euclidean metric. There exists an $\Area(Q_2)^{-1}$-quasiconformal affine map $F_3$ from $Q_2$ to $T$, which sends vertices of $Q_2$ to vertices of vertices of $T$ and is scaling on each boundary edge of $\partial Q_2$. Now, since $\Area(Q_2)$ is around $Cr^{-2}\Area(Q)$, $$F_3\circ F_2\circ F_1:Q\to T$$ is $Cr^2\Area(Q)^{-1}$-quasiconformal. Let $$\sca:\partial Q\to \partial T$$ be a map that sends vertices to vertices and is scaling on each boundary edge, with respect to the hyperbolic metric on $\partial Q$ and the Euclidean metric on $\partial T$. Then by construction, $$\sca \circ (F_3\circ F_2\circ F_1)^{-1}:\partial T\to \partial T$$ is $C$-bi-Lipschitz with respect to the Euclidean metric, therefore $C$-weakly-quasisymmetric. By \cref{extending bi-Lipschitz map 2}, it extends to a $C$-quasiconformal map from $T$ to $T$. Precomposing this map with $F_3\circ F_2\circ F_1$ gives the map desired in the lemma statement.
\end{proof}

\begin{proof}[Proof of \cref{approximation theorem}] We construct a map $F_r:X\to X_r$ triangle-by-triangle as follows. On vertices, $F_r$ is naturally defined. On edges, we define $F_r$ to be scaling (with respect to the metrics $d_{T_{r}^1}$ on $T_r^1$ and $d_{S_r}$, the canonical flat metric given by $S_r$, on $X_r$). Then $F_r$ conformally extends to all equilateral triangles (triangles in $T_E$) of $T_r$. To define $F_r$ on $T_{SE}$, note that by construction each semi-equilateral triangle $Q$ is contained in $V_i$ for some $i$. By \cref{se tri angle}, Step 6 of the triangulation construction, and conformal property of $f_i$, the triangular region $f_{i}(Q)$ has angles bounded below by $\pi/21$. By \cref{se tri angle}, Step 6 of the triangulation construction and the $C$-bi-Lipschitz property of $f_{i}$, the edges of the triangular region $f_i(Q)$ have lengths bounded above and below by $Cr\radius(W_i)$ and curvature bounded above by $C$. Therefore $f_{i}(\partial Q)$ is a $C$-quasicircle. Let $$\sca:\partial Q\to \partial T$$ denote a scaling map, with respect to the metric $d_{T_{r}^1}$ on $\partial Q$ and Euclidean metric on $\partial T$, that sends vertices of $\partial Q$ to vertices of $\partial T$. Since $$\sca\circ f^{-1}_{i}:f_{i}(\partial Q)\to \partial T$$ is $Cr^{-1}\radius(W_i)^{-1}$-Lipschitz with a $Cr\radius(W_i)$-Lipschitz inverse (with respect to the Euclidean metrics), it is $C$-weakly-quasisymmetric. Hence by \cref{extending bi-Lipschitz map 2}, $\sca\circ f^{-1}_{i}$ extends to a $C$-quasiconformal map from $f_{s(i)}(Q)$ to $T$ which means $F_r$ extends to a $C$-quasiconformal map on $Q$. Finally, by \cref{hyp tri vs tri}, $F_r$ extends to a $Cr^2\Area(Q)^{-1}$-quasiconformal map on each hyperbolic triangle $Q$ in $T_H$. In this way, we construct a map $F_r:X\to X_r$.
 
We use the characterization of the Teichm\"{u}ller metric in terms of extremal length from \cref{Teichmuller metric via extremal length}. We have on $\mathcal{M}_g$, $$d_T(X,X_r)=\frac{1}{2}\log \inf_f\sup_\gamma\frac{\Ext_{X}(\gamma)}{\Ext_{X_r}(f(\gamma))}$$ where the infimum runs through all quasiconformal maps $f:X\to X_r$ and the supremum runs through all free homotopy classes of simple closed curves $\gamma$ on $X$. To show this quantity is bounded by $Cr$, it suffices to show that for all free homotopy classes $\gamma$ of a simple closed curve on $X$, $$\textstyle\Ext_{X_r}(F_{r}(\gamma))\geq (1-Cr)\Ext_X(\gamma).$$ Recall that $$\textstyle\Ext_{X_r}(F_{r}(\gamma))=\displaystyle\sup_{\rho_{r}}\frac{\length_{\rho_r}(F_{r}(\gamma))^2}{\Area_{\rho_{r}}(X_r)}$$ where the supremum ranges over all conformal metrics $\rho_r$ on $X_r$. Similarly, $$\textstyle\Ext_X(\gamma)=\displaystyle\sup_\rho\frac{\length_\rho(\gamma)^2}{\Area_\rho(X)}$$ where the supremum ranges over all conformal metrics $\rho$ on $X$. By \cref{Jenkins-Strebel}, the $\rho$ which achieves this supremum is given by $|\phi|^{1/2}$ for a holomorphic quadratic differential $\phi$ on $X$. To show $$\textstyle\Ext_{X_r}(F_r(\gamma))\geq (1-Cr)\Ext_X(\gamma),$$ it suffices to exhibit a conformal metric $\rho_r$ on $X_r$ such that $$\frac{\length_{\rho_r}(F_r(\gamma))^2}{\Area_{\rho_r}(X_r)}\geq (1-Cr)\displaystyle\frac{\length_{\phi}(\gamma)^2}{\Area_{\phi}(X)}.$$ To do this, we consider the metric $(F_r)_*|\phi|^{1/2}$ (which is not a conformal metric on $X_r$), and let $\rho_r$ be the smallest conformal metric on $X_r$ such that $\rho_r\geq (F_{r})_*|\phi|^{1/2}$. Then $$\textstyle\length_{\rho_r}(F_r(\gamma))^2\geq \textstyle\length_{\phi}(\gamma)^2$$ by construction. It remains to see that $$\textstyle\Area_{\rho_r}(X_r)\leq (1+Cr)\textstyle\Area_{\phi}(X).$$ To see this, for each $1\leq i\leq N$ let $$m_i =\sup_{x\in V_i}\frac{|\phi|}{\rho_X^2}(x).$$ By the mean value property, $$\int_{U_i}|\phi|\geq C\radius(U_i)^2 m_i.$$ Therefore, by condition 4 of \cref{covering lemma}, 
\begin{equation}\label{eq51}\textstyle\Area_{\phi}(X)\geq \displaystyle C\sum_{i=1}^N \radius(U_i)^2 m_i.
\end{equation}

Now, for any equilateral triangle $Q$ that is a face of $T_r$, $\Area_{\rho_r}(F_r(Q))=\Area_\phi(Q)$ because $F_r$ is conformal on $Q$. Therefore, 
\begin{equation}\label{eq52}\textstyle\Area_{\rho_r}(F_r(T_E))\leq \Area_{\phi}(X)
\end{equation}

Any triangle $Q$ in $T_{SE}$ is contained in $W_i$ for some $1\leq i\leq N$. By $C$-quasiconformality of $F_r$ on $Q$,
\begin{equation}\label{eq53}
\begin{aligned}\textstyle\Area_{\rho_r}(F_r(Q))&\leq C\textstyle\Area_{\phi}(Q)\\&\leq Cm_i\textstyle\Area_{\rho_X}(Q).
\end{aligned}
\end{equation} Now, by Step 5 in the triangulation construction semi-equilateral triangles contained in $V_i$ are actually contained in a $Cr\radius(U_i)$ hyperbolic neighborhood of $\left(\cup_{i=1}^N L_i\right) \cap V_i$, and the latter has length at most $C\radius(U_i)$ in the hyperbolic metric by condition 3 and condition 5 of \cref{covering lemma}. Summing \cref{eq53} over all semi-equilateral triangles $Q$ we obtain 
\begin{equation}\label{eq54}
\begin{aligned}\textstyle\Area_{\rho_r}(F_r(T_{SE}))&\leq Cr\sum_{i=1}^N \radius(U_i)^2 m_i\\&\leq Cr \textstyle\Area_{\phi}(X)
\end{aligned}
\end{equation}
by \cref{eq51}.

Finally, any triangle $Q$ in $T_H$ is contained in $V_i$ for some $1\leq i\leq N$. By $Cr^2\Area_{\rho_X}(Q)^{-1}$-quasiconformality of $F_r$ on $Q$, we have 
\begin{equation}\label{eq55}
\begin{aligned}\textstyle\Area_{\rho_r}(F_r(Q))&\leq Cr^2\textstyle\Area_{\rho_X}(Q)^{-1}\textstyle\Area_{\phi}(Q)\\&\leq Cr^2 m_i.
\end{aligned}
\end{equation} Since by Step 6 of the triangulation construction at most $C\radius(U_i)^2r^{-1}$ hyperbolic triangles $Q$ are contained in any $V_i$, summing \cref{eq55} over all hyperbolic triangles $Q$ we obtain 
\begin{equation}\label{eq56}
\begin{aligned}\textstyle\Area_{\rho_r}(F_r(T_{H}))&\leq Cr\sum_{i=1}^N \radius(U_i)^2 m_i\\&\leq Cr \textstyle\Area_{\phi}(X)
\end{aligned}
\end{equation}
by \cref{eq51}.

\cref{eq52}, \cref{eq54} and \cref{eq56} together give
\begin{align*}\textstyle\Area_{\rho_r}(X_r)&=\textstyle\Area_{\rho_r}(F_r(T_E))+\Area_{\rho_r}(F_r(T_{SE}))+\Area_{\rho_r}(F_r(T_H))\\&\leq \textstyle(1+Cr)\Area_{\phi}(X)
\end{align*} which completes the proof.
\end{proof}

\section{Approximating triangulated surfaces with bounded degree triangulations}

\subsection{Approximation theorem} The goal of this section is to show the following theorem:

\begin{theorem}\label{bounded degree surface} There exists a map $B:\Tri{g}{T}\to \Tri{g}{\leq \sigma T}$, and universal constants $C, \sigma, \mu>0$, such that: 

\begin{enumerate}

\item There is a $C$-quasiconformal map $f_S:S\to B(S)$,

\item The maximum degree of any vertex of $B(S)$ is $7$,

\item Given any two vertices $x,y\in V_{\neq 6}(B(S))$ (not necessarily distinct), and $\gamma$ any arc from $x$ to $y$ that is homotopically nontrivial in $B(S)\setminus \{x,y\}$, we have $$\int_\gamma |\psi_{B(S)}|^{1/6}\geq 3,$$

\item $$|V_{\neq 6}(B(S))|\leq \mu(|V_{\neq 6}(S)|+g),$$ and

\item The fiber of $B$ over $B(S)$ has cardinality at most $C^{|V_{\neq 6}(B(S))|}$.
\end{enumerate}
Here, $C$, $\sigma$ and $\mu$ are universal constants.
\end{theorem}

To show \cref{bounded degree surface}, we first construct a local replacement for high degree vertices as follows.

For $d\in \mathbb{N}$, define the triangulated disk $TD_d$ to be the following triangulation of a topological disk by unit equilateral triangles: $d$ unit equilateral triangles are glued together to form a topological disk with one interior vertex of degree $d$. The boundary of $TD_d$, denoted $\partial TD_d$, consists of $d$ edges. We thus obtain a triangulated surface with boundary that we also call $TD_d$, which comes with a flat metric that may have a singularity at the interior vertex.

\begin{lemma}\label{tri disk vs disk} There is a conformal map $f:TD_d\to \overline{\mathbb{D}}$ such that $f:\partial TD_d\to S^1$ is a scaling map with respect to the restriction of the flat metric on $TD_d$ to $\partial TD_d$, and the Euclidean metric on $S^1$.
\end{lemma}

\begin{remark} Scaling map here means all distances get scaled by a constant factor which is $$\displaystyle\frac{\length(\partial TD_d)}{\length (S^1)}.$$
\end{remark}

\begin{proof} Denote by $W_{\pi/3}$ the closed sector of the unit circle with angle $\pi/3$, and denote by $T$ the unit equilateral triangle. By \cref{extending bi-Lipschitz map 2}, there exists a $C$-quasiconformal map $W_{\pi/3}\to T$ that sends boundary vertices of $W_{\pi/3}$ to boundary vertices of $T$ and is scaling on each boundary edge. Gluing, we obtain a conformal map from $TD_d$ to the unit cone of angle $d\pi/3$, which is scaling on the boundary. The unit cone of angle $d\pi/3$ admits a conformal map to $\overline{\mathbb{D}}$ which is scaling on the boundary. Composing these two maps, we obtain the statement in the lemma.
\end{proof}

For $d\geq 8$, we now construct the triangulated hyperbolic disk $TH_d$ to be another topological disk formed by gluing unit equilateral triangles satisfying $\partial TH_d\simeq \partial TD_d$ (that is, $TH_d$ also has $d$ boundary edges). However, $TH_d$ will have more interior points, and therefore more triangles. The number of triangles in $TH_d$ will be bounded above by $Cd$. We construct $TH_d$ inductively by constructing annular layers starting from its boundary.

Label the boundary points of $\partial TH_d\simeq \partial TD_d$ by $x_{0,0},...,x_{0,d-1}$. If $d\geq 8$, construct $TA_1$, a triangulated annulus with outer and inner boundaries $\partial TA_1^+$ and $\partial TA_1^-$, respectively, such that 

\begin{enumerate} 

\item $\partial TA_1^+=\partial TH_d$, 

\item $\deg_{TA_1}(x_{0,j})=3$ if $j$ is even and $4$ if $j$ is odd,

\item $TA_1$ has no interior vertices.

\end{enumerate}

(Recall that the degree of a vertex of a triangulated surface with boundary is the number of edges emanating from the vertex.) These conditions determine $TA_1$ uniquely. That is, $\partial TA_1^-$ consists of $d_1$ edges and vertices where $d_1=\lfloor d/2\rfloor$. The vertices of $\partial TA_1^-$ may be labelled $x_{1,0},...,x_{1,d_1-1}$ (in cyclic order), such that the following holds. 

When $d$ is even, in $TA_1$, $x_{0,j}$ is connected by an edge to $x_{1,j/2}$ if $j$ is even, and $x_{0,j}$ is connected by an edge to $x_{1,(j-1)/2}$ and $x_{1,(j+1)/2}$ if $j<d-1$ is odd. Finally, $x_{0,d-1}$ is connected by an edge to $x_{1,(d-2)/2}$ and $x_{1,0}$. 

When $d$ is odd, $x_{0,j}$ is connected by an edge to $x_{1,j/2}$ if $j<d-1$ is even, and $x_{0,j}$ is connected by an edge to $x_{1,(j-1)/2}$ and $x_{1,(j+1)/2}$ if $j<d-2$ is odd. Moreover, $x_{0,d-1}$ is connected by an edge to $x_{1,0}$, while $x_{0,d-2}$ is connected by an edge to $x_{1,(d-3)/2}$ and $x_{1,0}$.

\begin{figure}[ht]

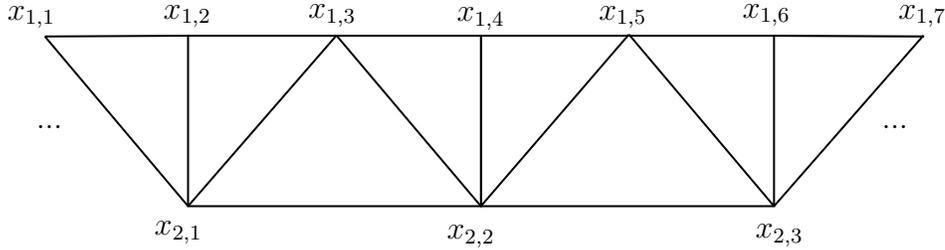

\caption{A piece of $TA_1$, with the triangles not to scale.}

\end{figure}

We now describe the second inductive step, constructing another annulus $TA_2$ with outer and inner boundaries $\partial TA_2^+$ and $\partial TA_2^-$, respectively, such that

\begin{enumerate}

\item $\partial TA_2^+=\partial TA_1^-$,

\item $\deg_{TA_2}(x_{1,j})=3$ if $j$ is even and $4$ if $j$ is odd,

\item $TA_2$ has no interior vertices.

\end{enumerate}

We continue this inductive process. At the $i$th inductive step, wherein the vertices of $TA_{i-1}^-$ are denoted $x_{i-1,0},...,x_{i-1,d_{i-1}-1}$  in cyclic order (where $d_{i-1}\sim d/2^{i-1}$), we construct annulus $TA_i$ with outer and inner boundaries $\partial TA_i^+$ and $\partial TA_i^-$ such that 

\begin{enumerate}

\item $\partial TA_i^+=\partial TA_{i-1}^-$,

\item $\deg_{TA_i}(x_{i,j})=3$ if $j$ is even and $4$ if $j$ is odd,

\item $A_i$ has no interior vertices.

\end{enumerate}

Here, $\partial TA_i^-$ consists of $d_i$ edges and vertices where $d_i=\lfloor d_{i-1}/2\rfloor$. The vertices of $\partial TA_i^-$ may by labelled $x_{i,0},...,x_{i,d_i-1}$ (in cyclic order) such that the following holds. 

When $d_{i-1}$ is even, in $TA_i$, $x_{i-1,j}$ is connected by an edge to $x_{i,j/2}$ if $j$ is even, and $x_{i-1,j}$ is connected by an edge to $x_{i,(j-1)/2}$ and $x_{i,(j+1)/2}$ if $j<d_{i-1}-1$ is odd. Finally, $x_{i-1,d_{i-1}-1}$ is connected by an edge to $x_{i,(d_{i-1}-2)/2}$ and $x_{i,0}$. 

When $d_{i-1}$ is odd, $x_{i-1,j}$ is connected by an edge to $x_{i,j/2}$ if $j<d_{i-1}-1$ is even, and $x_{i-1,j}$ is connected by an edge to $x_{i,(j-1)/2}$ and $x_{i,(j+1)/2}$ if $j<d_{i-1}-2$ is odd. Moreover, $x_{i,d_{i-1}-1}$ is connected by an edge to $x_{i,0}$, while $x_{i-1,d_{i-1}-2}$ is connected by an edge to $x_{i,(d_{i-1}-3)/2}$ and $x_{i,0}$.

We continue until we reach the $k$th step wherein $d_{k-1}\leq 7$. We let $TA_k=TD_{d_{k-1}}$ and glue $TA_k$ to $TA_{k-1}$ along the boundary $T\partial A_{k-1}^-$. Our construction of the $TA_1,...,TA_k$ naturally identifies $\partial TA_{i-1}^-$ with $\partial TA_i^+$. We define $TH_d$ to be the union of the $TA_1,...,TA_k$. We call the unique interior vertex of $TA_k$ the center of $TH_d$. From the construction described above, the following lemma is evident. 

\begin{lemma}\label{successive closed stars} Let $v$ be the center of $TH_d$. For $2\leq i\leq k$, the closed star of the union of $TA_{i},...,TA_k$ is the union of $TA_{i-1},...,TA_{k}$, and $TA_k$ is the closed star of $v$. In particular, the $k$th successive closed star of $v$ is $TH_d$.
\end{lemma}

We also have the following lemma. 

\begin{lemma}\label{degrees of TH_d vertices} There are at most $Cd$ vertices of $TH_d$. All interior vertices of $TH_d$ have degree at most $7$ and all boundary vertices have degree at most $4$. Moreover, each $TA_{i}$ for $2\leq i\leq k$ contains at least one outer boundary vertex that has degree $7$ in $TH_d$.
\end{lemma}

\begin{proof} The total number of vertices of $TA_i$ is at most $Cd/2^i$, so the number of vertices of $TH_d$ is at most $Cd$. By construction, outer boundary vertices of $TA_i$ have degree at most $4$ (in particular, $x_{i-1,0}$ has degree $3$). Inner boundary vertices of $TA_i$ have degree at most $5$, except for vertex $x_{i,0}$ which has degree $6$. From this we obtain the second claim in the lemma. Finally, for $2\leq i\leq k$, $x_{i-1,0}$ has degree $7$ in $TH_d$.
\end{proof}

Our next goal is to show the following lemma. 

\begin{lemma}\label{triangulated disk vs triangulated hyp disk} There exists a $C$-quasiconformal map $$f:TH_d\to TD_d$$ which on the boundaries agrees with an identification $\partial TH_d\simeq \partial TD_d$. 
\end{lemma}

To do this, we have the following preliminary lemma. 

\begin{lemma}\label{annulus vs tri annulus} Each triangulated annulus $TA_i$ admits a $C$-quasiconformal map to a closed annulus $$A(r,1)=\{x\in \mathbb{C}|r\leq |x|\leq 1\}$$ whose restriction $\partial TA_i\to \partial A(r,1)$ is a scaling map. Here, $C$ is a universal constant independent of $TA_i$.
\end{lemma}

\begin{proof} We will construct a quasiconformal map to $$A(\alpha,\beta)=\{x\in \mathbb{C}|\alpha\leq |x|\leq \beta\}.$$ Then, composing with a scaling map gives the lemma statement. We will choose $\alpha$ and $\beta$ later. We first partition $A(\alpha,\beta)$ into triangular regions. Each triangular region consists of three boundary components, two of which are straight lines and one which is a circular arc. Let $(r,\theta)$ be polar coordinates on $\mathbb{C}$ centered at $0$. Denote by $$C_-=\{(r,\theta)|r=\alpha\}$$ and $$C_+=\{(r,\theta)|r=\beta\}$$ the two boundary circles of $A(\alpha,\beta)$ of radius $\alpha$ and $\beta$ respectively. Recall that the vertices of $\partial TA_{i}^+$ are $x_{i-1,0},...,x_{i,d_{i-1}-1}$ and the vertices of  $\partial TA_i^-$ are $x_{i,0},...,x_{i,d_{i}-1}$, where $d_{i}=\lfloor d_{i-1}/2\rfloor$. Now, let $y_{i-1,j}$ denote the vertex $$(r,\theta)=(\beta,2\pi j/d_{i-1})$$ (for $0\leq j\leq d_{i-1}-1$). Let $y_{i,j}$ denote the vertex $$(r,\theta)=(\alpha, 2\pi j/d_i)$$ (for $0\leq j\leq d_i-1$). We construct a triangulation of $A(\alpha,\beta)$ as follows. We let the vertices of the triangulation be the $y_{i-1,j}$ and $y_{i,j}$. Two vertices $y_{i-1,j_1}$ and $y_{i,j_2}$ are connected by an edge if $x_{i-1,j_1}$ and $x_{i,j_2}$ in $TA_i$ are connected by an edge. In this case, the edge between $y_{i-1,j_1}$ and $y_{i,j_2}$ is a straight line. 

\begin{figure}[ht]

\hspace{.4in}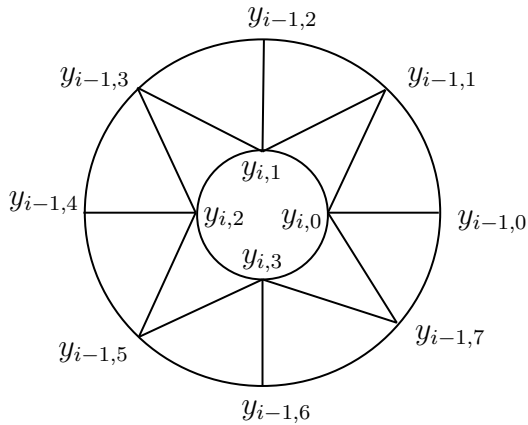

\caption{Triangulation of $A(\alpha,\beta)$ when $d_{i-1}=8$, which is combinatorially equivalent to $TA_i$ when $d_{i-1}=8$.}

\end{figure}

For $d_i$ sufficiently large, we take $\alpha=d_{i-1}-1$ and $\beta=d_{i-1}$, and this decomposition of $A(\alpha,\beta)$ into vertices and edges gives a triangulation of $A(\alpha,\beta)$. This is true since by the construction of $TA_i$, two vertices $y_{i-1,j_1}$ and $y_{i,j_2}$ are connected by an edge only if the distance between them is at most $C$. For small $d_i$ but still greater than $7$, we take $\alpha=1/10$ and $\beta=1$ to obtain a triangulation where the angles of each triangular region are nonzero. We claim that when $d_i$ is large, the angles of each triangular region are bounded below by a universal constant. We also claim that when $d_i$ is large, the lengths of the triangular sides are bounded below and above by universal constants. To see this, note that each triangular region $Q$ consists of two straight lines along with a circular arc. Associated to $Q$ is a genuine flat triangle $Q'$ which may be obtained by replacing the circular arc with a straight line, which we call the base of $Q'$. By construction, the circular arc has length bounded below and above by universal constants, and since $d_i$ is assumed to be sufficiently large, the straight line approximates the circular arc. So the length of the base of the triangle is bounded below and above by universal constants. Since the two circles $C_\alpha$ and $C_\beta$ are distance $1$ apart, the height of each triangle is also bounded below and above by universal constants. Hence, $\partial Q$ has three sides, two of which are straight lines and the third a circular arc, each with lengths bounded below and above by universal constants. Moreover the angles of $\partial Q$ are bounded below by a universal constant because the angles of $\partial Q'$ are bounded below by a universal constant. Therefore for all $d_i$, each triangular region $Q$ is a $C$-quasicircle, where $C$ is independent of $d_i$. The map from $\partial Q$ to $\partial T$ (the boundary of an equilateral triangle) which sends vertices of to vertices and is a scaling map on each of the three sides is $C$-bi-Lipschitz with respect to the Euclidean metric, therefore also $C$-weakly-quasisymmetric. Using \cref{extending bi-Lipschitz map 2} on each triangular region, we obtain a map to the unit equilateral triangle that is scaling on each boundary component. Gluing the inverses of these maps together, we obtain a map from $TA_i$ to $A(\alpha,\beta)$ which is a scaling map on the boundary.
\end{proof}

\begin{proof}[Proof of \cref{triangulated disk vs triangulated hyp disk}] We already have shown that there exists a $C$-quasiconformal map $$TD_d\to \overline{\mathbb{D}}$$ which is a scaling map on the boundary. To show \cref{triangulated disk vs triangulated hyp disk}, it suffices to construct a $C$-quasiconformal map $$TH_d\to \overline{\mathbb{D}}$$ that is a scaling map on the boundary. We do this inductively as follows. First, by \cref{annulus vs tri annulus}, $TA_1$ admits a $C$-quasiconformal map to the annulus $A(r_1,1)$ which is a scaling map on the boundary. Then $TA_2$ admits a $C$-quasiconformal map to the annulus $A(r_1r_2, r_1 )$ which is a scaling map on the boundary. In general, for $i\in \{1,...,k-1\}$,  $TA_i$ admits a $C$-quasiconformal map to the annulus $A(r_1...r_i,r_1...r_{i-1})$ which is scaling on the boundary. Finally, by \cref{tri disk vs disk} $TA_k$ (which is a triangulated disk) admits a $C$-quasiconformal map to the Euclidean disk of radius $r_1...r_{k-1}$ which is a scaling map on the boundary. These maps (after possible rotations) glue together to give the desired map $TH_d\to \overline{\mathbb{D}}$. 
\end{proof}

\begin{proof}[Proof of \cref{bounded degree surface}]
First, we replace $S$ with a surface $S_1$ in $\Tri{g}{16T}$ which is the $4$-subdivision of $S$ rescaled. Then we may replace the closed star around every vertex in $S_1$ of degree $d$ greater than $7$ (which is a copy of $TD_d$) by $TH_d$ to obtain a triangulated surface $S_2$. (These closed stars are all disjoint.) Then, take the rescaled $5$-subdivision of $S_2$ to obtain $B(S)$, a genus $g$ triangulated surface with at most $\sigma T$ triangles, for a constant $\sigma$ independent of $T$ and $g$. By construction, conditions 2 and 3 in the statement of \cref{bounded degree surface} are satisfied. Applying \cref{triangulated disk vs triangulated hyp disk} for each triangulated disk replacement and gluing, we have a $C$-quasiconformal map $f:S\to S_2$, and $S_2$ is naturally conformally equivalent to $B(S)$. This shows condition 1 in the statement of \cref{bounded degree surface}.

Now, $$|V_{\neq 6}(S)|=|V_{\neq 6}(S_1)|$$ and $$|V_{\neq 6}(B(S))|=|V_{\neq 6}(S_2)|.$$ Since $S_2$ is formed by replacing disjoint copies of $TD_d$ in $S_1$ with $TH_d$, by \cref{degrees of TH_d vertices}, we have $$|V_{\neq 6}(S_2)|\leq |V_{<6}(S_1)|+C\sum_{x\in V_{>6}(S_1)}\deg x.$$ By Euler characteristic considerations, $$\sum_{x\in V_{>6}(S_1)}\deg x\leq C|V_{<6}(S_1)|+Cg.$$ So
\begin{align*}|V_{\neq 6}(S_2)|&\leq C|V_{<6}(S_1)|+Cg\\&\leq C|V_{\neq 6}(S_1)|+Cg.
\end{align*} Hence $$|V_{\neq 6}(B(S))|\leq C(|V_{\neq 6}(S)|+g),$$ showing condition 4 in the statement of \cref{bounded degree surface}.

It remains to bound the cardinality of the fibers of $B$. Given $B(S)$, we may recover $S_2$ in the following way. We start by choosing a vertex of $B(S)$ with degree strictly greater than $6$ (which must exist by Euler characteristic considerations). This vertex must also be a vertex of $S_2$, and from this vertex we may inductively reconstruct $S_2$. Recall that $S_1$ may be constructed from $S_2$ by replacing certain (disjoint) copies of $TH_d$ by copies of $TD_d$. Each copy $TD_d$ contains its center, which is a vertex of degree $d>6$, and these centers must be distinct. There are therefore at most $C^{|V_{\neq 6}(B(S))|}$ total choices for the set of centers, and the set of centers has cardinality at most $|V_{\neq 6}(B(S))|$.

\begin{lemma}\label{two choices for d} Suppose $v\in V(S_2)$ is the center of a copy of $TH_d$. Given $S_2$ and $v$, there are at most two possible choices for $d$.
\end{lemma} 

\begin{proof} Suppose that $TH_{d}$, $TH_{d'}$ and $TH_{d''}$, with $d<d'<d''$ are all contained in $S_2$ and centered at $v$. By construction of $S_2$, all the boundary vertices of the closed star of $TH_d$ have degree $6$ in $S_2$. (This is true because $S_1$ is the rescaled $4$-subdivision of $S$, and to construct $S_2$ from $S_1$ we only replace stars of vertices of degree $d\geq 8$ by copies of $TH_d$.) Since $TH_{d'}$ and $TH_{d''}$ are also subsets of $S_2$ centered at $v$ and $d<d'<d''$, by \cref{successive closed stars} we must have that the closed star of $TH_d$ in $S_2$ is contained in the interior of $TH_{d''}$. In other words, boundary vertices of the closed start of $TH_d$ are interior vertices of $TH_{d''}$. This contradicts \cref{degrees of TH_d vertices}.
\end{proof}

We continue with the proof of \cref{bounded degree surface}. Given a set of centers, by \cref{two choices for d} there are at most $C^{|V_{\neq 6}(B(S))|}$ possibilities for the choice of triangulated hyperbolic disks $TH_d$ centered at these centers. Once these disks are chosen, there is a unique choice of replacement ($TD_d$) for each triangulated hyperbolic disk, hence $S_1$ may be reconstructed. Finally, given $S_1$, $S$ can be recovered by choosing a vertex of $S_1$ with degree strictly greater than $6$, which must also be a vertex of $S$, and inductively reconstructing $S$ starting from this vertex. Thus, given $B(S)$, there are at most $C^{|V_{\neq 6}(B(S))|}$ possibilities for $S$. This shows condition 5 in the statement of \cref{bounded degree surface}.
\end{proof}

\subsection{Upper bounds for triangulated surfaces in terms of locally bounded surfaces}

We have the following corollary of \cref{bounded degree surface}.

\begin{cor}\label{trim vs tricm} There exists a universal constant $C$ such that $$\Ntri(T,g,m,r)\leq C^{m+g} \Ntric(\sigma T,g,\mu (m+g),r+C)$$ for $g\geq 2$. Here, $\sigma$ and $\mu$ are the constants defined in the statement of \cref{bounded degree surface}.
\end{cor}

\begin{proof} Let $X\in \mathcal{T}_g$. Suppose $S\in \Tri{g}{\leq T,\leq m}$ such that $\Phi(S)\in B_{d_T}(X,r)$. By \cref{bounded degree surface}, there exists a triangulated surface $B(S)$ in $\Tric{g}{\leq \sigma T,\leq \mu (m+g)}$ such that $$d_T(\Phi(B(S)),\Phi(S))\leq C.$$ Here, $B$ is the map defined in statement of \cref{bounded degree surface}. Since 
\begin{align*}|V_{\neq 6}(B(S))|&\leq \mu (|V_{\neq 6}(S)|+g)\\&\leq \mu (m+g),
\end{align*} the fibers of $B$ have cardinality at most $C^{m+g}$. Summing over all $g'\leq g$, we have $$\Ntri(T,g,m,r)\leq C^{m+g} \Ntric(\sigma T,g,\mu (m+g), r+C),$$ as desired.
\end{proof}

\section{Upper bounds for combinatorial translation surfaces via triangulated surfaces}\label{bound tran by tri}

In this section, we shall prove the following result.

\begin{lemma}\label{tranc vs tri} There exists a universal constant $C$ such that $$\Ntranc(T,g,r)\leq (T/g)^{C(1+r)g}\sum_{ \substack{n\leq (1/100)g\\T_1+...+T_n\leq 2T\\ g_1+...+g_n\leq \mu^{-1}(1/100)g\\g_1,...,g_n\geq 2\\ m_1+...+m_n\leq \mu^{-1}(1/100)g}} \prod_{i=1}^n \Ntri(T_i,g_i,m_i, r+C)$$ for $g\geq 2$. Here, $\mu$ is the constant defined in the statement of \cref{bounded degree surface}.
\end{lemma}

\subsection{Hodge norms and roadmap to prove \cref{tranc vs tri}}\label{6.0} Let $X\in \mathcal{M}_g$. We will first compute $$|\{S_Y\in \textstyle\Tranc{g}{T}|\Phi(S_Y)\in B_{d_T}(X,r)\}|.$$ To do this, we will first compute $$|\{S_Y\in \textstyle\Tranc{g}{T}|\Phi(S_Y)\in B_{d_T}(X,(g/T)^{\kappa_0})\}|$$ for an appropriate integer $\kappa_0\in \mathbb{N}$ to be chosen later. Then we will use \cref{Teichmuller bound 3} to compute $$|\{S_Y\in \textstyle\Tranc{g}{T}|\Phi(S_Y)\in B_{d_T}(X,r)\}|.$$ 

For any $$Y\in B_{d_T}(X,(g/T)^{\kappa_0}),$$ there exists a diffeomorphism surfaces $$f:X\to Y$$ such that $f$ is $1+8(g/T)^{\kappa_0}$-quasiconformal. 

We define a map $$H_X: \{S_Y\in \textstyle\Tranc{g}{T}|\Phi(S_Y)\in B_{d_T}(X,(g/T)^{\kappa_0})\}\to H^1(X,\mathbb{C})$$ that sends $S_Y$ to the cohomology class represented by $f^*\phi_{S_Y}$. We count the quantity $$|\{S_Y\in \textstyle\Tranc{g}{T}|\Phi(S_Y)\in B_{d_T}(X,(g/T)^{\kappa_0})\}|$$ in two steps. First, we compute the number of $S_Y\in \Tranc{g}{T}$ such that $$\Phi(S_Y)\in B_{d_T}(X,(g/T)^{\kappa_0})$$ and $H_X(S_Y)$ is close in the Hodge metric to a fixed cohomology class in $H^1(X,\mathbb{C})$. Then we bound the number of cohomology classes, quantitatively. 

\begin{lemma}\label{quantitative form determines triangulation} Let $X\in \mathcal{M}_g$ and suppose $S_X\in \Tranc{g}{T}$ such that $\Phi(S_X)=X$. Then there are at most $$(T/g)^{Cg}\sum_{\substack{n\leq (1/100)g\\T_1+...+T_n\leq T\\ g_1+...+g_n\leq \mu^{-1}(1/100)g\\g_1,...,g_n\geq 2\\ m_1+...+m_n\leq \mu^{-1}(1/100)g}} \prod_{i=1}^n \Ntri(2T_i,g_i,m_i, r+C)$$ number of $S_Y\in \Tranc{g}{T}$ that satisfy the following properties:
\begin{enumerate}

\item $$Y=\Phi(S_Y)\in B_{d_T}(X,(g/T)^{\kappa_0})$$ and

\item $$\| \phi_{S_X}-f^* \phi_{S_Y}\|_X\leq \alpha_1(g/T)^{\kappa_1}g^{1/2}.$$

\end{enumerate} Here, $\alpha_1$ is a sufficiently small universal constant and $\kappa_1$ is a sufficiently large universal constant. We choose these constants in \cref{Constants}. 
\end{lemma}

As a corollary:

\begin{cor}\label{count cohomology classes} We have,
\begin{multline*}|\{S_Y\in \textstyle\Tranc{g}{T}|\Phi(S_Y)\in B_{d_T}(X,(g/T)^{\kappa_2}))\}|  \\ \leq  (T/g)^{Cg}\sum_{\substack{n\leq (1/100)g\\ T_1+...+T_n\leq T\\ g_1+...+g_n\leq \mu^{-1}(1/100)g\\g_1,...,g_n\geq 2\\ m_1+...+m_n\leq \mu^{-1}(1/100)g}} \prod_{i=1}^n \Ntri(2T_i,g_i,m_i, r+C)
\end{multline*}
where $\kappa_2$ is a sufficiently large universal constant chosen in \cref{Constants}, satisfying $\kappa_2\geq \kappa_0$, $(\kappa_2-1)/2\geq \kappa_1$ and $$100\alpha_1^{-1}(1/2)^{(\kappa_2-1)/2-\kappa_1}\leq 1.$$
\end{cor}

First, a preliminary lemma.

\begin{lemma}\label{Hodge norm of tran surface} The Hodge norm squared of $\phi_S$ is $$\int_X \phi_S\wedge * \overline{\phi_S}=(\sqrt{3}/{4})T.$$ 
\end{lemma}

\begin{proof} On the equilateral triangle with vertices $0,1,\frac{1}{2}+\frac{\sqrt{3}}{2}i$ in $\mathbb{C}$, the Hodge norm squared of $dz$ is $\frac{\sqrt{3}}{4}$. Since there are $T$ triangles in $S$, the lemma follows.
\end{proof}

\begin{proof}[Proof of \cref{count cohomology classes}] The key idea is to reduce \cref{count cohomology classes} to \cref{quantitative form determines triangulation} by using \cref{pullback of harmonic form close to harmonic representative}. \cref{pullback of harmonic form close to harmonic representative} allows us to deduce, from a condition about cohomology classes being close (condition 2 of \cref{quantitative form determines triangulation}), a much stronger condition about the individual forms being close averaged over the surface.

Let $S_Y\in \Tranc{g}{T}$ satisfying $$\Phi(S_Y)\in B_{d_T}(X,(g/T)^{\kappa_2}).$$ By \cref{Hodge norm of tran surface} and \cref{qc map Hodge norm cohomology}, $$\|H_X(S_Y)\|\leq C T^{1/2}.$$ Now, cover the $CT^{1/2}$ radius Hodge norm ball in $H^1(X,\mathbb{C})$ centered at $0$ by $(T/g)^{Cg}$ number of $(g/T)^{\kappa_3}g^{1/2}$ radius balls $B_1,...,B_{(T/g)^{Cg}}$. Here, $\kappa_3$ is a constant to be chosen, and we assume that $\kappa_3\geq (\kappa_2-1)/2$. Given $$S_Y,S_Z\in \textstyle\Tranc{g}{T}$$ satisfying $$\Phi(S_Y),\Phi(S_Z)\in B_{d_T}(X,(g/T)^{\kappa_2})$$ and $$H_X(S_Y), H_X(S_Z)\in B_k,$$ let $f:X\to Y$ and $g:Y\to Z$ be $1+8(g/T)^{\kappa_2}$-quasiconformal maps. Let $$(f^*\phi_{S_Y})^h$$ be the harmonic form on $X$ representing $H_X(S_Y)$, the cohomology class of $f^*\phi_{S_Y}$. Let $$((g\circ f)^*\phi_{S_Z})^h$$ be the harmonic form on $X$ representing $H_X(S_Z)$, the cohomology class of $(g\circ f)^*\phi_{S_Z}$.

By \cref{pullback of harmonic form close to harmonic representative} and \cref{Hodge norm of tran surface}, $$\|f^*\phi_{S_Y}-(f^*\phi_{S_Y})^h\|_X\leq 4(\sqrt{3}/4)^{1/2}(g/T)^{\kappa_2/2}T^{1/2}$$ and 
$$\|(g\circ f)^*\phi_{S_Z}-((g\circ f)^*\phi_{S_Z})^h\|_X\leq 8(\sqrt{3}/4)^{1/2}(g/T)^{\kappa_2/2}T^{1/2}.$$ Since $H_X(S_Y),H_X(S_Z)\in B_k$, $$\|(f^*\phi_{S_Y})^h-(f^*\phi_{S_Z})^h\|\leq 2(g/T)^{\kappa_3}g^{1/2}.$$ Summing, we obtain 
\begin{align*}\|f^*\phi_{S_Y}-(g\circ f)^*\phi_{S_Z}\|&\leq (20(g/T)^{(\kappa_2-1)/2}+2(g/T)^{\kappa_3})g^{1/2}\\&\leq 40(g/T)^{(\kappa_2-1)/2}g^{1/2}
\end{align*} assuming $\kappa_3\geq (\kappa_2-1)/2$. Pulling back to $Y$ under $f^{-1}$, by \cref{qc map Hodge norm diff form} we obtain $$\|\phi_{S_Y}-g^*\phi_{S_Z}\|\leq 100(g/T)^{(\kappa_2-1)/2}g^{1/2}.$$ Since $g/T\leq 1/2,$ the assumptions $\kappa_1\leq (\kappa_2-1)/2$ and $$100\alpha_1^{-1}(1/2)^{(\kappa_2-1)/2-\kappa_1}\leq 1$$ imply condition 2 of \cref{quantitative form determines triangulation} is satisfied. Assuming $\kappa_2\geq \kappa_0$, condition 1 of \cref{quantitative form determines triangulation} is satisfied. Applying \cref{quantitative form determines triangulation}, we obtain $$|(H_X)^{-1}(B_k)|\leq (T/g)^{Cg}\sum_{\substack{n\leq (1/100)g\\ T_1+...+T_n\leq T\\ g_1+...+g_n\leq \mu^{-1}(1/100)g\\g_1,...,g_n\geq 2\\ m_1+...+m_n\leq \mu^{-1}(1/100)g}} \prod_{i=1}^n \Ntri(2T_i,g_i,m_i, r+C)$$ for all $k\in \{1,...,(T/g)^{Cg}\}$. The lemma statement follows. 
\end{proof} 

\cref{count cohomology classes} now implies \cref{tranc vs tri}: 

\begin{proof}[Proof of \cref{tranc vs tri}] By \cref{Teichmuller bound 3}, we may cover $B_{d_T}(X,r)$ with $(T/g)^{C\kappa_2(1+r)g}$ many $B_{d_T}(\cdot,(g/T)^{\kappa_2})$ balls. Applying \cref{count cohomology classes} to each ball, then summing over all $T'\leq T$ and $g'\leq g$ gives the desired bound on $\Ntran(T,g,r)$. 
\end{proof}

We now turn to the proof of \cref{quantitative form determines triangulation}, which will take the rest of \cref{bound tran by tri}. Condition 2 in the statement of \cref{quantitative form determines triangulation} can be written as 
\begin{equation}\label{eq60.75}\int_X |f^*\phi_{S_Y}-\phi_{S_X}|^2_{S_X}|\phi_{S_X}|^2\leq \alpha_2(g/T)^{\kappa_4}g.
\end{equation} Here, $\alpha_2=\alpha_1^2$ and $\kappa_4=2\kappa_1$. Also, $|\cdot|_{S_X}$ here denotes the operator norm of a $1$-form at a particular point of $X$ with respect to the $S_X$-metric. 

\begin{remark} Roughly speaking, \cref{eq60.75} gives a measure of how Lipschitz the map $f$ is with respect to the metrics $d_{S_X}$ and $d_{S_Y}$, averaged over the entire surface $X$. However, \cref{eq60.75} also implies that on most of $X$, the integral of $\phi_X$ on a curve is close to the integral of the $1$-form $\phi_Y$ on the image of the curve.
\end{remark}

The first step to prove \cref{quantitative form determines triangulation} is to show that \cref{eq60.75} implies most vertices of $S_X$ and $S_Y$ of degree strictly greater than $6$ must be close to each other under $f$. Then we show that \cref{eq60.75} implies many edges of $S_X$ and $S_Y$ must be close to each other under $f$. Finally we show that many faces of $S_X$ and $S_Y$ must be close to each other under $f$. We do this precisely in \cref{6.1}, \cref{6.2} and \cref{6.3} below. 

The second step to prove \cref{quantitative form determines triangulation} is to show that the geometric conditions about many vertices, edges and faces being close together under $f$ imply that $f$ is close to a simplicial isomorphism on all of $S_X$ except for a part that has much lower genus. To do this, we decompose $S_X$ into around $g$ parallelograms of length and width at most around $T/g$. We use the geometric conditions to say that $f$ is close to a simplicial isomorphism on most of the parallelograms. The remaining parallelograms form a surface of much smaller genus. This allows us to reduce the problem of counting combinatorial translation surfaces in moduli space to the problem of counting triangulated surfaces in a lower dimensional moduli space. We do this in \cref{6.4}, \cref{6.5} and \cref{6.6}. Finally in \cref{6.7} we prove \cref{quantitative form determines triangulation}. 

\subsection{Vertices}\label{6.1}

In this section, we show that most vertices of $S_X$ of degree strictly greater than $6$ are close under $f$ to vertices of $S_Y$ of degree strictly greater than $6$.

Let $0<\epsilon_0<1/2$, to be chosen in \cref{Constants}. Assume that $$8\cdot (1/2)^{\kappa_0}\leq (1/10)^{10}.$$ Let $V_X$ be the set of vertices $x\in V(S_X)$ such that there exists $y\in V(S_Y)$ satisfying $d_{S_Y}(y, f(x))<\epsilon_0\cdot (g/T)$. Note that there is at most one choice of such $y$, since vertices of $S_Y$ are at least distance $1$ apart in the ${S_Y}$-metric.  

\begin{lemma}\label{vertices-x} We have, $|V_{>6}(S_X)\setminus V_X|\leq \alpha_3 g$ where $\alpha_3=10^{10}\alpha_2(g/T)^{\kappa_4-10}\epsilon_0^{-10}$. 
\end{lemma}

Note that $\alpha_3$ is not a universal constant, since it depends on $T/g$. To show \cref{vertices-x}, we first show the following lemma.

\begin{lemma}\label{one vertex} Suppose $x\in V_{>6}(S_X)$ and $$\int_{B_{S_X}(x,1/2)}|f^*\phi_{S_Y}-\phi_{S_X}|^2_{S_X}|\phi_{S_X}|^2\leq \beta.$$ Then there exists $y\in V_{>6}(S_Y)$ satisfying $$d_{S_Y}(y,f(x))<\eta(\beta).$$ Here, $$\eta(\beta)=10\beta^{1/10}<1/2.$$
\end{lemma}

\begin{proof} In the following argument we write $\eta=\eta(\beta)$. Suppose the contrary, that $f(x)$ is not within a $\eta$-neighborhood of any vertex in $V_{>6}(S_Y)$. Let $a=\deg x/6$. Since $S_X$ is a locally bounded combinatorial translation surface, $a\leq 7$. Let $(r_X,\theta_X)$ (where $0\leq \theta_X<2\pi a$) be polar coordinates on $B_{S_X}(x,\eta)$ around $x$ such that $$\phi_{S_X}=e^{i\theta_X}dr_X+i r_Xe^{i\theta_X}d\theta_X.$$ Let $(r_Y,\theta_Y)$ (where $0\leq \theta<2\pi$) be polar coordinates on $B_{S_Y}(f(x),\eta)$ around $y$ such that $$\phi_{S_Y}=e^{i\theta_Y}dr_Y+i r_Ye^{i\theta_Y}d\theta_Y.$$ Let $$C_r=\{(r_X,\theta_X)|r_X=r\}.$$ By assumption, $$\int_{r=\eta/200}^{\eta/100}\int_{C_r}|f^*\phi_{S_Y}-\phi_{S_X}|^2_{S_X}r_Xdr_Xd\theta_X\leq  \beta.$$ Therefore for some $u\in [\eta/200,\eta/100]$, we have $$\int_{C_u}|f^*\phi_{S_Y}-\phi_{S_X}|^2_{S_X}d\theta_X\leq 40000\beta \eta^{-2}.$$ By Cauchy-Schwarz, 
\begin{equation}\label{eq611}
\begin{aligned}\int_{C_u}|f^*\phi_{S_Y}-\phi_{S_X}|_{S_X}d\theta_X&\leq \left(\int_{C_u} |f^*\phi_{S_Y}-\phi_{S_X}|^2_{S_X}d\theta_X\right)^{1/2} (2\pi a \eta)^{1/2}\\&\leq 400\eta^{-1/2}(\pi a\beta)^{1/2}.
\end{aligned}
\end{equation}
Therefore, 
\begin{multline*} \left|\int_{C_u}\left|f^*\phi_{S_Y}\left(\frac{1}{r_X}\frac{\partial}{\partial \theta_X}\right)\right|d\theta_X -\int_{C_u}\left|\phi_{S_X}\left(\frac{1}{r_X}\frac{\partial}{\partial \theta_X}\right)\right| d\theta_X\right|\\
\begin{aligned} &\leq \int_{C_u}\left|f^*\phi_{S_Y}\left(\frac{1}{r_X}\frac{\partial}{\partial \theta_X}\right)-\phi_{S_X}\left(\frac{1}{r_X}\frac{\partial}{\partial \theta_X}\right)\right| d\theta_X\\&\leq \int_{C_u}|f^*\phi_{S_Y}-\phi_{S_X}|_{S_X}d\theta_X\\&\leq 400\eta^{-1/2}(\pi a\beta)^{1/2}.
\end{aligned}
\end{multline*}
Now, $$\textstyle\length_{S_X}(C_u)=\displaystyle u\int_{C_u}\left|\phi_{S_X}\left(\frac{1}{r_X}\frac{\partial}{\partial \theta_X}\right)\right|d\theta_X$$ and $$\textstyle\length_{S_Y}(f(C_u))=\displaystyle u\int_{C_u}\left|f^*\phi_{S_Y}\left(\frac{1}{r_X}\frac{\partial}{\partial \theta_X}\right)\right|d\theta_X.$$ Thus $$2\pi a u - 400\eta^{-1/2}(\pi a\beta)^{1/2}\leq \textstyle\length_{S_Y}(f(C_u))\leq 2\pi a u + 400\eta^{-1/2}(\pi a\beta)^{1/2}.$$ Since $u\in [\eta/200,\eta/100]$, $a\leq 7$ and $\beta\leq \eta^{10}/10^{10}$, $f(C_u)$ is contained in $B_{S_Y}(f(x),\eta)$.

For $v\in [0,2a\pi)$, let $x_v$ be a point on $C_u$ wherein $r_X=u$ and $\theta_X=v$. Letting $C_u(x_0,x_v)$ be an arc of $C_u$ from $x_0$ to $x_v$, we have
\begin{multline*} \left|\left(ue^{iv}-u\right)-\int_{C_u(x_0,x_v)}f^*\phi_{S_Y}\left(\frac{1}{r_X}\frac{\partial}{\partial \theta_X}\right)\right|\\ \begin{aligned} &= \left|\int_{C_u(x_0,x_v)}f^*\phi_{S_Y}\left(\frac{1}{r_X}\frac{\partial}{\partial \theta_X}\right)-\int_{C_u(x_0,x_v)}\phi_{S_X}\left(\frac{1}{r_X}\frac{\partial}{\partial \theta_X}\right)\right|\\&\leq \int_{C_u(x_0,x_v)}\left|f^*\phi_{S_Y}\left(\frac{1}{r_X}\frac{\partial}{\partial \theta_X}\right)-\phi_{S_X}\left(\frac{1}{r_X}\frac{\partial}{\partial \theta_X}\right)\right|\\&\leq \int_{C_u}|f^*\phi_{S_Y}-\phi_{S_X}|_{S_X}d\theta_X\\&\leq 400\eta^{-1/2}(\pi a\beta)^{1/2}.
\end{aligned}
\end{multline*} by \cref{eq611}. Hence, $$|f(x_v)-f(x_0)-(ue^{iv}-u)|_{S_Y}\leq 400\eta^{-1/2}(\pi a\beta)^{1/2}.$$ This means $f(C_u)$ lies in a $400\eta^{-1/2}(\pi a\beta)^{1/2}$-neighborhood of a radius $u$-ball passing through $f(x_0)$ (in the $S_Y$-metric). Therefore, \begin{equation}\label{eq612}\textstyle\Area_{S_Y}(f(B_{S_X}(x,u)))\leq \pi(u+400\eta^{-1/2}(\pi a\beta)^{1/2})^2.
\end{equation}

Finally, by $1+8(g/T)^{\kappa_0}$-quasiconformality of $f$,
\begin{multline*}\label{eq613}\textstyle\Area_{S_X}(B_{S_X}(x,u))^{1/2}-(1+8(g/T)^{\kappa_0})^{1/2}\Area_{S_Y}(f(B_{S_X}(x,u)))^{1/2}\\
 \begin{aligned} &\leq \left(\int_{B_{S_X}(x,u)}|\phi_{S_X}|^2\right)^{1/2}-\left(\int_{B_{S_X}(x,u)}|f^*\phi_{S_Y}|^2_{S_X}|\phi_{S_X}|^2\right)^{1/2} \\ &\leq \left(\int_{B_{S_X}(x,u)}|f^*\phi_{S_Y}-\phi_{S_X}|^2_{S_X}|\phi_{S_X}|^2\right)^{1/2}\\&\leq \beta^{1/2}.
\end{aligned}
\end{multline*} 
Therefore $$\textstyle\Area_{S_Y}(f(B_{S_X}(x,u)))^{1/2}\geq (1-8(g/T)^{\kappa_0})^{1/2}((a \pi)^{1/2} u-\beta^{1/2}).$$ Since $a\geq 2$, $g/T\leq 1/2$, $\beta\leq\eta^{10}/10^{10}$ and $u\in [\eta/200,\eta/100]$, assuming $8\cdot (1/2)^{\kappa_0}\leq (1/10)^{10}$ we have a contradiction with \cref{eq612}. 
\end{proof}

\begin{proof}[Proof of \cref{vertices-x}] Combining \cref{eq60.75} and  \cref{one vertex}, we obtain the desired result.
\end{proof}

\subsection{Edges}\label{6.2}

In this section, we show that many edges of $S_X$ and $S_Y$ are close to each other under $f$. 

\begin{lemma} \label{edges} Let $0<\epsilon\leq 1/1000$. Suppose $x_0\in V(S_X)$ and $f(x_0)\in B_{S_Y}(y_0,\epsilon)$ for a vertex $y_0\in V(S_Y)$. Suppose $e(x_0,x_1)$ is an edge in $S_X$ from $x_0$ to $x_1$. Suppose further that 
\begin{equation}\label{eq621}\int_{B_{S_X}(e(x_0,x_1),1/2)}|f^*\phi_{S_Y}-\phi_{S_X}|^2_{S_X}|\phi_{S_X}|^2\leq \beta,
\end{equation} where $\beta$ satisfies $$\tau(\epsilon,\beta)=100\beta^{1/4}\epsilon^{-1}+\epsilon\leq 1/1000.$$ Then $$f(x_1)\in B_{S_Y}(y_1,\tau(\epsilon,\beta))$$ for some vertex $y_1$ in $S_Y$. Also, $y_1$ is connected to $y_0$ by an edge $e(y_0,y_1)$ such that $$f(e(x_0,x_1))\subset B_{S_Y}(e(y_0,y_1),10\tau(\epsilon,\beta)).$$
\end{lemma}

\begin{proof} Let $T$ be a triangle in $S_X$ containing edge $e(x_0,x_1)$, and let $T'$ be the other triangle containing edge $e(x_0,x_1)$. As shown in \cref{triangle figure}, identify $T$ with the triangle in $\mathbb{C}$ with vertices at $0$, $1$ and $\displaystyle\frac{1}{2}+\displaystyle\frac{\sqrt{3}}{2}i$, with $x_0$ and $x_1$ identified with $0$ and $1$. Then $T'$ is naturally identified with the triangle in $\mathbb{C}$ with vertices at $0$, $1$ and $\displaystyle\frac{1}{2}-\displaystyle\frac{\sqrt{3}}{2}i$. Let $(r_0,\theta_0)$ be polar coordinates on $T\cup T'$ centered at $0$, and $(r_1,\theta_1)$ be polar coordinates on $T\cup T'$ centered at $1$. Note that the form $\phi_{S_X}$ may be written as $$\zeta e^{i\theta_0}d{r_0}+i\zeta r_0e^{i\theta_0}d\theta_0$$ in the $(r_0,\theta_0)$ coordinates and $$-\zeta e^{-i\theta_0}d{r_0}+i\zeta r_0e^{-i\theta_0}d\theta_0$$ in the $(r_1,\theta_1)$-coordinates, for some $6$th root of unity $\zeta$. Let $\upsilon<1/4$ be sufficiently small, to be chosen later in this proof. For $\theta\in [-\pi/3,\pi/3]$, define $$L_{\theta}=\{(r_0,\theta_0)\in T|\theta_0=\theta, \upsilon \leq r_0\leq (1/2)(\cos \theta)^{-1}\},$$ and $$R_\theta=\{(r_1,\theta_1)\in T|\theta_1=\pi-\theta, \upsilon \leq r_1\leq (1/2)(\cos \theta)^{-1}\}.$$ We assume that $L_\theta$ (resp. $R_\theta$) is oriented so that $r_0$ (resp. $r_1$) is increasing. 

\begin{figure}[ht]
\hspace{.35in}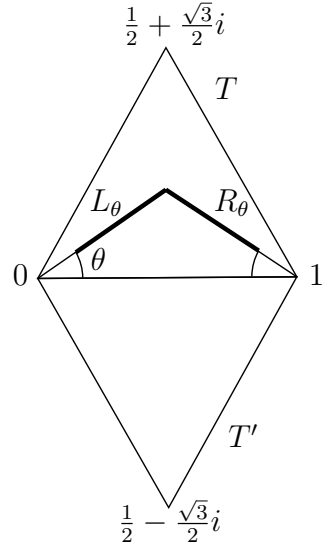
\caption{A diagram of $T$ and $T'$ with $L_\theta$ and $R_\theta$ in bold.}
\label{triangle figure}
\end{figure} 

By elementary trigonometry, we see that for $\theta\in [-\epsilon,\epsilon]$, $L_\theta\cup R_\theta\subset B_{S_X}(e(x_0,x_1),\epsilon)$. By the co-area formula along with \cref{eq621}, $$\int_{\theta=0}^{\epsilon}\int_{L_\theta}r_0|f^*\phi_{S_Y}-\phi_{S_X}|^2_{S_X}dr_0d\theta_0+\int_{\theta=0}^{\epsilon}\int_{R_\theta}r_1|f^*\phi_{S_Y}-\phi_{S_X}|^2_{S_X}dr_1d\theta_1\leq \beta.$$ This means for some $u\in [0,\epsilon]$, $$\int_{L_u}r_0|f^*\phi_{S_Y}-\phi_{S_X}|^2_{S_X}dr_0+\int_{R_u}r_1|f^*\phi_{S_Y}-\phi_{S_X}|^2_{S_X}dr_1\leq \beta\epsilon^{-1}.$$ By Cauchy-Schwarz,
\begin{multline}\label{eq622}
\int_{L_u}|f^*\phi_{S_Y}-\phi_{S_X}|_{S_X}dr_0+\int_{R_u}|f^*\phi_{S_Y}-\phi_{S_X}|_{S_X}dr_1\\  \leq \left(\int_{L_u}r_0|f^*\phi_{S_Y}-\phi_{S_X}|^2_{S_X}dr_0\right)^{1/2}\left(\int_{r_0=\upsilon}^{(1/2)(\cos \theta)^{-1}}r_0^{-1}\right)^{1/2}\\
+\left(\int_{R_u}r_1|f^*\phi_{S_Y}-\phi_{S_X}|^2_{S_X}dr_1\right)^{1/2}\left(\int_{r_1=\upsilon}^{(1/2)(\cos \theta)^{-1}}r_1^{-1}\right)^{1/2}
\end{multline}
thus 
\begin{equation}\label{eq622.5}\int_{L_u}|f^*\phi_{S_Y}-\phi_{S_X}|_{S_X}dr_0+\int_{R_u}|f^*\phi_{S_Y}-\phi_{S_X}|_{S_X}dr_1\leq 4\beta^{1/2}\epsilon^{-1/2}(\log(1/\upsilon))^{1/2}.
\end{equation}
Since $$\left|f_*\phi_{S_Y}\left(\frac{\partial}{\partial r_0}\right)-\phi_{S_X}\left(\frac{\partial}{\partial r_0}\right)\right|\leq |f^*\phi_{S_Y}-\phi_{S_X}|_{S_X}$$ on $L_u$ and $$\left|f_*\phi_{S_Y}\left(\frac{\partial}{\partial r_1}\right)-\phi_{S_X}\left(\frac{\partial}{\partial r_1}\right)\right|\leq |f^*\phi_{S_Y}-\phi_{S_X}|_{S_X}$$ on $R_u$, combining with \cref{eq622.5} we obtain 
\begin{multline}\label{eq623}\int_{L_u}\left|f_*\phi_{S_Y}\left(\frac{\partial}{\partial r_0}\right)-\phi_{S_X}\left(\frac{\partial}{\partial r_0}\right)\right|dr_0+\int_{R_u}\left|f_*\phi_{S_Y}\left(\frac{\partial}{\partial r_1}\right)-\phi_{S_X}\left(\frac{\partial}{\partial r_1}\right)\right|dr_1\\ \leq 4\beta^{1/2}\epsilon^{-1/2}(\log(1/\upsilon))^{1/2}.
\end{multline} For an appropriate $6$th root of unity $\zeta$, $$\phi_{S_Y}\left(\frac{\partial}{\partial r_0}\right)=e^{iu}\zeta$$ on $L_u$ and $$\phi_{S_Y}\left(\frac{\partial}{\partial r_1}\right)=-e^{-iu}\zeta$$ on $R_u$. Plugging into \cref{eq623} we have $$\int_{L_u}\left|f^*\phi_{S_Y}\left(\frac{\partial}{\partial r_0}\right)-e^{iu}\zeta\right|dr_0+\int_{R_u}\left|f^*\phi_{S_Y}\left(\frac{\partial}{\partial r_1}\right)+e^{-iu}\zeta\right|dr_1\leq 4\beta^{1/2}\epsilon^{-1/2}(\log(1/\upsilon))^{1/2}.$$ This means 
\begin{equation}\label{eq624}\left|\int_{L_u\cup R_u}f^*\phi_{S_Y}-\int_{L_u}e^{iu}\zeta dr_0-\int_{R_u}e^{-iu}\zeta dr_1\right|\leq 4\beta^{1/2}\epsilon^{-1/2}(\log(1/\upsilon))^{1/2}.
\end{equation} Note that when integrating over $L_u\cup R_u$, the orientation of $R_u$ switches, hence there is a sign change. By construction of $L_u$ and $R_u$, $$\int_{L_u}e^{iu}\zeta dr_0=((1/2)(\cos u)^{-1}-\upsilon)\zeta e^{iu}$$ and $$\int_{R_u}e^{-iu}\zeta dr_1= ((1/2)(\cos u)^{-1}-\upsilon)\zeta e^{-iu},$$ hence combining with \cref{eq624} we obtain 
\begin{equation}\label{eq625}\left|\zeta-\int_{L_u\cup R_u}f^*\phi_{S_Y}\right|\leq 4\beta^{1/2}\epsilon^{-1/2}(\log(1/\upsilon))^{1/2}+2\upsilon.
\end{equation} Pushing forward to $Y$ we have 
\begin{equation}\label{eq626}\left|\zeta-\int_{f(L_u\cup R_u)}\phi_{S_Y}\right|\leq 4\beta^{1/2}\epsilon^{-1/2}(\log(1/\upsilon))^{1/2}+2\upsilon.
\end{equation} From \cref{eq625} we also have $$\int_{L_u}\left|f^*\phi_{S_Y}\left(\frac{\partial}{\partial r_0}\right)\right|dr_0+\int_{R_u}\left|f^*\phi_{S_Y}\left(\frac{\partial}{\partial r_1}\right)\right|dr_1\leq 1 + 4\beta^{1/2}\epsilon^{-1/2}(\log(1/\upsilon))^{1/2}+2\upsilon.$$ The left-hand-side is the length of $f(L_u\cup R_u)$ in the $S_Y$-metric. Hence, 
\begin{equation}\label{eq626.5}\textstyle\length_{S_Y}(f(L_u\cup R_u))\leq 1+4\beta^{1/2}\epsilon^{-1/2}(\log(1/\upsilon))^{1/2}+2\upsilon.
\end{equation} 

Now, let $C_{0,r}$ be the circle of radius $r$ around $x_0$, with respect to the $S_X$-metric. Similarly, let $C_{1,r}$ be the circle of radius $r$ around $x_1$, with respect to the $S_X$-metric. By assumption, for $\upsilon<1/4$ we have 
$$\int_{r=\upsilon}^{2\upsilon}\int_{C_{0,r}}|f^*\phi_{S_Y}-\phi_{S_X}|^2_{S_X}r_0dr_0d\theta_0\leq \beta.$$ Therefore for some $w\in [\upsilon,2\upsilon]$, we have $$\int_{C_{0,w}}|f^*\phi_{S_Y}-\phi_{S_X}|^2_{S_X}d\theta_0\leq \beta \upsilon^{-2}.$$ By Cauchy-Schwarz, 
\begin{equation*}
\begin{aligned}\int_{C_{0,w}}|f^*\phi_{S_Y}-\phi_{S_X}|_{S_X}d\theta_0&\leq \left(\int_{C_{0,w}} |f^*\phi_{S_Y}-\phi_{S_X}|^2_{S_X}d\theta_0\right)^{1/2} (2\pi \upsilon)^{1/2}\\&\leq (2\pi)^{1/2}\upsilon^{-1/2} \beta^{1/2}.
\end{aligned}
\end{equation*}
Moreover, \begin{multline*} \left|\int_{C_{0,w}}\left|f^*\phi_{S_Y}\left(\frac{1}{r_0}\frac{\partial}{\partial \theta_0}\right)\right|d\theta_0 -\int_{C_{0,w}}\left|\phi_{S_X}\left(\frac{1}{r_0}\frac{\partial}{\partial \theta_0}\right)\right| d\theta_0\right|\\
\begin{aligned} &\leq \int_{C_{0,w}}\left|f^*\phi_{S_Y}\left(\frac{1}{r_0}\frac{\partial}{\partial \theta_0}\right)-\phi_{S_X}\left(\frac{1}{r_0}\frac{\partial}{\partial \theta_0}\right)\right| d\theta_0\\&\leq \int_{C_{0,w}}|f^*\phi_{S_Y}-\phi_{S_X}|_{S_X}d\theta_0\\&\leq (2\pi)^{1/2}\upsilon^{-1/2} \beta^{1/2}.
\end{aligned}
\end{multline*}
Now, $$\textstyle\length_{S_X}(C_{0,w})=\displaystyle w\int_{C_{0,w}}\left|\phi_{S_X}\left(\frac{1}{r_0}\frac{\partial}{\partial \theta_0}\right)\right|d\theta_0$$ and $$\textstyle\length_{S_Y}(f(C_{0,w}))=\displaystyle w\int_{C_{0,w}}\left|f^*\phi_{S_Y}\left(\frac{1}{r_0}\frac{\partial}{\partial \theta_0}\right)\right|d\theta_0.$$  Thus $$2\pi w - (2\pi)^{1/2}\upsilon^{-1/2} \beta^{1/2}\leq \textstyle\length_{S_Y}(f(C_{0,w}))\leq 2\pi w + (2\pi)^{1/2}\upsilon^{-1/2} \beta^{1/2}.$$ Since $w\in [\upsilon,2\upsilon]$ and assuming $\beta\leq \upsilon^4$, $f(C_{0,w})$ is contained in $B_{S_Y}(f(x_1),20\upsilon)$. Since $f(C_{0,\upsilon})$ is enclosed by $f(C_{0,w})$, $f(C_{0,\upsilon})$ is also contained in $B_{S_Y}(f(x_0),20\upsilon)$. Analogously, we may show that $f(C_{1,\upsilon})$ is contained in $B_{S_Y}(f(x_1),20\upsilon)$.

The arc $L_u\cup R_u$ has two endpoints, one wherein $r_0=\upsilon$ and $\theta_0=u$, which we label $x_{0,u}$, and the other wherein $r_1=\upsilon$ and $\theta_1=\pi-u$, which we label $x_{1,u}$. Since $f(x_{0,u})$ lies on $f(C_{0,\upsilon})$ and $f(x_{1,u})$ lies on $f(C_{1,\upsilon})$, $$f(x_{0,u})\in B_{S_Y}(f(x_0,20\upsilon))$$ and $$f(x_{1,u})\in B_{S_Y}(f(x_1,20\upsilon)).$$ By \cref{eq626} we have, $$d_{S_Y}(f(x_0),f(x_1))\leq 1+4\beta^{1/2}\epsilon^{-1/2}(\log(1/\upsilon))^{1/2}+ 50\upsilon,$$ so 
\begin{align*}d_{S_Y}(y_0,f(x_1))&\leq d_{S_Y}(y_0,f(x_0))+d_{S_Y}(f(x_0),f(x_1))\\&\leq 1+4\beta^{1/2}\epsilon^{-1/2}(\log(1/\upsilon))^{1/2}+ 50\upsilon+\epsilon.
\end{align*} 
Let $y_1$ be a vertex of $S_Y$ nearest to $f(x_1)$. Since $f(x_1)$ lies on a triangle in $S_Y$, 
\begin{equation}\label{eq627}d_{S_Y}(y_1,f(x_1))\leq {\sqrt{3}}^{-1}.
\end{equation} Therefore 
\begin{align*}d_{S_Y}(y_0,y_1)&\leq d_{S_Y}(y_0,f(x_1))+d_{S_Y}(y_0,f(x_1))\\&\leq 1+4\beta^{1/2}\epsilon^{-1/2}(\log(1/\upsilon))^{1/2}+ 50\upsilon+\epsilon+{\sqrt{3}}^{-1}.
\end{align*}
Choose $\upsilon$ later such that $$4\beta^{1/2}\epsilon^{-1/2}(\log(1/\upsilon))^{1/2}+50\upsilon+\epsilon<1/1000.$$ We must have that either $$d_{S_Y}(y_0,y_1)=0$$ or $$d_{S_Y}(y_0,y_1)=1.$$ The former possibility is ruled out by \cref{eq626}, therefore the latter equation is true. This necessarily means that $y_0$ and $y_1$ are adjacent vertices on $S_Y$.

We must show that $$d(x_1,f(y_1))\leq \tau(\epsilon,\beta)$$ and also find an edge of $S_Y$ connecting $y_0$ and $y_1$. To do this, letting $q$ be a shortest path from $f(x_1)$ to $y_1$, we have \begin{equation}\label{eq628}
\begin{aligned}d_{S_Y}(f(x_1),y_1)&=\left|\int_q \phi_Y\right|\\&\leq {\sqrt{3}}^{-1}
\end{aligned}
\end{equation} by \cref{eq627}. Let $p$ be a shortest path from $y_0$ to $f(x_0)$ in $S_Y$. Then 
\begin{equation}\label{eq629}\left|\int_{p}\phi_{S_Y}\right|\leq\epsilon.
\end{equation} Let $s_0$ be a shortest path from $f(x_{0,u})$ to $f(x_0)$ and $s_1$ be a shortest path from $f(x_{1,u})$ to $f(x_1)$. Finally, let $t$ be a shortest path from $y_0$ to $y_1$ that is homotopic to the path $$p\cup s_0 \cup f(L_u\cup R_u)\cup s_1 \cup q$$ (which also has the endpoints $y_0$ and $y_1$). By construction $q$ is homotopic to $$t^{-1}\cup p\cup s_0\cup f(L_u\cup R_u)\cup s_1,$$ where here $t^{-1}$ denotes the path $t$ with the opposite orientation. So to compute $$\int_q \phi_Y,$$ it suffices to compute $$\int_{t^{-1}\cup p\cup s_0\cup f(L_u\cup R_u)\cup s_1} \phi_Y=\int_{p} \phi_Y+\int_{s_0}\phi_Y +\int_{f(L_u\cup R_u)} \phi_Y+\int_{t} \phi_Y+\int_{s_1}\phi_Y.$$ Since $s_0\subset B(f(x_1),20\upsilon)$, 
\begin{equation}\label{eq629.1}\left|\int_{s_0}\phi_Y\right|\leq 20\upsilon.
\end{equation} Similarly, \begin{equation}\label{eq629.2}\left|\int_{s_1}\phi_Y\right|\leq 20\upsilon.
\end{equation} From \cref{eq626}, \cref{eq628}, \cref{eq629}, \cref{eq629.1} and \cref{eq629.2}, we have 
\begin{equation*}\label{eq6210}\left|\zeta-\int_{t} \phi_Y\right|\leq 4\beta^{1/2}\epsilon^{-1/2}(\log(1/\upsilon))^{1/2}+50\upsilon+\epsilon+\sqrt{3}^{-1}.
\end{equation*} Since $$\int_t \phi_Y\subset \mathbb{Z}+e^{\pi i/3}\mathbb{Z},$$ we must necessarily have $$\int_t\phi_Y=\zeta.$$ Combining with \cref{eq626}, \cref{eq629}, \cref{eq629.1} and \cref{eq629.2} we obtain 
\begin{align*}d(y_1,f(x_1))&=\left|\int_q \phi_Y\right|\\&\leq \left|\int_{t^{-1}\cup p\cup s_0\cup f(L_u\cup R_u)\cup s_1} \phi_Y\right|\\&\leq 4\beta^{1/2}\epsilon^{-1/2}(\log(1/\upsilon))^{1/2}+50\upsilon+\epsilon,
\end{align*}
as desired. Finally, since $t$ is a shortest path from $y_0$ to $y_1$ and $$\int_t\phi_Y=\zeta,$$ $t$ must be an edge which we denote $e(y_0,y_1)$. 

It remains to show that $$f(e(x_0,x_1))\subset B_{S_Y}(e(y_0,y_1),10\tau(\epsilon,\beta)).$$ To do this, we have shown by \cref{eq626.5} the existence of $u\in [0,\epsilon]$ such that $$\textstyle\length_{S_Y}(f(L_u\cup R_u))\leq 1+4\beta^{1/2}\epsilon^{-1/2}(\log(1/\upsilon))^{1/2}+2\upsilon.$$ Analogously, there exists $u'\in [-\epsilon,0]$ such that $$\textstyle\length_{S_Y}(f(L_{u'}\cup R_{u'}))\leq 1+4\beta^{1/2}\epsilon^{-1/2}(\log(1/\upsilon))^{1/2}+2\upsilon.$$ Now, the endpoints of $f(L_u\cup R_u)$ and $f(L_{u'}\cup R_{u'})$ are contained in a $20\upsilon$-ball around $f(x_0)$ or $f(x_1)$ in the $S_Y$-metric. Also, $$f(x_0)\in B_{S_Y}(y_0,\epsilon)$$ while $$f(x_1)\in B_{S_Y}(y_1,4\beta^{1/2}\epsilon^{-1/2}(\log(1/\upsilon))^{1/2}+50\upsilon+\epsilon).$$ Thus both $f(L_u\cup R_u)$ and $f(L_{u'}\cup R_{u'})$ must be contained in $$B_{S_Y}(e(y_0,y_1),40\beta^{1/2}\epsilon^{-1/2}(\log(1/\upsilon))^{1/2}+500\upsilon+10\epsilon).$$ The $L_u\cup R_u$ and $L_{u'}\cup R_{u'}$, along with arcs of $C_{0,\upsilon}$ and $C_{1,\upsilon}$ bound a closed simply connected region $R$ in $S_X$ such that $$e(x_0,x_1)\subset B_{S_X}(x_0,\upsilon)\cup R\cup B_{S_X}(x_1,\upsilon).$$ Hence,
\begin{align*} f(e(x_0,x_1))&\subset f(B_{S_X}(x_0,\upsilon))\cup f(R)\cup f(B_{S_X}(x_1,\upsilon))\\&\subset B_{S_Y}(f(x_0),20 \upsilon)\cup f(R)\cup B_{S_Y}(f(x_1),20\upsilon)).
\end{align*} Therefore, $$f(e(x_0,x_1))\subset B_{S_Y}(e(y_0,y_1),40\beta^{1/2}\epsilon^{-1/2}(\log(1/\upsilon))^{1/2}+500\upsilon+10\epsilon)$$ as well. Finally, choosing $$\upsilon=\beta^{1/4}$$ satisfies $$\beta\leq \upsilon^4,$$
\begin{align*}4\beta^{1/2}\epsilon^{-1/2}(\log(1/\upsilon))^{1/2}+50\upsilon+\epsilon&\leq 100\beta^{1/4}\epsilon^{-1}+\epsilon\\&\leq 1/1000,
\end{align*} and $$\upsilon\leq 1/4,$$ which we assumed in the proof. 
\end{proof}

\subsection{Faces}\label{6.3} In this section, we note that if vertices and edges of $S_X$ are close to vertices and edges of $S_Y$ under $f$, then so are the faces they bound.

\begin{lemma} \label{faces} Let $0<\epsilon\leq 1/4$. Suppose $x_1,x_2,x_3\in V(S_X)$ and $y_1,y_2,y_3\in V(S_Y)$ such that $$f(x_1)\in B_{S_Y}(y_1,\epsilon),$$ $$f(x_2)\in B_{S_Y}(y_2,\epsilon),$$ and $$f(x_3)\in B_{S_Y}(y_3,\epsilon).$$ Suppose $x_1,x_2,x_3$ are the vertices of a triangle in $S_X$. Let $e(x_i,x_j)$ be the edge of the triangle that connects vertices $x_i$ and $x_j$. Suppose additionally that for each $i,j\in \{1,2,3\}$ satisfying $i\neq j$, there is an edge $e(y_i,y_j)$ in $S_Y$ such that $$f(e(x_i,x_j))\subset B_{S_Y}(e(y_i,y_j),\epsilon).$$ Then the $y_1,y_2,y_3$ along with the $e(y_i,y_j)$ bound a triangle in $S_Y$.
\end{lemma}

\begin{proof} 
Consider the space $$P=B_{S_Y}(e(y_1,y_2),\epsilon)\cup B_{S_Y}(e(y_2,y_3),\epsilon)\cup B_{S_Y}(e(y_3,y_1),\epsilon).$$ Note that $P$ is homotopy equivalent to $S^1$. In particular, $\pi_1(P)\simeq \mathbb{Z}$. The groupoid version of Van Kampen's theorem, along with our hypotheses, implies that $$f(e(x_1,x_2))\cup f(e(x_2,x_3))\cup f(e(x_3,x_1))$$ is a generator of $\pi_1(P)$. Note that $$e(y_1,y_2)\cup e(y_2,y_3)\cup e(y_3,y_1)$$ is also a generator of $\pi_1(P)$. So $$f(e(x_1,x_2))\cup f(e(x_2,x_3))\cup f(e(x_3,x_1))$$ and $$e(y_1,y_2)\cup e(y_2,y_3)\cup e(y_3,y_1)$$ are freely homotopic in $P$. Since $P\subset S_Y$, they are freely homotopic in $S_Y$ as well. The former loop is contractible in $S_Y$ since $$e(x_1,x_2)\cup e(x_2,x_3)\cup e(x_3,x_1)$$ is contractible in $S_X$. Therefore $$e(y_1,y_2)\cup e(y_2,y_3)\cup e(y_3,y_1)$$ is contractible in $S_Y$. Gauss-Bonnet along with the fact that $S_Y$ is a combinatorial translation surface implies that $$e(y_1,y_2)\cup e(y_2,y_3)\cup e(y_3,y_1)$$ bounds a triangle in $S_Y$ with vertices $y_1$, $y_2$ and $y_3$.
\end{proof}

\subsection{Parallelogram decomposition} \label{6.4}

In this section, given $S\in \Tranc{g}{T}$, we decompose $S$ into parallelograms wherein each parallelogram contains a vertex of degree greater than $6$. This decomposition applied to $S_X$ will allow us to use \cref{vertices-x} and \cref{edges} repeatedly to show (in \cref{6.5}) that most edges of $S_X$ are close to edges of $S_Y$ under $f$. 

The translation surface structure of $S$ determines foliations (with singularities) on $S$ whose leaves are constant trajectories for $\phi_S$. In particular, the horizontal foliation is the foliation obtained in this way satisfying $\phi_S=\pm 1$ on each leaf. The skew vertical foliation is the foliation obtained in this way satisfying $\phi_S=\pm e^{\pi i/3}$ on each leaf. Both these foliations have singularities at the zeros of $\phi_S$, but are otherwise nonsingular.

\begin{lemma} All leaves of the horizontal and skew vertical foliations of $S$ are closed.
\end{lemma}

\begin{proof} Let $$T=\mathbb{C}/(\mathbb{Z}+e^{\pi i/3}\mathbb{Z}).$$ The combinatorial translation structure of $S$ is equivalent to a branched cover $S\to T$ branched only over $0$. The form $\phi_S$ is simply the pullback of $dz$ on $T$ under the branched covering map. Since the leaves of the horizontal and skew vertical foliations of $T$ are closed, so are the leaves of horizontal and skew vertical foliations of $S$.
\end{proof}

The singular leaves of the horizontal foliation are the leaves that travel through a vertex of degree strictly greater than $6$. Label them $L_1,...,L_k$. Note that each $L_i$ is the union of some edges of the triangulation $S$.

\begin{lemma} The surface $S\setminus \bigcup_{i=1}^k L_i$ is a disjoint union of annuli. Moreover, the $S$-metric restricted to each annulus is flat (with no singularities).
\end{lemma}

\begin{proof} Let $A$ be a connected component of $S\setminus \bigcup_{i=1}^k L_i$. The geodesic curvature of a boundary component of $A$ with respect to the $S$-metric on $A$ is $0$. The curvature at any point in the interior of $A$ is also $0$. By Gauss-Bonnet, $\chi(A)=0$. Therefore, $A$ is a torus or an annulus. However, by construction, $A$ must have at least one boundary component, so $A$ cannot be a torus. The lemma follows.
\end{proof}

Denote by $\mathcal{A}$ the set of annular components of $S\setminus \bigcup_{i=1}^k L_i$. Let $A\in \mathcal{A}$.

\begin{lemma}\label{annulus contains singularity} Each component of $\partial A$ contains at least one vertex in $V_{>6}(S)$.
\end{lemma}

\begin{proof} Suppose the contrary. Let $L$ be a boundary component of $A$ which does not contain any vertices in $V_{>6}(S)$. Then $L$ is a closed leaf of the horizontal foliation on $S$. But $L$ does not contain any singularities, which is a contradiction.
\end{proof}

We now further decompose $S$ into parallelograms. Let $A\in \mathcal{A}$. Consider the skew vertical foliation on $S$ restricted to $A$. Let $M_{A,1},...,M_{A,j_A}$ be the leaves of the restricted foliation that pass through $\partial A\cap V_{>6}(S)$. Note that the $M_{i,A}$ are the union of some edges of triangulation $S$. By construction and \cref{annulus contains singularity}:

\begin{lemma} \label{R is a quadrilateral} Each component of $A\setminus \bigcup_{i=1}^{j_A} M_{A,i}$ is a parallelogram. 
\end{lemma}

Such a parallelogram has four corner vertices and four edges. Two parallel edges are segments of leaves of the horizontal foliation. We call these horizontal edges. The other two parallel edges are segments of leaves of the skew vertical foliation. We call these skew vertical edges. By construction, we have: 

\begin{lemma}\label{R_i contains vertex of degree greater than 6} Let $R$ be parallelogram in $A\setminus \bigcup_{i=1}^{j_A} M_{A,i}$. Each skew vertical edge of a parallelogram contains a vertex in $V_{>6}(S)$ as one of its endpoints. In particular, $R$ contains at least two corner vertices in $V_{>6}(S)$.
\end{lemma}

We have decomposed $S$ into annuli, and further decomposed each annulus into parallelograms. The union of the boundaries of these parallelograms is naturally a $1$-complex on $S$ which we will denote $B_1$.

That is, $$B_1=\bigcup_{i=1}^k L_i\cup \bigcup_{A\in \mathcal{A}}\bigcup_{i=1}^{j_A} M_{A,i}.$$

The $1$-complex $B_1$ naturally decomposes into a horizontal component $$B_1^{\hor}=\bigcup_{i=1}^k L_i$$ and a skew vertical component $$B_1^{\skv}=\bigcup_{A\in \mathcal{A}}\bigcup_{i=1}^{j_A} M_{A,i}.$$ By construction:

\begin{lemma}\label{hor} $B_1^{\hor}$ is the union of some leaves of the horizontal foliation.
\end{lemma}

Let $B$ be the $2$-polytope homeomorphic to $S$ whose $1$-skeleton is $B_1$. That is, the vertices of $B$ (denoted $V(B)$) are the intersection points $B_1^{\hor}\cap B_1^{\skv}$. (Note that this set contains $V_{>6}(S)$.) Away from these points, $B_1$ is locally a leaf of either the horizontal or skew vertical foliation. The edges of $B$ (denoted $E(B)$) are the connected components of $B_1\setminus V(B)$. Each such edge is contained in either $B_1^{\hor}$ or $B^{\skv}$. The faces of $B$ (denoted $F(B)$) are the connected components of $S\setminus B_1$. By construction, $B$ is a polytope. Note that the faces of $B$ are geometrically parallelograms, but need not be quadrilaterals as faces of a polytope. 

\begin{lemma} Let $\gamma$ be an edge of the polytope $B$. Suppose $\gamma\subset B_1^{\skv}$. Then $\gamma$ is the edge of a rectangle $R\in F(B)$. 
\end{lemma}

\begin{proof}\label{skv} Since $\gamma\in E(B)$, $\gamma$ must not intersect $B_1^{\hor}$ except at its endpoints. Therefore, $\gamma$ is contained in some annulus $A$ in the annular decomposition. By construction, the lemma follows.
\end{proof}

We have some additional combinatorial properties of $B$. Recall the construction of directional weights from \cref{def of comb tran surface}.

\begin{lemma}\label{property of B_X} The polytope $B$ satisfies the following properties.

\begin{enumerate} 

\item If $e\in E(S)$ is contained in $B_1$ and $x\in V(S)$ is a vertex which is an endpoint of $e$, then $\zeta(e,x)=1$, $-1$, $e^{\pi i/3}$ or $-e^{\pi i /3}$. 

\item If $x\in V_{>6}(S)$ and $e\in E(S)$ is an edge emanating from $x$ satisfying $\zeta(e,x)=1$, $-1$, $e^{\pi i/3}$ or $-e^{\pi i/3}$, then $e\subset B_1$. 

\item If $x\in V(B)$ and $e\in E(S)$ is an edge emanating from $x$ satisfying $\zeta(e,x)=1$ or $-1$, then $e\subset B_1$. 

\item Given two vertices $x,y\in V(B)$ and a path $\gamma$ from $x$ to $y$ on $S$, $$\int_\gamma \phi_{S}\in 5\mathbb{Z}+5e^{\pi i/3}\mathbb{Z}.$$ 

\end{enumerate}
\end{lemma}

\begin{proof}

Property 1 is true since each edge of $B$ is a segment of a leaf of the horizontal or skew vertical foliation. Property 2 is true by construction. 

To see property 3: if $x\in V(B)$, then $x\in B_1^{\hor}$, so property 3 follows from \cref{hor}. 

Finally, we show property 4. By condition 2 of \cref{lb comb tran surface}, it suffices to show that for all $x\in V(B)$, there exists $y\in V_{>6}(S)$ and a path $\gamma$ from $x$ to $y$, such that $$\int_\gamma \phi_S\in 5\mathbb{Z}+5e^{\pi i/3}\mathbb{Z}.$$ Suppose $x\in V(B)$. Then there is a segment in $B_1^{\skv}$ emanating from $x$, which by \cref{skv} is a skew vertical edge $\gamma$ of a parallelogram $R$. By \cref{R_i contains vertex of degree greater than 6}, the other endpoint $y$ of $\gamma$ lies in $V_{>6}(S)$. The parallelogram $R$ is contained in an annulus $A\in \mathcal{A}$. The annulus $A$ has two boundary components, one of which contains $x$ and the other $y$. The boundary component which contains $x$ must contain another vertex $x'\in V_{>6}(S)$, by \cref{annulus contains singularity}. Let $p$ be a path from $x$ to $x'$ in $A$ along the boundary component of $A$. Since $p$ lies on a boundary component of $A$, it lies in $B_1^{\hor}$, so $$\int_p \phi_S\in \mathbb{Z}.$$ Since $\gamma$ is a skew vertical edge of $R$, it lies in $B_1^{\skv}$, so $$\int_\gamma \phi_S\in e^{\pi i/3}\mathbb{Z}.$$ However, $$\int_{p\cup\gamma} \phi_S=\int_p \phi_S+\int_\gamma \phi_S\in 5\mathbb{Z}+5e^{\pi i/3}\mathbb{Z}.$$ Therefore$$\int_{\gamma}\phi_S\in 5\mathbb{Z}+5e^{\pi i/3}\mathbb{Z}$$ which completes the proof of property 4.
\end{proof}

Given a parallelogram $R\in F(B)$, let $\ell(R)$ denote the length of its horizontal edges. Let $w(R)$ denote the length of its skew vertical edges. Property 4 of \cref{property of B_X} implies:

\begin{lemma} \label{length and width of R_i} We have, $\ell(R),w(R)\geq 5$.
\end{lemma}

Finally, we bound the number of parallelograms in our decomposition. 

\begin{lemma}\label{number of parallelograms} We have, $|F(B)|\leq 12(g-1)$.
\end{lemma}

\begin{proof}[Proof of \cref{number of parallelograms}] Since $S$ is a combinatorial translation surface, all vertices of $S$ have degree a multiple of $6$, so 
\begin{equation}
\begin{aligned} \sum_{x\in V_{>6}(S)}\deg x&\leq -12|V(S)|+4|E(S)|\\&=-12\chi(S)\\ &\leq 24(g-1).
\end{aligned}
\end{equation}
By \cref{R_i contains vertex of degree greater than 6}, the lemma follows.
\end{proof}

\subsection{Lower genus triangulated surface from parallelogram decomposition}\label{6.5}

In \cref{6.4}, we constructed a parallelogram decomposition for locally bounded combinatorial translation surfaces. In this section, we apply this decomposition to $S_X$. We show that $f$ is close to a simplicial isomorphism on most of the parallelograms. As a consequence, we show that the part of $S_X$ on which $f$ is not close to a simplicial isomorphism is a lower genus surface.

Decompose $S_X$ into parallelograms as in \cref{6.4} and let $B_X$ denote the corresponding $2$-polytope. Enumerate the parallelograms in $F(B_X)$ by $R_1,...,R_N$. Recall that $V_X$ is the set of vertices $x\in V(S_X)$ such that there exists $y\in V(S_Y)$ satisfying $d_{S_Y}(y, f(x))<\epsilon_0\cdot (g/T)$. Let $E_0$ be the set of $R_i$ which contain a vertex in $V_X$.

\begin{lemma}\label{counting 1} We have, $N-|E_0|\leq 42 \alpha_3 g$. 
\end{lemma}

\begin{proof} By \cref{vertices-x} we have that $$|V_{>6}(S_X)\setminus V_X|\leq \alpha_3 g.$$ Now, if $R_i\notin E_0$, then by \cref{R_i contains vertex of degree greater than 6}, $R_i$ contains a vertex in $V_{>6}(S_X)\setminus V_X$, which means $R_i$ also contains a triangle with a vertex belonging to $V_{>6}(S_X)\setminus V_X$. Since $S_X$ is locally bounded, the number of such triangles is at most $42 \alpha_3 g$. Since the interiors of the $R_i$ are all disjoint, at most $42 \alpha_3 g$ number of $R_i$ can contain such a triangle. The lemma follows.
\end{proof}

Let $\delta\in \mathbb{N}$, to be chosen later. Let $E_1\subset E_0$ be the subset of $R_i$ in $E_0$ which satisfy the additional property that $$\ell(R_i),w(R_i)\leq \delta(T/g).$$ 

\begin{lemma}\label{counting 2} We have, $|E_0|-|E_1|\leq \delta^{-1} g/2$.
\end{lemma}

\begin{proof} If either $\ell(R_i)$ or $w(R_i)$ is greater than $\delta(T/g)$, then $$\textstyle\Area_{S_X}(R_i)\geq (\sqrt{3}/2)\delta(T/g).$$ Since $$\textstyle\Area_{S_X}(S_X)=(\sqrt{3}/4)T,$$ there can be at most $\delta^{-1} g/2$ such $R_i$.
\end{proof}

Let $0<\epsilon_1<1/1000$ also satisfy $\epsilon_0\leq \epsilon_1\delta^{-1}/40$, to be chosen in \cref{Constants}. Let $E_2\subset E_1$ be the subset of $R_i$ in $E_1$ which satisfies 
\begin{equation*}\label{eq640.5}\int_{B_{S_X}(R_i,3+\epsilon_1)}|f^*\phi_{S_Y}-\phi_{S_X}|^2|\phi_{S_X}|^2\leq (\epsilon_1\delta^{-1}/80)^8(g/T)^{36}.
\end{equation*}

\begin{lemma}\label{properties of rectangles} If $R_i\in E_2$, then $R_i$ satisfies the following two properties: 
\begin{enumerate}

\item If $$x\in V(S_X)\cap \overline{B_{S_X}(R_i,3)},$$ then there exists a unique vertex $y\in V(S_Y)$ such that $$f(x)\in B_{S_Y}(y,\epsilon_1).$$ This vertex will henceforth be denoted $\mathfrak{f}(x)$.

\item Let $$x_1,x_2\in V(S_X)\cap \overline{B_{S_X}(R_i,3)}$$ be vertices and $$e(x_1,x_2)\in E(S_X)\cap R_i$$ an edge connecting $x_1$ and $x_2$. Then there exists a unique edge $$e(\mathfrak{f}(x_1),\mathfrak{f}(x_2))\in E(S_Y)$$ connecting the vertices $\mathfrak{f}(x_1)$ and $\mathfrak{f}(x_2)$ and satisfying, $$f(e(x_1,x_2))\subset B_{S_Y}(\mathfrak{f}(e(x_1,x_2)),\epsilon_1).$$ This unique edge will henceforth be denoted $\mathfrak{f}(e(x_1,x_2))$.
\end{enumerate}
\end{lemma}

\begin{remark} The uniqueness condition in property 1 automatically follows from the existence condition in property 1 since any two distinct vertices in $S_Y$ are at least distance $1$ apart in the $S_Y$-metric. Because $S_Y$ is a locally bounded combinatorial translation surface, two vertices of $S_Y$ cannot have two distinct edges between them. Thus the uniqueness condition in property 2 automatically follows from property 1 and the existence condition in property 2.
\end{remark}

\begin{proof} Suppose $R_i\in E_2$ and does not satisfy one of the properties in the lemma statement. Let $$\xi=\epsilon_1\delta^{-1}/80.$$ Since $\delta\geq 1$ and $g/T\leq 1/2$ by construction, 
\begin{equation}\label{eq641}(g/T)(8+2\delta(T/g))\xi\leq \epsilon_1/10.
\end{equation} Define $$\xi_k=(g/T)(1+k)\xi$$ and $$\beta=\xi^{8}(g/T)^{36}.$$ It is easy to check that 
\begin{equation}\label{eq642}\tau(\xi_k,\beta)\leq \xi_{k+1},
\end{equation} where $\tau$ is as defined in \cref{edges}. Now, since $R_i\in E_1$, there exists a vertex $x\in V_X\cap \overline{R_i}$. For all $$x'\in V(S_X)\cap \overline{B_{S_X}(R_i,3)},$$ denote by $$d_G(x',x)$$ the graph distance between $x$ and $x'$ with respect to the $1$-skeleton of $S_X$ restricted to $\overline{B_{S_X}(R_i,3)}$. Note that for all $$x'\in V(S_X)\cap \overline{B_{S_X}(R_i,3)},$$ 
\begin{equation}\label{eq643} d_G(x,x')\leq 2\delta(T/g)+6 
\end{equation} since $R_i\in E_0$. \cref{eq641} and \cref{eq643} imply that for all $$x'\in (S_X)\cap \overline{B_{S_X}(R_i,3)},$$
\begin{equation}\label{eq643.5}\xi_{d_{G}(x,x')+1}\leq \epsilon_1/10.
\end{equation} 

We will show a stronger condition than the lemma statement. Namely, we will show that: 

\begin{lemma}\label{properties of rectangles strong} \begin{enumerate}\item If $$x_1\in V(S_X)\cap \overline{B_{S_X}(R_i,3)},$$ then there exists a unique vertex $\mathfrak{f}(x_1)\in V(S_Y)$ such that $$f(x_1)\in B_{S_Y}(\mathfrak{f}(x_1),\xi_{d_G(x,x_1)+1}).$$

\item Let $$x_1,x_2\in V(S_X)\cap \overline{B_{S_X}(R_i,3)}$$ be vertices and $$e(x_1,x_2)\in E(S_X)\cap R_i$$ an edge connecting $x_1$ and $x_2$. Then there exists a unique edge $\mathfrak{f}(e(x_1,x_2))\in E(S_Y)$ connecting the vertices $\mathfrak{f}(x_1)$ and $\mathfrak{f}(x_2)$ and satisfying, $$f(e(x_1,x_2))\subset B_{S_Y}(\mathfrak{f}(e(x_1,x_2)),10\xi_{d_G(x,x_1)+1}).$$ 

\end{enumerate}
\end{lemma}

By \cref{eq643.5}, \cref{properties of rectangles strong} implies \cref{properties of rectangles}. So it suffices to prove \cref{properties of rectangles strong}.

\begin{proof}[Proof of \cref{properties of rectangles strong}] First, because of our assumption that $$\epsilon_0\leq \epsilon_1\delta^{-1}/40,$$ we have $$\epsilon_0\cdot (g/T)\leq \xi_1.$$ Since $x\in V_X$, there exists vertex $\mathfrak{f}(x)\in V(S_Y)$ satisfying $$f(x)\in B_{S_Y}(\mathfrak{f}(x),\epsilon_0\cdot (g/T))\subset B_{S_Y}(\mathfrak{f}(x),\xi_1).$$ 

Now, assume that the lemma statement is false. Then there is a vertex closest to $x$ with respect to the graph distance at which either property 1 or property 2 fails. More precisely, there exist adjacent $$x_1,x_2\in V(S_X)\cap \overline{B_{S_X}(R_i,3)},$$ with edge $e(x_1,x_2)$ between them in $R_i$, with a unique vertex $y_1\in V(S_Y)$ such that $$f(x_1)\in B_{S_Y}(y_1,\xi_{d_G(x_1,x)})$$ and satisfying: 
\begin{enumerate}

\item there does not exist a vertex $y_2\in V(S_Y)$ such that $$f(x_2)\in B_{S_Y}(y_2,\xi_{d_G(x_1,x)+1})$$ or 

\item condition 1 holds but there does not exist an edge $e(y_1,y_2)\in E(S_Y)$ connecting $y_1$ and $y_2$ such that $$f(e(x_1,x_2))\subset B_{S_Y}(e(y_1,y_2),10\xi_{d_G(x_1,x)+1}).$$

\end{enumerate}
Then by \cref{edges} 
\begin{equation}\label{eq644}\int_{B_{S_X}(e(x_1,x_2),\xi_{d_G(x_1,x)})}|f^*\phi_{S_Y}-\phi_{S_X}|^2|\phi_{S_X}|^2 > \beta.
\end{equation}
(Note that \cref{edges} requires $\beta$ to be sufficiently small compared to $\xi_{d_G(x_1,x)}$ which holds in this case because of \cref{eq641}, \cref{eq642} and \cref{eq643} combined.) \cref{eq644} implies that 
\begin{align*}\int_{B_{S_X}(R_i,3+\epsilon_1)}|f^*\phi_{S_Y}-\phi_{S_X}|^2|\phi_{S_X}|^2& > \beta\\&= \xi^8(g/T)^{36}\\&=(\epsilon_1\delta^{-1}/80)^8(g/T)^{36}.
\end{align*} Therefore $R_i\notin E_2$, which is a contradiction.
\end{proof} This completes the proof of \cref{properties of rectangles} as well.
\end{proof}

\begin{lemma}\label{counting 3} We have, $$|E_1|-|E_2|\leq 10^5(\epsilon_1\delta^{-1}/80)^{-8}\alpha_2(g/T)^{\kappa_4-36} g.$$
\end{lemma}

\begin{proof} By \cref{eq60.75}, $$\int_X|f^*\phi_{S_Y}-\phi_{S_X}|^2|\phi_{S_X}|^2\leq \alpha_2(g/T)^{\kappa_4} g.$$ By definition, if $R_i\in E_1\setminus E_2$, $$\int_{B_{S_X}(R_i,3+\epsilon_1)}|f^*\phi_{S_Y}-\phi_{S_X}|^2|\phi_{S_X}|^2\geq (\epsilon_1\delta^{-1}/80)^8(g/T)^{36}.$$ Since $S_X\in \Tranc{g}{T}$, degrees of vertices in $S_X$ are bounded above by $42$ and distances between distinct vertices of degree strictly greater than $6$ are bounded below by $5$. Therefore, any radius $3+\epsilon_1$-ball in the flat metric on $S_X$ centered at a vertex of $S_X$ nontrivially intersects at most $10^5$ vertices, edges and triangles. Combined with the fact that the $R_i$s are disjoint, this means any point in $S_X$ is contained in at most $10^5$ of the $B_{S_X}(R_i,3+\epsilon_1)$. The lemma follows.
\end{proof}

Given a subcomplex $Q\subset S_X$, let $ST(Q)$ denote its closed star. Let $$G_X=\bigcup_{R_i\in E_2} ST(ST(ST(R_i))).$$ 

\begin{lemma}\label{simplicial map} There exists a $2$-subcomplex $G_Y$ of $S_Y$ along with a simplicial map $$\mathfrak{f}:G_X\to G_Y$$ such that 
\begin{enumerate}

\item $\mathfrak{f}$ is an isomorphism,

\item for every vertex $x\in V(G_X)$, $$f(x)\in B_{S_Y}(\mathfrak{f}(x),\epsilon_1),$$ 

\item for every edge $e\in E(G_X)$, $$f(e)\subset B_{S_Y}(\mathfrak{f}(e),\epsilon_1)$$ and 

\item for every triangular face $t\in F(G_X)$, $$f(t)\subset B_{S_Y}(\mathfrak{f}(t),\epsilon_1).$$ 

\end{enumerate}
\end{lemma}

\begin{proof} In \cref{properties of rectangles} we already constructed $\mathfrak{f}$ on vertices and edges of $G_X$. By \cref{faces}, $\mathfrak{f}$ extends to faces of $G_X$ as well, is a simplicial map, and satisfies conditions 2, 3 and 4 in the statement of the lemma. We show now that $\mathfrak{f}$ is injective onto its image. Suppose the contrary. Then there are distinct triangular faces $t_{X,1},t_{X,2}\in F(G_X)$ such that $$\mathfrak{f}(t_{X,1})=\mathfrak{f}(t_{X,2})=t_Y\in F(S_Y).$$ Conditions 2 and 3 imply that $$f(\partial t_{X,1}),f(\partial t_{X_2})\subset B_{S_Y}(\partial t_Y,\epsilon_1)$$ while condition 4 implies that $$f(t_{X,1}),f(t_{X,2})\subset B_{S_Y}(t_Y,\epsilon_1).$$ As a result the intersection $$f(\interior{(t_{X,1})})\cap f(\interior{(t_{X,2})})$$ must be nonempty. However, this implies $f$ is not injective, which is a contradiction. Therefore $\mathfrak{f}$ is injective onto its image. Let $G_Y=\mathfrak{f}(G_X)$. The lemma follows. 
\end{proof}

Let $$V_{\corner}(F(B_X)\setminus E_2)$$ the set of corner vertices of $R_i$ over all $R_i\in F(B_X)\setminus E_2$. Let $$G'_X=\overline{\left(\bigcup_{R_i\in E_2}R_i\right)\setminus \left(\bigcup_{x\in V_{\corner}(F(B_X)\setminus E_2)\cap\left(\partial\cup_{R_i\in E_2}R_i\right)}ST(x)\right)}.$$ (By \cref{length and width of R_i}, these stars are all disjoint.) Let $$J_X=\overline{S_X\setminus G'_X}.$$ Also, let $$G'_Y=\mathfrak{f}(G'_X),$$ and $$J_Y=\overline{S_Y\setminus G'_Y}.$$ 

\begin{lemma}\label{property of G'} 

\begin{enumerate}

\item Each connected component of the simplicial complexes $G'_X$, $G'_Y$, $J_X$ and $J_Y$ is a triangulated surface with boundary. 

\item Let $x$ be a vertex of $G'_X$ contained in $\partial G'_X$. Then  $$ST(ST(ST(x))))\subset G_X.$$ Here, the closed stars are taken in $S_X$. 

\item Let $x$ be a vertex of $G'_X$ contained in $\partial G'_X$. Then $$ST(ST(ST(\mathfrak{f}(x))))\subset G_Y.$$ Here, the closed stars are taken in $S_Y$. 

\item The simplicial complex $$\overline{S_X\setminus \left(\cup_{R_i\in E_2}R_i\right)}$$ is homotopy equivalent to $J_X$. 

\end{enumerate}
\end{lemma}

\begin{proof} First, we show that each connected component of $G'_X$ is a triangulated surface with boundary. To do this, it suffices to show that each vertex in $V(B_X)$ lying on $\partial G'_X$ has a neighborhood in $\partial G'_X$ homeomorphic to a half-disk. (Neighborhoods of the other boundary points are automatically homeomorphic to half-disks by construction.) Suppose the contrary. Then $\partial G'_X$ contains a vertex $v\in V(B_X)$, whose $1/4$-neighborhood in $G'_X$ is not homeomorphic to a half-disk. By property 4 of \cref{property of B_X}, $v$ is not contained in the closed star of any other vertex in $V(B_X)$. Therefore, locally near $v$, $G'_X$ looks like $\cup_{R_i\in E_2}R_i$. This means there are edges $e_1,e_2,e_3$ and $e_4$ emanating from $v$, such that $$e_1,e_2,e_3,e_4\subset \partial \bigcup_{R_i\in E_2}R_i=\partial \bigcup_{R_i\in F(B_X)\setminus E_2}R_i.$$ Now, we divide into two cases. In the first case, $v\in V_{=6}(S_X)$. Let $$R_j\in F(B_X)\setminus E_2$$ be a parallelogram containing the vertex $v$ (which exists since $v$ lies on the boundary of $\cup_{R_i\in E_2}R_i$). Suppose $v$ is not a corner vertex of $R_j$. Then there exists $e\in E(S_X)$ emanating from $v$ such that $\zeta(e,v)=\pm 1$ or $\pm e^{\pi i/3}$,  $e\subset R_j$ but $e$ does not lie on $\partial R_j$. By property 1 of \cref{property of B_X}, $e$ must be one of the $e_1,e_2,e_3$ or $e_4$. Since $$e\subset \partial \bigcup_{R_i\in F(B_X)\setminus E_2}R_i,$$ $e\subset\partial R_j$, which is a contradiction. For the second case, suppose $v\in V_{>6}(S_X)$. Then by properties 1 and 2 of \cref{property of B_X}, $v$ is a corner vertex of all parallelograms that contain it, which means that $$v\in V_{\corner}(F(B_X)\setminus E_2).$$ This is also a contradiction. Thus, connected components of $G'_X$ are triangulated surfaces with boundary. 

Since $J_X=\overline{S_X\setminus G'_X}$, connected components of $J_X$ are also triangulated surfaces with boundary. Since $G'_Y\simeq G'_X$ and $J_Y=\overline{G'_Y\setminus J_Y}$, connected components of $G'_Y$ and $J_Y$ are triangulated surfaces with boundary as well. 

Next, since $$G_X= \bigcup_{R_i\in E_2}ST(ST(ST(R_i)))$$ and $$G'_X\subset \bigcup_{R_i\in E_2}R_i,$$ $$ST(ST(ST(G'_X)))\subset G_X.$$ Since $$\mathfrak{f}:G_X\to G_Y$$ is an isomorphism, $$ST(ST(ST(\mathfrak{f}(G'_X))))\subset G_Y.$$

Finally, we show the last statement in the lemma. The subcomplex $J_X$ is obtained by adding the closed star of some boundary vertices of $\overline{S_X\setminus \left(\cup_{R_i\in E_2}R_i\right)}$ to it. These additions are all disjoint and deformation retract onto subsets of $\partial \overline{S_X\setminus \left(\cup_{R_i\in E_2}R_i\right)}$, hence  $J_X$ deformation retracts onto $\overline{S_X\setminus \left(\cup_{R_i\in E_2}R_i\right)}$.
\end{proof}

Now, suppose $J_X$ has $n$ connected components. Let $J_X^i$ be the $i$th connected component, a surface of genus $h_i$ with $b_i$ number of boundary components. By \cref{simplicial map}, since $\partial J_X^i\subset G_X$, $$f(J_X^i)\subset S_Y$$ is isotopic to a connected component of $J_Y$ that we label $J_Y^i$. Also by \cref{simplicial map}, all connected components of $J_Y$ arise in this manner. Similarly, $f(G'_X)$ is isotopic to $G'_Y$. Therefore and $G'_X$ and $G'_Y$ are homeomorphic, and $J_X$ and $J_Y$ are homeomorphic.

\begin{lemma}\label{J genus} The Euler characteristic of $J_X$ satisfies $$\chi(J_X)\geq -\alpha_4 g$$ where $$\alpha_4=96(42 \alpha_3 +\delta^{-1}/2+10^5(\epsilon_1\delta^{-1}/80)^{-8}\alpha_2(g/T)^{\kappa_4-36}).$$ Similarly, $$\chi(J_Y)\geq -\alpha_4 g.$$ Moreover, $J_X$ and $J_Y$ each have at most $\alpha_4 g$ connected components.
\end{lemma} 

\begin{proof} By \cref{counting 1}, \cref{counting 2} and  \cref{counting 3}, 
\begin{equation}\label{eq644.5}
\begin{aligned}|F(B_X)\setminus E_2|&=N-|E_2|\\&\leq (42 \alpha_3 +\delta^{-1}/2+10^5(\epsilon_1\delta^{-1}/80)^{-8}\alpha_2(g/T)^{\kappa_4-36})g.
\end{aligned}
\end{equation} Now, $\overline{S_X\setminus \left(\cup_{R_i\in E_2}R_i\right)}$ is naturally a subpolytope of $B_X$. We denote this subpolytope by $B'_X$. The vertices of $B'_X$ (denoted $V(B'_X)$) are the vertices of $B_X$ in $\overline{S_X\setminus \left(\cup_{R_i\in E_2}R_i\right)}$, the edges of $B'_X$ (denoted $E(B'_X)$) are the edges of $B_X$ in $\overline{S_X\setminus \left(\cup_{R_i\in E_2}R_i\right)}$, and the faces of $B'_X$ are the $R_i\in F(B_X)\setminus E_2$. As in the case of $B_X$, the faces of $B'_X$ are not necessarily quadrilaterals as polytope faces since not every vertex in $V(B'_X)$ is a corner vertex of every parallelogram that contains it.

Given a vertex $v\in V(B'_X)$, let the degree of $v$ in $B'_X$ denote the number of edges of $B'_X$ emanating from it (with loops contributing $2$ to the degree). Denote by $V_{=2}(B'_X)$ the set of vertices of $B'_X$ that have degree $2$ in $B'_X$, and $V_{>2}(B'_X)$ the set of vertices of $B'_X$ that have degree strictly greater than $2$ in $B'_X$. 

We claim that $$V_{>2}(B'_X)\subset V_{\corner}(F(B_X)\setminus E_2).$$ To see this, note that if $v\in V_{>2}(B'_X)$ and $v\in V_{>6}(S_X)$, then by properties 1 and 2 of \cref{property of B_X}, $v$ is a corner vertex of every $R_i$ in $F(B_X)\setminus E_2$ containing $v$. Now suppose $v\in V_{>2}(B'_X)$ and $v\in V_{=6}(S_X)$. Because $v\in V_{>2}(B'_X)$, there are at least two distinct $R_i$ and $R_j$ parallelograms in $F(B_X)\setminus E_2$ containing $v$. If $v$ is a corner vertex of neither, then $R_i\cup R_j$ contains a neighborhood of $v$, which would be a contradiction to the fact that that degree of $v$ strictly greater than $2$. So $v$ is a corner of one of the parallelograms.

Since each $R_i$ has at most $4$ corner vertices, \cref{eq644.5} implies $$|V_{>2}(B'_X)|\leq 4(42 \alpha_3 +\delta^{-1}/2+10^5(\epsilon_1\delta^{-1}/80)^{-8}\alpha_2(g/T)^{\kappa_4-36})g.$$ The total number of edges of $B'_X$ is half the sum of the degrees of all the vertices in $V(B'_X)$. Since $S_X$ is locally bounded, the sum of the degrees of all the vertices in $V(B'_X)$ is at most $$2|V_{=2}(B'_X)|+42|V_{>2}(B'_X)|.$$ (Note that since $B'_X$ is a polytope, it does not have any vertices of degree $1$.) Thus, 
\begin{align*}\chi(B'_X)&\geq |V(B'_X)|-|E(B'_X)|\\&\geq -24|V_{>2}(B'_X)|\\&\geq -96(42 \alpha_3 +\delta^{-1}/2+10^5(\epsilon_1\delta^{-1}/80)^{-8}\alpha_2(g/T)^{\kappa_4-36})g\\&=-\alpha_4 g.
\end{align*} By \cref{property of G'}, $B'_X$ and $J_X$ are homotopy equivalent, so $$\chi(J_X)\geq -\alpha_4g.$$ Since $J_Y$ is homeomorphic to $J_X$, we have $$\chi(J_X)\geq -\alpha_4g$$ as well. The number of connected components of $J_X$ is at most $|F(B_X)\setminus E_2|$ since each face in $F(B_X)\setminus E_2$ belongs to at most one connected component. Hence this number, along with the number of connected components of $J_Y$, is at most $\alpha_4 g$.
\end{proof}

Let $$K_X^i=(J_X^i)^d,$$ the conformal double of $J_X^i$ as in \cref{conformal double}. Since $J_X^i$ is a triangulated surface of genus $h_i$ with $b_i$ boundary components, its conformal double is a triangulated surface (without boundary) of genus $g_i=2h_i+b_i-1$. By construction, $\partial J_X^i$ is identified with $\partial (J_X^i)^{-1}$ which in turn is identified with $\partial (J_Y^i)^{-1}$. Let $$K_Y^i=J_Y^i \cup (J_X^i)^{-1},$$ glued along this identification. Then $K_Y^i$ is also a triangulated surfaces of genus $g_i=2h_i+b_i-1$.

Recall that $n$ is the number of connected components of $J_X$ (and therefore also $J_Y$,  $K_X$ and $K_Y$).

\begin{lemma}\label{K genus} We have, $$\sum_{i=1}^n g_i \leq \alpha_5 g$$ where $g_i$ denotes the genus of $K_X^i$ which is equal to the genus of $K_Y^i$. Here, $\alpha_5=2\alpha_4$.
\end{lemma}

\begin{proof} By \cref{J genus}, $$\chi(K_X)=2\chi(J_X)\geq -2\alpha_4 g.$$ Now, $\chi(K_X^i)=2-2g_i,$ so
\begin{align*}2n-2 \sum_{i=1}^n g_i&=\chi(K_X)\\&\geq -2\alpha_4 g.
\end{align*} Since $n\leq \alpha_4 g$ by \cref{J genus}, we have $$\sum_{i=1}^n g_i\leq 2\alpha_4 g$$ as desired.
\end{proof}

Next, we bound the number of vertices of $K_X$ and $K_Y$ that have degree not equal to $6$.

\begin{lemma}\label{K vertices} We have, $$|V_{\neq 6}(K_X)|\leq \alpha_6 g$$ and $$|V_{\neq 6}(K_Y)|\leq \alpha_6 g$$ where $\alpha_6= 500\alpha_4$.
\end{lemma}

\begin{proof} First, we compute $|V_{<6}(K_X)|$. Since $K_X$ is the double of $J_X$ and $J_X$ is a subset of $S_X$ (which, being in $\Tranc{g}{T}$, does not have any vertices of degree strictly less than $6$), all vertices of $K_X$ having degree strictly less than $6$ must lie on $\partial J_X\subset K_X$. Such a vertex will have degree strictly less than $4$ considered as a vertex of $J_X$, so it suffices to compute the number of vertices in $\partial J_X$ that have degree strictly less than $4$ in $J_X$. Denote by $V_{<4}(\partial J_X)$ the set of vertices in $\partial J_X$ that have degree strictly less than $4$ in $J_X$. (Recall that degree for a vertex in a triangulated surface with boundary is defined as the number of edges emanating from the vertex. For a vertex on the boundary, this quantity is different from the number of faces containing the vertex.) 

Since $J_X=\overline{S_X\setminus G'_X}$ and $G'_X$ are created by removing stars of boundary vertices of $\cup_{R_i\in E_2}R_i$ that are also in $V_{\corner}(F(B_X)\setminus E_2)$, any vertex in $V_{<4}(\partial J_X)$ is either in $V_{\corner}(F(B_X)\setminus E_2)$ or adjacent to a vertex in $V_{\corner}(F(B_X)\setminus E_2)$. Since by \cref{eq644.5}, \begin{align*}|V_{\corner}(F(B_X)\setminus E_2)|&\leq 4|F(B_X)\setminus E_2|\\&\leq \alpha_4 g
\end{align*} and each vertex in $J_X$ has degree at most $42$, $|V_{<4}(\partial J_X)|\leq 43\alpha_4 g$. Therefore 
\begin{equation}\label{eqn761}|V_{<6}(K_X)|\leq 43 \alpha_4 g
\end{equation} as well. Note that 
\begin{align*}\chi(K_X)&=|V(K_X)|-|E(K_X)|+|F(K_X)|\\&=|V(K_X)|-(1/3)|E(K_X)|.
\end{align*} Since there are at least $$3|V_{=6}(K_X)|+(7/2)|V_{>6}(K_X)|$$ edges in $K_X$, $$\chi(K_X)\leq |V_{<6}(K_X)|-(1/6)|V_{>6}(K_X)|$$ which means $$|V_{>6}(K_X)|\leq -6\chi(K_X)+6|V_{<6}(K_X)|.$$ From \cref{K genus} and \cref{eqn761}, we obtain \begin{align*} |V_{\neq 6}(K_X)|&\leq |V_{<6}(K_X)|+ |V_{>6}(K_X)|\\&\leq 500 \alpha_4 g.
\end{align*}

Since $\partial J_X$ is identified with $\partial J_Y$, a similar argument shows that $|V_{<6}(K_Y)|\leq 43\alpha_4 g$. Since $\chi(K_X)=\chi(K_Y)$, we obtain $$|V_{\neq 6}(K_Y)|\leq 500 \alpha_4 g$$ too.
\end{proof}

\subsection{Quasiconformal extension on a triangulated surface}\label{6.6} The goal of this section is to show:

\begin{lemma}\label{K distance} There exists a universal constant $C$ such that $$d_T(\Phi(K_X^i),\Phi(K_Y^i))\leq C$$ for all $i$ satisfying $g_i\geq 2$.
\end{lemma}

We will do this by using the description of the Teichm\"{u}ller metric via extremal length.

First, we consider $\Phi(K_X^i)$ and $\Phi(K_Y^i)$ as points in $\mathcal{T}_g$. To do this, it suffices to fix a marking of $\Phi(K_X^i)$ and construct a diffeomorphism $K_X^i\to K_Y^i$. To do this, we first construct a diffeomorphism $$F:J_X^i\to J_Y^i.$$ On $$J_X^i\setminus ST(\partial J_X^i),$$ we let $$F=f.$$ On $\partial J_X^i$, we let $$F=\mathfrak{f}.$$ By \cref{simplicial map} and \cref{property of G'}, we may extend $F$ to $ST(\partial J_X^i)$ so that $$F(x)\in B_{J_Y^i}(\mathfrak{f}(x),1).$$ Gluing $F$ with the identity map $(J_X^i)^{-1}\to (J_X^i)^{-1}$, we obtain the desired diffeomorphism $K_X^i\to K_Y^i$.

Let $\gamma$ be a free homotopy class on $K_X^i$. Since we have constructed a diffeomorphism $K_X^i\to K_Y^i$, $\gamma$ also corresponds to a free homotopy class on $K_Y^i$ that we will also denote by $\gamma$. By \cref{Teichmuller metric via extremal length}, to show \cref{K distance}, we need to show, 

$$\frac{\Ext_{K_Y^i}(\gamma)}{\Ext_{K_X^i}(\gamma)}\leq C.$$

Now, $$\textstyle\Ext_{K_Y^i}(\gamma)=\displaystyle\sup_\rho \frac{\length_\rho(\gamma)^2}{\Area_\rho(K_Y^i)},$$ where $\rho$ ranges over conformal metrics on $K_Y^i$. By \cref{Jenkins-Strebel}, the supremum is achieved by the flat metric $|\phi|^{1/2}$ associated to a holomorphic quadratic differential $\phi$ on $K_Y^i$.

We have, $$\textstyle\Ext_{K_X^i}(\gamma)=\displaystyle\sup_\rho \frac{\length_\rho(\gamma)^2}{\Area_\rho(K_X^i)}.$$ We will show that $$\textstyle\Ext_{K_X^i}(\gamma)\geq C\Ext_{K_Y^i}(\gamma)$$ by constructing a conformal metric $\rho$ on $K_X^i$ for which $$\textstyle\length_{\rho}(\gamma)\geq C\length_{\phi}(\gamma)$$ and $$\textstyle\Area_{\rho}(K_X^i)\leq C\Area_{\phi}(K_Y^i).$$

Let $$A_X^i=J_X^i\setminus ST(\partial J_X^i)$$ and
$$A_Y^i=J_X^i\setminus ST(\partial J_Y^i).$$ Let $$A^i_{X,1/4}= B_{A_X^i}(\partial A_X^i,1/4)$$ and $$A^i_{Y,1/4}=B_{A_Y^i}(\partial A_Y^i,1/4).$$ By construction and \cref{length and width of R_i}, $A^i_{X,1/4}$ and $A^i_{Y,1/4}$ are both the union of disjoint annuli. 

Now, we construct a conformal metric on $K_X^i$ as follows. On $A^i_X$, let $\rho_{\phi'}$ be the smallest conformal metric such that $\rho_{\phi'}\geq f^*(|\phi|^{1/2}).$ This is well defined since $$f(A_X^i)\subset B_{S_Y}(A_Y^i,\epsilon_1)\subset J_Y^i\subset K_Y^i$$ by \cref{simplicial map} and \cref{property of G'}. Extend $\rho_{\phi'}$ to the rest of $K_X^i$ (for instance, multiplying by a partition of unity) so that the area increases infinitesimally.

Define $$C^i_X=J_X^i \cap B_{K_X^i}(ST(ST(\partial J_X^i)),\epsilon_1)$$ and $$C^i_Y=J_Y^i\cap B_{K_Y^i}(ST(ST(\partial J_Y^i)),\epsilon_1).$$ By \cref{length and width of R_i}, $C_X^i$ and $C_Y^i$ are the disjoint union of annuli. By \cref{simplicial map} and \cref{property of G'} $$\mathfrak{f}:C^i_X\to C^i_Y$$ is an isometry onto its image, and agrees with the identity map on $\partial J^i_X$. Hence, there is an isometry $$\id\cup\mathfrak{f}:(J^i_X)^{-1}\cup C^i_X\to (J^i_Y)^{-1}\cup C^i_Y.$$ On $(J_X^i)^{-1}\cup C^i_X$, we let $\rho_{\phi}$ be the pullback of $|\phi|^{1/2}$ defined on $(J_X^i)^{-1}\cup C^i_Y$. As before, we may extend $\rho_{\phi}$ to the rest of $K_X^i$ so that the area increases infinitesimally.

Cover $C^i_X$ by balls $\{U_k\}$ of radius $1/4$ in the $S_X$-metric such that every $1/16$-radius ball on $C^i_X$ is contained in some $U_k$. Because $S_X$ is a locally bounded combinatorial translation surface, we may choose the cover so that each ball $U_k$ is either flat, or a cone with $\cent(U_k)$ the cone point. We may also assume that every $2$-ball in $B_X$ nontrivially intersects at most $C$ of the $U_k$s. For each $U_k$, there is a holomorphic map $g_k:\mathbb{D}\to U_k$ sending $0$ to $\cent(U_k)$. Take 
\begin{equation}\label{def of m eq} m_k=\max_{U_j\cap B_{S_Y}(U_k,2)\neq 0} \frac{|(\mathfrak{f}\circ g_j)^*(\phi)|}{|dz|^2}.
\end{equation} By the mean value property, 
\begin{equation}\label{mvp eq} m_k\leq \max_{U_j\cap B_{S_Y}(U_k,2)\neq 0} C\int_{U_j} |\phi|.
\end{equation} Let $$\rho_k=(g_k)^*(m_k\cdot \rho_{\Euc})$$ on $U_k$, and extend it to the rest of $K_X^i$ as before. 

Finally, we let $$\rho=\rho_{\phi'}+\rho_{\phi}+\sum_{k}\rho_k.$$ By construction, 
\begin{align*}\textstyle\Area_{\rho}(K_X^i)&\leq
\textstyle\Area_{\rho_{\phi'}}(K_X^i)+\Area_{\rho_{\phi}}(K_X^i)+\displaystyle\sum_{k}\textstyle\Area_{\rho_k}(K_X^i)\\& \leq \textstyle C\Area_{\phi}(K_Y^i)+\displaystyle\sum_{k}\textstyle\Area_{\rho_k}(K_X^i) && \text{quasiconformality of } f\\& \leq \textstyle C\Area_{\phi}(K_Y^i)+C\displaystyle\sum_{k}\int_{U_k}|\phi| && \text{\cref{mvp eq}}\\&\leq C\textstyle\Area_{\phi}(K_Y^i).
\end{align*}

It remains to show that $$\textstyle\length_{\rho}(\gamma)\geq C\length_{\phi}(\gamma).$$ Let $\alpha$ be a length minimizing curve on $K_X^i$ the free homotopy class $\gamma$. We divide $\alpha$ into disjoint pieces in the following way. 

Now, $K_X^i$ may be decomposed as $$\left((J^i_X)^{-1}\cup ST(\partial J_X^i)\right)\bigcup A^i_{X,1/4}\bigcup A_X^i\setminus A^i_{X,1/4}.$$ Decompose $\alpha$ into arcs $\alpha_1\cup ...\cup\alpha_\ell$ such that each $\alpha_j$ is contained in one of $A^i_{X,1/4}$, $(J^i_X)^{-1}\cup ST(\partial J_X^i)$ or $A_X^i\setminus A^i_{X,1/4}$ with endpoints on its boundary. Label the arcs so that $\alpha_j$ and $\alpha_{j+1}$ share an endpoint, and $\alpha_\ell$ and $\alpha_1$ share an endpoint.

We now construct a curve $\beta$ on $K_Y^i$. We do this by constructing arcs $\beta_1,...,\beta_\ell$, where each arc $\beta_j$ will be constructed from the arc $\alpha_j$ on $K_X^i$. We will construct the $\beta_j$ so that it shares an endpoint with $\beta_{j+1}$, and $\beta_1$ and $\beta_\ell$ also share an endpoint. Thus, the $\beta_j$ will glue together to form $\beta$. 

If $\alpha_j\subset  A_X^i\setminus A^i_{X,1/4}$, we let $$\beta_j=f(\alpha_j).$$ If $\alpha_j\subset  (J^i_X)^{-1}\cup ST(\partial J_X^i),$ we let $$\beta_j=(\id\cup \mathfrak{f})(\alpha_j).$$ By construction, in both these cases, $$\textstyle\length_{\phi}(\beta_j)\leq \length_{\rho}(\alpha_j).$$

Next, we treat the case wherein $\alpha_j\subset  A^i_{X,1/4}$. The region $A^i_{X,1/4}$ has two types of boundary components. Let $\partial_0 A^i_{X,1/4}$ be union of the boundary components of $A^i_{X,1/4}$ which are also boundary components of $(J^i_X)^{-1}\cup ST(\partial J_X^i)$. Let $\partial_{1/4} A^i_{X,1/4}$ be the union of the boundary components which are also boundary components of $A_X^i\setminus A^i_{X,1/4}$. If $\alpha_i$ has both endpoints on $\partial_0 A^i_{X,1/4}$, we let $$\beta_j=\mathfrak{f}(\alpha_j).$$ If $\alpha_j$ has both endpoints on $\partial_{1/4} A^i_{X,1/4}$, we let $$\beta_j=f(\alpha_j).$$ In these situations, $$\textstyle\length_{\phi}(\beta_j)\leq \length_{\rho}(\alpha_j)$$ as before. 

It remains to treat the case wherein $\alpha_j$ has one endpoint $x$ on $\partial_{1/4} A^i_{X,1/4}$ and the other endpoint $y$ on $\partial_0 A^i_{X,1/4}$. Given a point $z$ on $\alpha_j$, let $\alpha_j^{x}$ denote the arc of $\alpha_j$ from $x$ to $z$, and $\alpha_j^{y}$ denote the arc of $\alpha_j$ from $z$ to $y$. By construction, $$\alpha_j=\alpha_j^{x}\cup \alpha_j^{y}.$$ Now, there exists $$z\in \alpha_j \cap B_{K^i_X}(x,1/16)$$ such that $$\textstyle\length_{S_X}(\alpha^x_j)\geq 1/16.$$ By construction of the $U_k$s, there is some $U_k$ containing $x$ and $z$. Since $U_k$ is a radius $1/4$-ball and $x\in \partial_{1/4} A^i_{X,1/4}$, $U_k$ is flat. Hence, $g_k:\mathbb{D}\to U_k$ is an isometry from the Euclidean metric on $\mathbb{D}$ to the $S_X$-metric on $U_k$. Thus, $$\textstyle\length_{\Euc}(g_k^{-1}(\alpha_j^x))\geq 1/16.$$ This means $$\textstyle\length_\rho(\alpha_j^x)\geq C\cdot m_k.$$ Since $U_k$ is a radius $1/4$-ball in the $S_X$-metric containing $x$, $U_k\subset C_X^i$. 

By construction, $\mathfrak{f}(z)\in C_Y^i$. Note that $f(x)\in C_Y^i$ by \cref{simplicial map} and \cref{property of G'}.  Again by \cref{simplicial map} and \cref{property of G'}, $$d_{S_Y}(f(x),\mathfrak{f}(z))\leq 1.$$ Let $$\beta^x_j\subset B_{S_Y}(f(x),1)\cap C_Y^i$$ be an arc from $f(x)$ to $\mathfrak{f}(z)$. We may construct $\beta^x_j$ so that it passes through at most $C$ of the $\mathfrak{f}(U_h)$. For each $\mathfrak{f}(U_h)$ through which $\beta^x_j$ passes, we may assume $$\textstyle\length_{\Euc}((\mathfrak{f}\circ g_h)^{-1}(\beta^x_j\cap \mathfrak{f}(U_h)))\leq 1/2$$ by replacing with an another arc with the same endpoints if necessary. By \cref{def of m eq}, on each such $\mathfrak{f}(U_h)$, $$\textstyle\length_{\phi}(\beta^x_j\cap \mathfrak{f}(U_h))\leq m_k/2.$$ Therefore, $$\textstyle\length_{\phi}(\beta^x_j)\leq C\cdot m_k.$$ Now, define $$\beta^y_j=\mathfrak{f}(\alpha^y_j)$$ and $$\beta_j=\beta^x_j\cup \beta^y_j.$$ By construction, $\beta_j$ is well defined with endpoints $f(x)$ and $\mathfrak{f}(y)$. Also, $$\textstyle\length_{\phi}(\beta_j)\leq C\textstyle\length_{\rho}(\alpha_j).$$ 

Let $\beta=\beta_1\cup...\cup \beta_\ell$. By construction, under the identification of free homotopy classes coming from the diffeomorphism $K_X^i\to K_Y^i$, $\alpha$ and $\beta$ belong to the same free homotopy class. Finally, we have $$\textstyle\length_{\phi}(\beta)\leq C\textstyle\length_{\rho}(\alpha)$$ which implies $$\textstyle\length_{\rho}(\gamma)\geq C\length_{\phi}(\gamma)$$ as desired. This completes the proof of \cref{K distance}.

\subsection{Proof of \cref{quantitative form determines triangulation}}\label{6.7} We must compute the number of possible $S_Y$. To do this, it suffices to compute the number of possible $G'_Y$, the number of possible $J_Y$ and the number of possible gluings of $J_Y$ with $G'_Y$. Now, $G'_Y$ admits a simplicial isomorphism to $G'_X$, which itself is a subcomplex of $S_X$ that is uniquely determined by the subset $E_2\subset F(B_X)$. (Note that $E_2$ depends on $S_Y$, but $F(B_X)$ does not.) So the number of possibilities of $G'_Y$ is the number of possibilities of $G'_X$, and this quantity is bounded above by the number of possibilities for $E_2$. Since $$|F(B_X)|\leq Cg$$ by \cref{number of parallelograms}, the number of possibilities for $E_2$ is bounded above by $C^g$. 

To compute the number of possibilities for $J_Y$, we first compute the number of possibilities for $K_Y$.  We start by treating the case of the connected components which have genus $0$ or $1$. Relabelling as necessary, we may assume that $$g_1,...,g_{n'}\in \{0,1\}$$ and $$g_{n'+1},...,g_n\geq 2.$$ Here, $g_i$ denotes the genus of $K_X^i$ and $K_Y^i$.

\begin{lemma}\label{genus 0 and 1} There are at most $(T/g)^{Cg}$ total choices for the $K_Y^1$,..., $K_Y^{n'}$.
\end{lemma}

\begin{proof} To do this, we first construct a graph  $A^i$ on each $K_Y^i$ for $i\in \{1,...,n'\}$. We construct a graph $A_Y^i$ on $J_Y^i$ and a graph $A_X^i$ on $(J_X^i)^{-1}$, then glue them together.  

First, we construct $A_Y^i$ on $J_Y^i$. Let $$V_{\notflat}(J_Y^i)=\left(\interior(J_Y^i)\cap V_{\neq 6}(J_Y^i) \right)\cup \left(\partial J_Y^i\cap V_{\neq 4}(J_Y^i)\right).$$ 

Recall from \cref{6.4} that $S_Y$ has a horizontal foliation. Restrict this foliation to $J_Y^i$. Let $\mathcal{L}$ be the set of leaves of this restricted foliation which pass through a vertex in $V_{\notflat}(J_Y^i)$. Define $$A_Y^i=\partial J_Y^i\bigcup_{L\in \mathcal{L} }.$$

Now, $A_Y^i$ is naturally a graph on $J_Y^i$. The vertices of $A_Y^i$ in $\interior(J_Y^i)$ are $$\interior(J_Y^i)\cap V_{\neq 6}(J_Y^i).$$ The vertices of $A_Y^i$ in $\partial J_Y^i$ are $$\partial J_Y^i\cap \bigcup_{L\in \mathcal{L} }.$$ Let $$V(A_Y^i)=\left(\interior(J_Y^i)\cap V_{\neq 6}(J_Y^i)\right)\cup \left(\partial J_Y^i\cap \bigcup_{L\in \mathcal{L} }\right).$$ The edges of $A_Y^i$ are the connected components of $A_Y^i\setminus V(A_Y^i)$. Each edge of $A_Y^i$ is contained in either $$\partial J_Y^i$$ or $$\bigcup_{L\in \mathcal{L} }.$$ If an edge of $A_Y^i$ is contained in $$\partial J_Y^i,$$ we call it a boundary edge; otherwise, an interior edge. Each edge of $A_Y^i$ is the union of some parallel edges of $J_Y^i$, and therefore has both a length and a directional weight associated to it. By construction, $$|V(A_Y^i)|\leq C\cdot V_{\notflat}(J_Y^i).$$ Moreover, since $S_Y$ is locally bounded, the degrees of vertices of $A_Y^i$ is bounded by $C$. Therefore the number of edges is bounded by $C|V_{\notflat}(J_Y^i)|$. 

We claim that each face of $F$ of $J_Y^i\setminus A_Y^i$ is a flat cylinder, or embeds isometrically in the plane. To see this, first note that by construction, the angles of $\partial F$ are at most $\pi$, and the interior of $F$ is flat. Hence, by Gauss-Bonnet, $\chi(F)\geq 0$ which means $\chi(F)=0$ or $\chi(F)=1$. In the case $\chi(F)=0$, $\partial F$ must not contain any angles less than $\pi$, hence $F$ is a flat cylinder. In the case $\chi(F)=1$, $F$ is a flat topological disk with piecewise geodesic boundary such that $\partial F$ does not contain any angles strictly greater than $\pi$. So $F$ is geodesically convex. Therefore the exponential map gives an isometry between a planar domain and $F$. 

Analogously, we may construct a graph $A_X^i$ on $(J_X^i)^{-1}$ which satisfies the same vertex and edge bounds as above. Taking the union of $A_X^i$ and $A_Y^i$, we obtain a graph $A^i$ on $K_X^i$. Since $$ST(\partial J_X^i)\simeq ST(\partial J_Y^i),$$ we have $$V_{\notflat}((J_X^i)^{-1})\cup V_{\notflat}(J_Y^i)=V_{\neq 6}(K_Y^i).$$ Thus, $A^i$ has at most $| V_{\neq 6}(K_Y^i)|$ edges. 

The number of topological possibilities for $A^i$ as a graph embedded in a torus or sphere is at most $C^{| V_{\neq 6}(K_Y^i)|}.$ This is true because $A^i$ is contained in a rooted map on the sphere or torus with at most $C\cdot | V_{\neq 6}(K_Y^i)|$ edges. There are at most $C^{| V_{\neq 6}(K_Y^i)|}$ possibilities for such a rooted map (see \cite{BC86}, in which a bound for the number of rooted maps with prescribed number of edges on a surface is given). Now, $A^i$ may be obtained from the rooted map by removing some edges and vertices. There are at most $C^{| V_{\neq 6}(K_Y^i)|}$ ways to do this. Therefore, there are at most $C^{| V_{\neq 6}(K_Y^i)|}$ topological possibilities for $A^i$. This means, there are at most 
$C^{| V_{\neq 6}(K_Y^i)|}$ topological possibilities for $A_Y^i$. By \cref{K vertices}, there are at most $C^g$ possibilities for all the $A_Y^i$ for $i\in \{1,...,n'\}$. To reconstruct $J^i_Y$, it suffices to assign directional weights and lengths to each edge of each $A_Y^i$. There are at most $(T/g)^{Cg}$ choices to do that. Hence there are at most $(T/g)^{Cg}$ choices for the $J^i_Y$ for $i\in \{1,...,n'\}$. There are at most $(T/g)^{Cg}$ ways to glue all the $J^i_Y$ to the $(J_X^i)^{-1}$. The lemma follows.
\end{proof}

By \cref{K genus}, \cref{K distance}, \cref{K vertices} and \cref{genus 0 and 1} there are at most $$(T/g)^{Cg}\sum_{\substack{n\leq \alpha_4 g\\ T_1+...+T_n\leq 2T\\ g_1+...+g_n\leq \alpha_5 g\\ g_1,...,g_n\geq 2\\ m_1+...+m_n\leq \alpha_6 g}} \prod_{i=1}^n \Ntri(T_i,g_i,m_i, r+C)$$ choices of $K_Y$. 

We must compute how many choices for $J_Y$ there are given $K_Y$. Since $K_Y$ is formed by gluing $J_Y$ and $(J_X)^{-1}$, it suffices to count the number of embeddings from $(J_X^i)^{-1}$ to $K^i_Y$ for $i\in \{1,...,n\}$. Such an embedding is determined by its values on one triangle. Thus there are at most $(T/g)^{Cg}$ such embeddings.

The gluing of $J_Y$ and $G'_Y$ is determined by the decomposition of $S_X$ as $G'_X\cup J_X$, the isomorphism $G'_Y$ to $G'_X$, and the embedding of $(J_X)^{-1}$ into $K_Y$. 

Finally, as long as $\alpha_4\leq 1/100$ and $\alpha_5,\alpha_6\leq \mu^{-1}(1/100)$ (see \cref{Constants}), the number of choices for $S_Y$ is at most $$(T/g)^{Cg}\sum_{\substack{n\leq (1/100)g\\ T_1+...+T_n\leq 2T\\ g_1+...+g_n\leq \mu^{-1}(1/100)g\\g_1,...,g_n\geq 2\\ m_1+...+m_n\leq \mu^{-1}(1/100)g}} \prod_{i=1}^n \Ntri(T_i,g_i,m_i, r+C).$$ This completes the proof of \cref{quantitative form determines triangulation}.

\subsection{Constants}\label{Constants} In this section, we show that the constants $\alpha_1$, $\delta$, $\epsilon_0$, $\epsilon_1$, $\kappa_0$, $\kappa_1$, $\kappa_2$ and $\kappa_3$  can be appropriately chosen. We list the relationships between these constants below: 

\begin{enumerate}

\item $\mu\in \mathbb{N}$ is a universal constant, introduced in \cref{bounded degree surface}

\item $\kappa_0\in \mathbb{N}$, to be chosen, introduced in the beginning of \cref{6.0}

\item $\kappa_1\in \mathbb{N}$, to be chosen, introduced in statement of \cref{quantitative form determines triangulation}

\item $\kappa_2\in \mathbb{N}$, to be chosen, introduced in statement of \cref{count cohomology classes}

\item $\kappa_3\in \mathbb{N}$, to be chosen, introduced in proof of \cref{count cohomology classes}

\item $\alpha_1>0$, to be chosen, introduced in statement of \cref{quantitative form determines triangulation}

\item $\alpha_2=\alpha_1^2$, introduced in \cref{eq60.75}

\item $\kappa_4=2\kappa_1$, introduced in \cref{eq60.75}

\item $0<\epsilon_0<1/2$, to be chosen, introduced in the beginning of \cref{6.1}

\item $\alpha_3=10^{10}\alpha_2(g/T)^{\kappa_4-10}\epsilon_0^{-10}$, introduced in statement of \cref{vertices-x}

\item $0<\epsilon_1<1/1000$, to be chosen, introduced before \cref{properties of rectangles}

\item $\delta\in \mathbb{N}$, to be chosen, introduced in the beginning of \cref{6.5}

\item $\alpha_4=96(42 \alpha_3 +\delta^{-1}/2+10^5(\epsilon_1\delta^{-1}/80)^{-8}\alpha_2(g/T)^{\kappa_4-36})$, introduced in statement of \cref{J genus}

\item $\alpha_5=2\alpha_4$, introduced in statement of \cref{K genus}

\item $\alpha_6=500 \alpha_4$, introduced in statement of \cref{K vertices}

\item $\epsilon_0\leq \epsilon_1\delta^{-1}/40$, assumed in \cref{6.5} before \cref{properties of rectangles}

\item $\alpha_4\leq 1/100$, assumed in proof of \cref{quantitative form determines triangulation} in \cref{6.7}

\item $\alpha_5\leq (1/100)\mu^{-1}$, assumed in proof of \cref{quantitative form determines triangulation} in \cref{6.7}

\item $\alpha_6\leq (1/100)\mu^{-1}$, assumed in proof of \cref{quantitative form determines triangulation} in \cref{6.7}

\item $\kappa_1\leq (\kappa_2-1)/2$, assumed in statement of \cref{count cohomology classes}

\item $100\alpha_1^{-1}(1/2)^{(\kappa_2-1)/2-\kappa_1}\leq 1$, assumed in statement of \cref{count cohomology classes}

\item $\kappa_3\geq (\kappa_2-1)/2$, assumed in proof of \cref{count cohomology classes}

\item $\kappa_2\geq \kappa_0$, assumed in statement of \cref{count cohomology classes}

\item $8\cdot (1/2)^{\kappa_0}\leq (1/10)^{10}$, assumed in the beginning of \cref{6.1}

\end{enumerate}

Noting that $g/T\leq 1/2$, it is clear that $\kappa_0$, $\kappa_1$, $\kappa_2$, $\kappa_3$, $\alpha_1$, $\epsilon_0$, $\epsilon_1$ and $\delta$ can be chosen to satisfy conditions 16 through 24, in the following way. First, choose $\epsilon_1$ sufficiently small. Then, choose $\delta$ sufficiently large. Then, choose $\epsilon_0$ sufficiently small so that condition 16 is satisfied. Finally, choose $\alpha_1$ sufficiently small and $\kappa_1$ sufficiently large so that $\kappa_4>36$, $\alpha_3$, $\alpha_4$, $\alpha_5$ and $\alpha_6$ are sufficiently small. Then choose $\kappa_0$ sufficiently large such that condition 24 is satisfied. Finally, choose $\kappa_2$ and $\kappa_3$ sufficiently large so that conditions 20 through 23 are satisfied.

\section{Upper bounds for triangulated surfaces via combinatorial translation surfaces} The first goal of this section is to show:

\begin{lemma}\label{tric vs tranc} For $g\geq 2$, we have, $$\Ntric(T,g,m,r)\leq C^{(1+r)g}(T/m)^{Cm}\Ntranc(6T_i,6g+5m,r+C)$$ for a universal constant $C$.
\end{lemma}

Note that if $m=0$, then $g=1$. See also \cref{counting function notation}.

To show \cref{tric vs tranc}, in \cref{8.1} and \cref{8.3} we associate to each triangulated surface a branched $6$-cover which is a combinatorial translation surface. In \cref{8.4} we enumerate the number of combinatorial possibilities for the branched $6$-cover given the branch points. In \cref{8.2}, \cref{8.5}, \cref{8.6} and \cref{8.6 1/2}, we study where the branch points lie. In \cref{8.7}, we prove \cref{tric vs tranc} by showing that the connected components of the branched $6$-cover of a triangulated surface lie in the union of a relatively small number of balls of radius around a constant in a higher dimensional moduli space. 

The second goal of this section is to complete the proof of \cref{upper bound}, which follows from \cref{trim vs tricm}, \cref{tranc vs tri} and \cref{tric vs tranc}. We do this in \cref{8.8}.

\subsection{Triangulated surfaces and $6$-differentials}\label{8.1}
Given a Riemann surface $X$, a meromorphic $6$-differential is a global section of the sixth tensor power of the sheaf of meromorphic differentials on $X$. Locally, in a neighborhood in $X$ with holomorphic coordinate $z$, a $6$-differential behaves like $f(z)dz^6$ where $f$ is a holomorphic function. Although generally triangulated surfaces may not admit a combinatorial translation structure (and therefore do not admit a canonical holomorphic $1$-form), they do admit a canonical meromorphic $6$-differential. Given a triangulated surface $S$, we may associate a meromorphic $6$-differential $\psi_S$ as follows: we identify each triangle of $S$ with the unit equilateral triangle in $\mathbb{C}$ with vertices at $0$, $1$ and $\frac{1}{2}+\frac{\sqrt{3}}{2}i$, and define $$\psi_S=dz^6$$ under this identification. Gluings of triangles must preserve the form $dz^6$, and $\psi_S$ extends holomorphically over vertices of $S$, therefore $\psi_S$ is a globally defined $6$-differential on $S$. It has a zero/pole of order $(\deg x)-6$ at a point $x\in V_{\neq 6}(S)$. The flat metric on $S$ (equivalently denoted $S$-metric) coming from the Euclidean metric on each individual unit equilateral triangle has length element given by $$ds_{S}=|\psi_S|^{1/6},$$ and area element $$|\psi_S|^{1/3}.$$ Distances in this metric shall be denoted by $d_{S}(\cdot,\cdot)$. 

\subsection{Canonical cover}\label{8.3}

A triangulated surface $S$ has an associated canonical holomorphic differential on an appropriate $6$-degree branched cover that we construct as follows. We cover $S\setminus V_{\neq 6}(S)$ (which is also $S\setminus \{\text {zeros and poles of }\psi_S\}$) by open balls $\{U_i\}$ with transition functions $f_{i,i'}$ on $U_i\cap U_{i'}$ whenever the latter is nonempty. Let $\zeta=e^{\pi i/3}$. To each $U_i$, we associate $U_{i,1},...,U_{i,6}$, each consisting of the pair $(U_i,\omega_{i,j})$ where $$\omega_{i,j}=\zeta^j\psi_S^{1/6}$$ for some branch of $\psi_S^{1/6}$ on $U_i$. We glue $U_{i,j}$ and $U_{i',j'}$ if $$U_i\cap U_{i'}\neq \emptyset$$ and $$f_{i,i'}^*\omega_{i,j}=\omega_{i',j'}$$ on $U_i\cap U_{i'}$. Compactifying, we obtain a (possibly disconnected) degree $6$ branched cover $$F:\widetilde{S}\to S$$ such that $$F^*(\psi_S)=\widetilde{\phi}_S^6$$ for some holomorphic $1$-form $\widetilde{\phi}_S$ on $\widetilde{S}$. Furthermore, the surface $\widetilde{S}$ is equipped with a canonical degree $6$ automorphism $$A:\widetilde{S}\to \widetilde{S}$$ defined by $$A:U_{i,j}\to U_{i,j'}$$ is the canonical identification if $$\omega_{i,j}=\zeta \omega_{i,j'}.$$ The surface $S$ may be recovered from the cover by quotienting by $A$, i.e. $$S=\widetilde{S}/A.$$

\begin{prop}\label{covers} Let $S$ be a locally bounded $\leq T$-triangle triangulated surface of genus at most $g$ with $$|V_{\neq 6}(S)|\leq m.$$ Then $\widetilde{S}$ has at most six connected components that we label $\widetilde{S}^1,...,\widetilde{S}^6$ (where some components may be empty). For each $i\in \{1,...,6\}$, $\widetilde{S}^i$ is a locally bounded $\leq 6T$-triangle combinatorial translation surface of genus at most $6g+5m$.
\end{prop}

\begin{proof} Since $\widetilde{S}$ is a degree $6$ branched cover of $S$, $\widetilde{S}$ has at most six connected components. If $\widetilde{S}^i$ is a connected component, then $F|_{\widetilde{S}^i}$ is a degree $d$ branched cover for some $d\leq 6$. Denote by $$c_1,...,c_{n'}$$ the critical points on $\widetilde{S}^i$ of this covering map, and $$p_1,...,p_n$$ the branch points on $S$. Since branch points must be vertices of $S$ of degree not equal to $6$, $$n\leq m.$$ Since each branch point has a preimage which is a critical point, $$n'\geq n.$$ By the Riemann-Hurwitz formula, $$2\genus(\widetilde{S}^i)-2+n'= d(2g-2)+dn.$$ Thus $\widetilde{S}^i$ has genus at most $6g+5m$. Since all branch points of $F$ are vertices of $S$, the pullback of the triangulation $S$ under $F$ gives a triangulation of $\widetilde{S}^i$ with corresponding $1$-form $\widetilde{\phi}_S$ by construction. Hence $\widetilde{S}^i$ is a combinatorial translation surface. Since $$F:\widetilde{S}\to S$$ is a degree $6$ branched cover, $\widetilde{S}^i$ has at most $6T$ triangles. 

Finally, we claim that $\widetilde{S}^i$ is locally bounded. To see this, note that if $x$ is a vertex of the triangulation of $\widetilde{S}^i$, then $F(x)$ is a vertex of $S$. If $F(x)$ is not a branch point, then $$\deg x=\deg (F(x))\leq 7$$ as $S$ is locally bounded. If $F(x)$ is a branch point, then the ramification index of $x$, denoted $e(x)$, is at most $6$, so $$\deg x=e(x)\deg (F(x))\leq 42.$$ Thus $\widetilde{S}^i$ satisfies condition 1 in \cref{lb tri surface} definition of local boundedness for combinatorial translation surfaces. Next we show condition 2 in \cref{lb tri surface}. Since $S$ is locally bounded, we also have a triangulation $S_{\lb}$ of $S$ by equilateral triangles of side length $5$ in the flat metric, such that $S$ is a $5$-subdivision of $S_{\lb}$. The pullback of the triangulation $S_{\lb}$ under $F$ gives a triangulation $\widetilde{S}^i_{\lb}$ of the surface $\widetilde{S}^i$ wherein each triangle has side length $5$, such that the triangulation $\widetilde{S}^i$ is the $5$-subdivision of the triangulation $\widetilde{S}^i_{\lb}$. Thus $\widetilde{S}^i$ is a locally bounded combinatorial translation surface.
\end{proof}

Given such a canonical cover, we can reconstruct the original surface up to a factor.

\begin{lemma}\label{recover surface from cover} Let $\widetilde{S}'$ be a locally bounded $\leq 6T$-triangle combinatorial translation surface of genus at most $6g+m$. There are at most $CT$ number of 
triangulated surfaces $S$ for which $\widetilde{S}'$ is a component of the canonical cover of $S$. Here, $C$ is a universal constant.
\end{lemma}

\begin{proof} The automorphism on the canonical cover of $S$ determines a degree $1$, $2$, $3$ or $6$ automorphism $$A':\widetilde{S}'\to \widetilde{S}'.$$ Since $A'$ preserves $\psi_{\widetilde{S}'}$, $A'$ is a simplicial isomorphism. Furthermore, $A'$ is determined by its value on one triangle. Thus there are at most $CT$ number of choices for $A'$. Finally, $$S\simeq \widetilde{S}'/A',$$ so there are at most $CT$ number of choices for $S$.
\end{proof}

\subsection{Combinatorics of branched $6$-covers}\label{8.4}

In this section, we enumerate the number of degree $6$ branched covers of a fixed surface with fixed branched points. 

\begin{lemma}\label{combinatorial 6-covers} Let $S_g$ be a topological surface of genus $g$ and $P\subset S$ a nonempty set of $n$ marked points on $S_g$. There are at most $C^{g+n}$ choices of branched $6$-covers $\widetilde{S}_g\to S_g$ (up to isomorphism of branched covers), such that the branch points are contained in the set $P$. 
\end{lemma}

\begin{proof} Cut $S$ along elements of $H_1(S_g,P,\mathbb{Z})$ until we obtain $S'_g$, a simply connected surface with boundary such that $P\subset \partial S'_g$. The boundary $\partial S'_g$ may be decomposed into $4g+2n-2$ curves which come in pairs $$\gamma^{+,1},\gamma^{-,1},...,\gamma^{+,{2g+n-1}},\gamma^{-,{2g+n-1}}$$ that are glued together to form $S'_g$. Here, each $\gamma^{\pm,i}$ represents one element of $H_1(S_g,P,\mathbb{Z})$. Any degree $6$ branched cover $\widetilde{S}_g$ with branched points in the set $P$ necessarily admits a decomposition into isomorphic copies of $S'_g$ denoted by ${S'}^1_g,...,{S'}^6_g$. Boundary curve $\gamma^{+,i}_j$ (for $j\in \{1,...6\}$ must necessarily be glued to $\gamma^{-,i}_{j'}$ for some $j'\in \{1,...,6\}$) under this decomposition of $\widetilde{S}_g$. Each decomposition corresponds to an isomorphism class of branched $6$-covers $\widetilde{S}_g\to S_g$. For a fixed $i$, there are $C$ choices to glue all the $\gamma^{+,i}_j$ and $\gamma^{-,i}_{j'}$, and $i$ ranges from $1$ to $2g+n-1$. Hence, the total number of gluing choices (and total number of branched covers) is bounded by $C^{g+n}$. 
\end{proof}

For $n>1$, we denote by $\mathcal{F}_{g,n}$ the set of topological branched covers enumerated above.

\subsection{Mean value property and 6-differentials}\label{8.2}

In this section, we show a rough mean value property for $6$-differentials arising from triangulated surfaces, which will be useful in \cref{8.5}.

\begin{lemma}\label{mean value theorem 6 diff} Let $S$ be a triangulated surface and let $X=\Phi(S)$. Suppose $v,w\in V(S)$, $v\neq w$ and let $\mathbb{D}\subset X$ a conformal identification of $\mathbb{D}$ with a subset of $X$ containing $v$ and $w$ such that $0$ is identified with $v$. Suppose further that $$|w|\leq r<3/4$$ in $\mathbb{D}$. Then $$\int_{\mathbb{D}}|\psi_S|^{1/3}\geq Cr^{-2}.$$
\end{lemma}

\begin{proof} Denote by $C_\alpha$ the circle of radius $\alpha$ around $0$ in $\mathbb{D}$. Note that for $r\leq \alpha\leq 1$ the length of $C_\alpha$ in the $d_S$ metric must be at least $1$. Thus, writing $$\psi=g(z)dz^6$$ for a meromorphic function $g$ on $\mathbb{D}$, we have that $$|g(z_0)|^{1/6}\geq Cr^{-1}$$ for some $z_0\in C_r$. Let $$a=g(z_0)^{1/6}.$$ Define $$\psi'=\frac{g(z)}{(z-z_0)^6}dz^6,$$ a meromorphic $6$-differential on $\mathbb{D}$. On a small neighborhood around $z_0$, $$\psi'=\phi'^6$$ for a $1$-form $\phi'$. Locally, $$\phi'=\frac{g(z)^{1/6}}{(z-z_0)}dz$$ (note that $g(z)^{1/6}$ makes sense on a small neighborhood of $z_0$). Therefore the residue of $\phi'$ at $z_0$ is $a$. Let $f:\widetilde{X}\to X$ be the canonical $6$-cover associated to $\psi$ (constructed in \cref{8.1}). This means that $$f^*(\psi')=\widetilde{\phi}'^6$$ for a meromorphic $1$-form $\phi$ defined on $f^{-1}(\mathbb{D})$. The $1$-form $\widetilde{\phi}'$ is holomorphic except for a pole at $f^{-1}(z_0)$. On a neighborhood of $f^{-1}(z_0),$ $$f^{-1}(\phi')=\widetilde{\phi}'$$ so the residue of $\widetilde{\phi}'$ at $f^{-1}(z_0)$ is $a$. Since $r<3/4$, for $7/8\leq \alpha\leq 1$ we have $$\int_{f^{-1}(C_\alpha)}\widetilde{\phi}'=2\pi i a.$$ So $$\int_{f^{-1}(C_\alpha)}|\widetilde{\phi}'| \geq Ca.$$ Pushing forward this integral to $\mathbb{D}$ we obtain 
\begin{align*}\int_{C_\alpha}|g(z)|^{1/6}|dz| &\geq  C\int_{C_\alpha}\frac{|g(z)|^{1/6}}{|z-z_0|}|dz|\\&\geq Ca,
\end{align*}
since for $z\in C_\alpha$, $|z-z_0|\geq 1/8$. By Jensen's inequality, we have 
\begin{align*}\int_{C_{\alpha}}|g(z)|^{1/3}|dz|&\geq C\left(\int_{C_\alpha}|g(z)|^{1/6}|dz|\right)^2\\&\geq C a^2.
\end{align*} Then, 
\begin{align*}\int_{\mathbb{D}}|g(z)|^{1/3}|dz|^2&\geq\int_{\alpha=7/8}^{1}\int_{C_{\alpha}}|g(z)|^{1/3}|dz|d\alpha\\&\geq C\int_{\alpha=7/8}^{1}a^2d\alpha\\&\geq C a^2\\&\geq Cr^{-2}
\end{align*} as desired.
\end{proof}

We also have a quantitative version: 

\begin{lemma}\label{mean value theorem 6 diff quant} Let $S$ be a triangulated surface and let $X=\Phi(S)$. Denote by $\rho_X$ the hyperbolic metric on $X$, and $\rho_{\mathbb{D}}$ the Poincare metric on $\mathbb{D}$. Suppose $$v,w\in V(S)$$ with $v\neq w$. Let $U\subset X$ be a region containing $v$ and $w$ and $$f:U\to B_{\rho_{\mathbb{D}}}(0,s)$$ a $K$-bi-Lipschitz diffeomorphism (with respect to metrics $\rho_X$ on $U$ and $\rho_{\mathbb{D}}$ on $B_{\rho_{\mathbb{D}}}(0,s)$) such that $f(v)=0$. Let $r<3/4$ and suppose that $$d_{\rho_X}(v,w)\leq (1/100)rs/K.$$ Then $$\int_U|\psi_S|^{1/3}\geq Cr^{-2}$$ where $C$ is a universal constant.
\end{lemma}

\begin{proof} Since $$f:U\to B_{\rho_{\mathbb{D}}}(0,s)$$ is a $K$-bi-Lipschitz map, $U$ contains $$B_{\rho_X}(v,s/K),$$ the ball of radius $s/K$ around $v$. Note that $$w\in B_{\rho_X}(v,s/K)$$ too. This ball may be isometrically identified with $$\mathbb{D}_{\tanh(s/2K)}=\{z\in \mathbb{D}||z|\leq \tanh(s/2K)\}$$ (with $v$ identified with $0$) equipped with the restriction of the Poincare metric $\rho_{\mathbb{D}}$. Now,
\begin{align*}d_{\rho_X}(v,w)&\leq (1/100)rs/K\\&\leq 2\tanh^{-1}(r\tanh(s/(2K)))
\end{align*} since $s\leq 1$, $r< 3/4$ and $K\geq 1$. So in $\mathbb{D}_{\tanh(s/(2K))}$ (under our isometric identification), $$|w|\leq r\tanh(s/(2K)).$$ By \cref{mean value theorem 6 diff}, $$\int_{B_{\rho_X}(v,s/K)}|\psi|^{1/3}\geq Cr^{-2}.$$ Since $B_{\rho_X}(v,s/K)\subset U$, the lemma statement follows.
\end{proof}

\subsection{Location of branch points}\label{8.5} Fix an arbitrary constant $r_0>0$. Let $X\in \mathcal{T}_g$. In this section, to each locally bounded triangulated surface whose conformal class lies in $B_{d_T}(X,r_0)$, we associate combinatorial data that is a discrete measure of where the vertices of degree not equal to $6$ are located. 

Denote by $\rho_X$ the hyperbolic metric on $X$. We take hyperbolic disks $U_1,...,U_N$, $V_1,...,V_N$ and $W_1,...,W_N$ as in \cref{covering lemma}. By \cref{tri systole} and \cref{covering lemma}, we may assume $N\leq CT$. 

\begin{definition}\label{construction of balls}  For each $i\in \{1,...,N\}$ and $M\geq 1$, define $$\{W^M_{i,j}\}|_{j\in \{1,...,M\}}$$ to be a collection of hyperbolic balls on $X$ satisfying the following properties:

\begin{enumerate}

\item For all $j,j'\in \{1,...,M\}$, $$\radius(W^M_{i,j})=\radius(W^M_{i,j'})\leq CM^{-1/2}\radius(W_i)$$ where $C$ is a universal constant, 

\item $$W^M_{i,j}\subset V_i$$ and $$\cent(W^M_{i,j})\in W_i,$$

\item the $$\{B_{\rho_X}(\cent(W^M_{i,j}),\radius(W^M_{i,j})/2)\}_{j\in \{1,...,M\}}$$ cover $W_i$ and

\item each $x\in U_i$ is contained in at most $C$ of the $W^M_{i,j}$ for a universal constant $C$.

\end{enumerate}

\end{definition}

Such a collection may be constructed by taking $\sim M^{-1/2}\radius(W_i)$ radius balls around a maximal $\sim (1/8)M^{-1/2}\radius(W_i)$-separated set on $W_i$.

\begin{definition}\label{D} Define $\mathcal{D}$ to be the set of all values of the following data of $(I,L,W_L)$:

\begin{enumerate}
\item subset $$I\subset \{1,...,N\}$$ with $$|I|\leq \alpha m$$

\item function $$L:I\to \mathbb{N}$$ such that $$\sum_{i\in I} L(i)\leq \alpha T$$ and

\item subset $$W_L\subset \bigcup_i \{W^{\kappa L(i)}_{i,j}\}|_{j\in \{1,...,\kappa L(i)\}}$$ with $$|W_L|\leq \alpha m.$$
\end{enumerate}
\end{definition}

Here, $\alpha, \kappa>1$ are sufficiently large universal constants  which will be chosen in the proofs of \cref{branch points} and \cref{S gives data} later in this section.

\begin{lemma}\label{branch points} We have $$|\mathcal{D}|\leq (\kappa T/m)^{C\alpha m}.$$ 
\end{lemma}

\begin{proof} The number of subsets $I$ is bounded above by $(T/m)^{C\alpha m}$. Given $I$, the number of functions $L$ is bounded above by $(T/m)^{C\alpha m}$. Given $I$ and $L$, $$\left|\bigcup_{i\in I} \{W^{L(i)}_{i,j}\}|_{j\in \{1,...,\kappa L(i)\}}\right|\leq \alpha \kappa T.$$ Trivially, $m\leq 3T$. Thus for $\kappa\geq 10$, the number of subsets of $$\bigcup_i \{W^{L(i)}_{i,j}\}|_{j\in \{1,...,\kappa L(i)\}}$$ with cardinality at most $\alpha m$ is bounded above by $(\kappa T/m)^{C\alpha m}$. Therefore $$|\mathcal{D}|\leq (\kappa T/m)^{C\alpha m},$$ as desired.
\end{proof}

For all $Y\in B_{d_T}(X,r_0)$, by \cref{qc to lipschitz}, we choose a $e^{\xi r_0}$-quasiconformal map $$f:X\to Y$$ that is $e^{\xi r_0}$-bi-Lipschitz with respect to the hyperbolic metrics $\rho_X$ on $X$ and $\rho_Y$ on $Y$. Here, $\xi$ is a universal constant from \cref{qc to lipschitz}. 

\begin{lemma}\label{S gives data} Let $$S\in \textstyle\Tri{g}{T,m}$$ and suppose $$Y=\Phi(S)\in d_T(X,r_0).$$ Then there exist an element $$(I,L,W_L)\in \mathcal{D}$$ associated to $S$ satisfying

\begin{enumerate}

\item $i\in I$ if and only if $W_i$ contains a vertex in $f^{-1}(V_{\neq 6}(S))$,

\item each ball $$B\in W_L$$ contains a unique vertex in $f^{-1}(V_{\neq 6}(S))$, and this vertex lies in $$B_{\rho_X}(\cent(B),\radius(B)/2),$$

\item each vertex in $f^{-1}(V_{\neq 6}(S))$ is contained in a unique $B\in W_L$ and

\item each point $x\in X$ is contained in at most $C$ of the $W\in W_L$.

\end{enumerate} 
\end{lemma}

\begin{proof}[Proof of \cref{S gives data}] We let $I$ be the subset of $i\in \{1,...,N\}$ for which $f(W_i)$ contains an element of $V_{\neq 6}(S)$. Condition 1 in the statement of \cref{S gives data} is automatically satisfied. We also let $$L(i)=\left\lceil \int_{f(U_i)}|\psi_S|^{1/3}\right\rceil.$$ Note that $$|I|\leq |V_{\neq 6}(S)|\leq m.$$ 

Now, $$\int_Y |\psi_S|^{1/3}\leq CT.$$ Also, each point of $Y$ is contained in at most $C$ of the $f(U_i)$, by condition 4 of \cref{covering lemma}. Therefore, $$\sum_{i\in I} L(i)\leq \alpha T$$ for a sufficiently large constant $\alpha$.

Finally, we let $W_L$ be the subset of $B$ in $$\bigcup_{i\in I}\{W_{i,j}^{L(i)}\}_{j\in 1,...,\kappa(L(i))}$$ which satisfy the property that $$B_{\rho_X}(\cent(B),\radius(B)/2)$$ contains a vertex in $f^{-1}(V_{\neq 6}(S))$. 

We now show the uniqueness part of condition 2 in the statement of \cref{S gives data}. Suppose the contrary; that there exists $B\in W_L$ containing $f^{-1}(x)$ and $f^{-1}(y)$ for vertices $x,y\in V_{\neq 6}(S)$. Then $$f^{-1}(x),f^{-1}(y)\in V_i$$ for some $i\in \{1,...,N\}$.  For sufficiently large $\kappa$, $$\radius(B)\leq \radius(W_i),$$ implying $$f^{-1}(y)\in B_{\rho_X}(f^{-1}(x),2\radius(W_i))\subset U_i.$$ Now, \cref{mean value theorem 6 diff quant} applied to $$f^{-1}:f(B_{\rho_X}(f^{-1}(x),2\radius(W_i)))\to B_{\rho_X}(f^{-1}(x),2\radius(W_i))$$ implies $$d_{\rho_Y}(x,y)\geq CL(i)^{-1/2}\radius(U_i)/e^{\xi r_0}.$$ Thus 
\begin{align*}d_{\rho_X}(f^{-1}(x),f^{-1}(y))&\geq CL(i)^{-1/2}\radius(U_i)/e^{2\xi r_0}\\&\geq CL(i)^{-1/2}\kappa^{-1/2}\radius(U_i)\\&\geq CL(i)^{-1/2}\kappa^{-1/2}\radius(W_i)\\&\geq 2\radius(B) && \text{condition 2 in \cref{construction of balls}}
\end{align*} for $\kappa$ sufficiently large. Hence we have a contradiction to the assumption that $$f^{-1}(x),f^{-1}(y)\in B.$$ This completes the proof of condition $2$ in the statement of \cref{S gives data}.

Now, condition 4 in the statement of \cref{S gives data} follows from condition 4 of \cref{covering lemma} along with condition 4 in \cref{construction of balls}. Condition 3 in the statement of \cref{S gives data} can be ensured by deleting unnecessary elements of $W_L$.
\end{proof}

For $$S\in \textstyle\Tric{g}{T,m}$$ such that $$Y=\Phi(S)\in B_{d_T}(X,r_0),$$ we choose an element $$(I,L,W_L)\in \mathcal{D}$$ as in \cref{S gives data}. We label this particular choice of element in $\mathcal{D}$ by $$(I_S,L_S,W_{L_S}).$$ 

\subsection{Quasiconformal map between surfaces with marked points}\label{8.6} In this section, we associate to triples in $\mathcal{D}$ a set of marked points on $X$. Then, given a triangulated surface $S$ whose conformal class lies near $X$, we construct a quasiconformal map from $X$ to $S$ which take the marked points on $X$ to $V_{\neq 6}(S)$.

Let $(I,L,W_L)\in \mathcal{D}$ and suppose it arises as the triple associated to a triangulated surface. Enumerate the elements $B_1,....,B_N$ in $W_L$. For each $i\in \{1,...,N\}$, choose $$x_i\in B_{\rho_X}(\cent(B_i),\radius(B_i)/2)$$ so that no element of $W_L$ contains $x_i$. Such a choice can be made because of \cref{S gives data} and our assumption that $(I,L,W_L)$ arises as the triple associated to a triangulated surface. Let $$O=\{x_i\}|_{i\in \{1,...,N\}}.$$

Suppose $$S\in \textstyle\Tri{g}{T,m}$$ and $$\Phi(S)\in B_T(X,r_0).$$ Let $$f:X\to S$$ be the associated $e^{\xi r_0}$-quasiconformal, $e^{\xi r_0}$-bi-Lipschitz map. Suppose $$(I_S,L_S,W_{L_S})=(I,L,W_L).$$ Let $$P=V_{\neq 6}(S).$$ There is a natural identification $$\theta:O\to P$$ which sends $x_i\in O$ to the unique $p_i\in V_{\neq 6}(S)$ satisfying $q_i=f^{-1}(p_i)\in B_i$ (unique by \cref{S gives data}). 

\begin{lemma}\label{branch point distance comparison} There is a $C$-quasiconformal map $$F:X\to S$$ isotopic to $f$ such that $$F(O)=P$$ and $$F|_O=\theta.$$ Here, $C$ is a universal constant.
\end{lemma}

First, we prove an elementary lemma.

\begin{lemma}\label{qc diffeo cor} For all $z_1,z_2\in \mathbb{D}$ with $$|z_1|,|z_2|\leq 3/4,$$ there exists a diffeomorphism $$f_{z_1,z_2}:\mathbb{D}\to \mathbb{D}$$ satisfying the following conditions:

\begin{enumerate}

\item $f_{z_1,z_2}|_{\partial \mathbb{D}}$ is the identity map,

\item $Df_{z_1,z_2}|_{\partial \mathbb{D}}$ is the identity map,

\item $f_{z_1,z_2}(z_1)=z_2$ and

\item $f_{z_1,z_2}$ is $C$-quasiconformal.

\end{enumerate}
Here, $C$ is a universal constant.
\end{lemma}

\begin{proof} We first treat the case where $z_1=0$. Denote by $f_{0,3/4}$ any chosen diffeomorphism $\mathbb{D}\to \mathbb{D}$ satisfying conditions (1), (2) and (3) above for $z=3/4$. Let $C$ be the quasiconformal dilatation of $f_{0,3/4}$. Note that $C$ is finite since $f$ is defined on a compact set. Now, suppose $z\in \mathbb{R}$ such that $z\leq 3/4$. The function 
\begin{equation*} f_{0,z}(w)=
\begin{cases}
 (4/3)z f_{0,3/4}((3/4)z^{-1}w) & \text{ if } |w|\leq (4/3)z\\
 w & \text{ otherwise}
\end{cases}
\end{equation*}
satisfies the conditions in the lemma statement. Suppose next that $z\in \mathbb{D}$ with $|z|\leq (3/4)$. We have already constructed $f_{0,|z|}$, since $|z|\in \mathbb{R}$. We take $$f_{0,z}(w)=\displaystyle\frac{z}{|z|}f_{0,|z|}\left(\left(\displaystyle\frac{z}{|z|}\right)^{-1}w\right).$$ Now removing the assumption that $z_1=0$, we simply take $$f_{z_1,z_2}=f_{0,z_2}\circ f^{-1}_{0,z_1},$$ and note that $f_{z_1,z_2}$ satisfies the desired properties.
\end{proof}

\begin{proof}[Proof of \cref{branch point distance comparison}] 

It suffices to construct a $C$-quasiconformal diffeomorphism $$F':X\to X$$ isotopic to the identity such that $$F'(x_i)=q_i.$$ Then, taking $$F= f\circ F'$$ gives the desired map in the lemma statement. 

To do this, recall that $B_1,...,B_N$ are the elements of $W$. By \cref{covering lemma} and condition 2 in \cref{construction of balls}, the $B_i$ are hyperbolic disks of maximal radius $\arcsinh(1)$. Recall that each $B_i$ contains exactly one element of $O$ (which is $x_i$), and exactly one element of $f^{-1}(P)$ (which is $q_i$). By condition 2 in \cref{S gives data} along with the upper bound on $\radius(B_i)$, under a conformal identification of $B_i$ with $\mathbb{D}$ sending its center to $0$, both $x_i$ and $q_i$ lie inside $$\mathbb{D}_{3/4}=\{z\in \mathbb{D}||z|< 3/4\}.$$ Thus there exists $f_{x_i,q_i}$ satisfying the conditions in the statement of \cref{qc diffeo cor}. The map $f_{x_i,q_i}$ extends to a $C$-quasiconformal map $$f^i:X\to X$$ that is the identity outside $B_i$. Composing all such maps $f^i$ (for $i\in \{1,...,N\}$) in an arbitrary order we obtain a map $$F':X\to X$$ such that $$F'(x_i)=q_i.$$ Now, by condition $3$ of \cref{S gives data}, the set $B_1,...,B_N$ satisfies the property that any $x\in X$ is contained in at most $C$ of the $B_i$. Therefore $F'$ is $C$-quasiconformal.
\end{proof}

\subsection{Distance between covers in higher dimensional moduli space}\label{8.6 1/2} In this section, we show that if a triangulated surface $S$ has conformal class close to $X$, then a component of its canonical cover is close to a branched cover of $X$ in a higher dimensional moduli space.

Let $$(I,L,W_L)\in \mathcal{D}$$ and assume that $(I,L,W_L)$ arises as the triple associated to a triangulated surface. Recall that $$O\subset X$$ is a subset, depending on the triple $(I,L,W_L)$, generated in \cref{8.6}.

The set $\mathcal{F}_{g,m}$ enumerates the combinatorial $6$-covers over $X$ whose branch points are a subset of $O$. That is, associated to $\mathfrak{f}\in \mathcal{F}_{g,m}$ is a topological branched cover $\widetilde{X}_{\mathfrak{f}}\to X$. The holomorphic structure of $X$ determines a holomorphic structure on $\widetilde{X}_{\mathfrak{f}}$. Note that $\widetilde{X}_{\mathfrak{f}}$ may be disconnected. Choose a connected component and label it $\widetilde{X}'_{\mathfrak{f}}$.

Suppose $$S\in \textstyle\Tri{g}{T,m}$$ and $$\Phi(S)\in B_T(X,r_0).$$ Let $f:X\to S$ be the associated $e^{\xi r_0}$-quasiconformal, $e^{\xi r_0}$-bi-Lipschitz map. Suppose $$(I_S,L_S,W_{L_S})=(I,L,W_L).$$

The composition $$\widetilde{S}\to S\xrightarrow{F^{-1}} X$$ is a topological branched cover over $X$ having $O$ as the set of branch points. Hence, there is an element $\mathfrak{f}\in \mathcal{F}_{g,m}$ such that the diagram
$$\begin{tikzcd}
\widetilde{X}_\mathfrak{f} \arrow{r} \arrow[swap]{d} & \widetilde{S} \arrow{d} \\%
X \arrow{r}{F}& S
\end{tikzcd}$$
commutes and the top vertical map (henceforth denoted $\widetilde{F}$) is a homeomorphism. By \cref{branch point distance comparison}, $F$ is $C$-quasiconformal. Also, the vertical maps are conformal. Therefore $\widetilde{F}$ is $C$-quasiconformal. Let $\widetilde{S}'$ be the connected component of $\widetilde{S}$ determined by $F(\widetilde{X}'_\mathfrak{f})$. 

\begin{lemma}\label{distance comparison} We have, $$d_{T}(\widetilde{X}'_\mathfrak{f},\Phi(\widetilde{S}'))\leq C.$$
\end{lemma}

\subsection{Proof of \cref{tric vs tranc}}\label{8.7} 

Let $m>0$ and $X\in \mathcal{M}_g$. We must count the number of locally bounded triangulated surfaces lying in $B_{d_T}(X,r_0)$ that have at most $m$ vertices of degree not equal to $6$. Any such triangulated surface has associated to it a triple $(I,L,W)$ satisfying the conditions in \cref{S gives data}. By \cref{branch points}, there are at most $(T/m)^{Cm}$ such triples. 

Fixing a triple $(I,L,W)$, it suffices to count triangulated surfaces with this particular triple as its associated triple. By \cref{covers}, any such triangulated surface has associated to it a topological branched cover in $\mathcal{F}_{g,\leq m}$. By \cref{combinatorial 6-covers}, there are at most $C^{g+m}$ such topological branched covers. 

It suffices to count triangulated surfaces with a fixed underlying topological branched cover. In this situation, by \cref{distance comparison}, a connected component of the canonical cover must be closed to a fixed branched cover of $X$ in a higher dimensional moduli space. By \cref{covers}, this connected component is a combinatorial translation surface. By \cref{recover surface from cover} it suffices to count combinatorial translation surfaces lying in a ball in the higher genus moduli space. Thus we obtain $$\Ntric(T,g,m,r_0)\leq C^g(T/m)^{Cm}\Ntranc(6T_i,6g+5m,C).$$ Combining with \cref{Teichmuller bound 3} gives \cref{tric vs tranc}.

\subsection{Proof of \cref{upper bound}}\label{8.8} Recall that for $g\geq 2$, $$\Ntri(T,g,m,r)\leq C^{m+g} \Ntric(\sigma T,g,\mu (m+g),r+C)$$ by \cref{trim vs tricm}, $$\Ntric(T,g,m,r)\leq C^{(1+r)g}(T/m)^{Cm}\Ntranc(6T,6g+5m,r+C)$$ by \cref{tric vs tranc} and $$\Ntranc(T,g,r)\leq (T/g)^{C(1+r)g}\sum_{\substack{n\leq (1/100)g\\ T_1+...+T_n\leq 2T\\ g_1+...+g_n\leq \mu^{-1}(1/100)g\\g_1,...,g_n\geq 2\\ m_1+...+m_n\leq \mu^{-1}(1/100)g}} \prod_{i=1}^n \Ntri(T_i,g_i,m_i, r+C)$$ by \cref{tranc vs tri}.

Choosing $r_1$ sufficiently small (independent of $T$ and $g$), we may apply \cref{Teichmuller bound 3} to obtain 
\begin{equation}\label{eq760.1}\Ntri(T,g,m,r_1)\leq C^{m+g}\Ntric(\sigma T,g,\mu (m+g), r_1),
\end{equation} \begin{equation}\label{eq760.2}\Ntric(T,g,m,r_1)\leq C^g(T/m)^{Cm}\Ntranc(6T,6g+5m,r_1)
\end{equation} and \begin{equation*}\Ntranc(T,g,r_1)\leq (T/g)^{Cg}\sum_{\substack{n\leq (1/100)g\\ T_1+...+T_n\leq 2T\\ g_1+...+g_n\leq \mu^{-1}(1/100)g\\g_1,...,g_n\geq 2\\ m_1+...+m_n\leq \mu^{-1}(1/100)g}} \prod_{i=1}^n \Ntri(T_i,g_i,m_i, r_1).
\end{equation*}

Together, these three bounds give
\begin{multline}\label{eq761}
\Ntranc(T,g,r_1)\\
\begin{aligned} &\leq (T/g)^{Cg}\sum_{\substack{n\leq (1/100)g\\ T_1+...+T_n\leq 2T\\ g_1+...+g_n\leq \mu^{-1}(1/100)g\\g_1,...,g_n\geq 2\\ m_1+...+m_n\leq \mu^{-1}(1/100)g}} \prod_{i=1}^n \Ntri(T_i,g_i,m_i, r_1)\\
&\leq (T/g)^{Cg}\sum_{\substack{n\leq (1/100)g\\ T_1+...+T_n\leq 2T\\ g_1+...+g_n\leq (1/100)g\\g_1,...,g_n\geq 2\\ m_1+...+m_n\leq (1/50)g}} \prod_{i=1}^n C^{m_i+g_i} \Ntric(\sigma T_i, g_i, m_i,r_1)\\&\leq (T/g)^{Cg}\sum_{\substack{n\leq (1/100)g\\ T_1+...+T_n\leq 2T\\ g_1+...+g_n\leq (1/100)g\\g_1,...,g_n\geq 2,\\ m_1+...+m_n\leq (1/50)g\\ m_1,...,m_n\geq 1}} \prod_{i=1}^n(T_i/m_i)^{Cm_i}\Ntranc(6\sigma T_i, 6g_i+5m_i, r_1)\\&\leq
(T/g)^{Cg}\sum_{\substack{n\leq (1/100)g\\ T_1+...+T_n\leq 12\sigma T\\ g_1+...+g_n\leq g/4}} \prod_{i=1}^n \Ntranc(T_i, g_i, r_1).
\end{aligned}
\end{multline}
where the last inequality follows from renaming variables along with the arithmetic mean-geometric mean inequality.

We claim that $$\Ntranc(T,g,r_1)\leq \Theta^{g}(T/g)^{\Omega g}$$ for universal constants $\Omega$ and $\Theta$ that we choose later. We prove this by induction. The quantity $\Ntranc(T, 4, r_1)$ is bounded by a polynomial in $T$. (See \cite{EO01} for more precise results.) To show the induction hypothesis, assume $$\Ntranc(T,g',r_1)\leq \Theta^{g'}(T/g')^{\Omega g'}$$ for all $g'\leq g$. We have, 
\begin{align*}\Ntranc(T,g,r_1)&\leq (T/g)^{Cg}\sum_{\substack{n\leq (1/100)g\\ T_1+...+T_n\leq 12\sigma T\\ g_1+...+g_n\leq g/4}} \prod_{i=1}^n \Ntranc(T_i, g_i, r_1) && \text{\cref{eq761}}\\&\leq (T/g)^{Cg}\sum_{\substack{n\leq (1/100)g\\ T_1+...+T_n\leq 12\sigma T\\ g_1+...+g_n\leq g/4}} \prod_{i=1}^n \Theta^{g_i}(T_i/g_i)^{\Omega g_i} \\&\leq (T/g)^{Cg}\sum_{\substack{n\leq (1/100)g\\ T_1+...+T_n\leq 12\sigma T\\ g_1+...+g_n\leq g/4}} \Theta^{g/4}(48\sigma T/g)^{\Omega g/4} && \substack{\text{\normalsize{arithmetic mean-}}\\ \text{\normalsize{geometric mean}}}\\&\leq (T/g)^{Cg}\Theta^{g/4}(48\sigma T/g)^{\Omega g/4}
\end{align*}
where the last inequality follows since the sum contains at most $(T/g)^{Cg}$ terms. Noting that $T/g\geq 2$, there exists a choice of constants $\Theta$ and $\Omega$ such that $$(T/g)^{Cg}\Theta^{g/4}(48\sigma T/g)^{\Omega g/4}\leq \Theta^g(T/g)^{\Omega g}.$$  This completes the induction step. Therefore, we have 
\begin{equation}\label{eq762}
\begin{aligned}\Ntranc(T,g,r_1)&\leq \Theta^g(T/g)^{\Omega g}\\&\leq C^T.
\end{aligned}
\end{equation} Finally,
\begin{align*}\Ncomb(T,g,r_1)&\leq \Ntri(T,g,3T,r_1)\\&\leq C^T \Ntric(\sigma T,g,3\mu T,r_1) && \text{\cref{eq760.1}}\\&\leq C^T\Ntranc(6\sigma T,6g+15\mu T,r_1) && \text{\cref{eq760.2}}\\&\leq C^T && \text{\cref{eq762}}
\end{align*}
where the second inequality follows from the fact that a $T$-triangle triangulated surface has at most $3T$ vertices. Combining with \cref{Teichmuller bound 3} gives $$\Ncomb(T,g,r)\leq C^{T+rg},$$ which completes the proof of \cref{upper bound}.


\begin{thebibliography}{}

\bibitem{Ahl66} L. V. Ahlfors, Lectures on quasiconformal mappings, D. Van Nostrand Company Inc, Princeton, NJ, 1966.

\bibitem{BC86} E. A. Bender, E. R. Canfield, The asymptotic number of rooted maps on a surface, J. Combin. Theory
Ser. A 43 (1986) no. 2, 244–257.

\bibitem{Bis14} C. J. Bishop, True trees are dense, Invent. Math. 197, 433–452 (2014).

\bibitem{Bis15} C. J. Bishop, Constructing entire functions by quasiconformal folding, Acta Math. 214 (1) 1 - 60, 2015.

\bibitem{BM04} R. Brooks and E. Makover, Random construction of Riemann surfaces, J. Differential Geom. 68
(2004), no. 1, 121–157. 

\bibitem{Bus10} P. Buser, Geometry and spectra of compact Riemann surfaces, Birkhauser Boston, 2010.

\bibitem{BS94} P. Buser and P. Sarnak, On the period matrix of a Riemann surface of large genus. With an appendix by J. H.
Conway and N. J. A. Sloane. Invent. Math. 117 (1994), no. 1, 27–56.

\bibitem{BCP19a} T. Budzinski, N. Curien, B. Petri, Universality for random surfaces in unconstrained genus, Electron. J. Combin., 26 (4): Paper 4.2, 2019

\bibitem{BCP19b} T. Budzinski, N. Curien, B. Petri, The diameter of random Belyi surfaces, to appear, Alg. Geom. Top.

\bibitem{BL19} T. Budzinski, B. Louf, Local limits of uniform triangulations in high genus, Invent. Math.,  223, 1–47 (2021).

\bibitem{BCP21} T. Budzinski, N. Curien, B. Petri, On the minimal diameter of closed hyperbolic surfaces, Duke Math. J., 170 (2): 365-377, 2021.

\bibitem{BPS12} F. Balacheff, H. Parlier, S. Sabourau, Short loop decompositions of surfaces and the geometry of Jacobians, Geom. Funct. Anal. Vol. 22 (2012) 37–73.

\bibitem{CP12} W. Cavendish and H. Parlier, Growth of the Weil-Petersson diameter
of moduli space, Duke Math. J. 161 (2012), no. 1, 139–171.

\bibitem{Cur16} N. Curien. Planar stochastic hyperbolic triangulations. Probability Theory and Related
Fields, 165(3):509–540, 2016.

\bibitem{Dem09} J. P. Demailly, Complex Analytic and Differential Geometry, \url{https://www-fourier.ujf-grenoble.fr/~demailly/manuscripts/agbook.pdf}, 2009.

\bibitem{DE86} A. Douady and C. J. Earle, Conformally natural extension of homeomorphisms of the circle, Acta Math. 157: 23-48 (1986).

\bibitem{EO01} A. Eskin and A. Okounkov, Asymptotics of numbers of branched coverings of a torus and volumes of moduli spaces of holomorphic differentials, Invent. Math. (2001) 145-1, 59-103.

\bibitem{FKM13} A. Fletcher, J. Kahn and V. Markovic. The Moduli Space of Riemann Surfaces of Large Genus. Geom. Funct. Anal. 23, 867–887 (2013).

\bibitem{Gam06} A. Gamburd, Poisson-Dirichlet distribution for random Belyi surfaces, Ann. Probab., Volume 34, Number 5 (2006), 1827-1848.

\bibitem{Gar87} F. Gardiner, Teichm¨uller Theory and Quadratic Differentials, Wiley Interscience,
New York, 1987.

\bibitem{Gru01} S. Grushevsky, An explicit upper bound for Weil-Petersson volumes of the moduli spaces of punctured
Riemann surfaces, Math. Ann. 321 (2001)

\bibitem{GPY11} L. Guth, H. Parlier, and R. Young, Pants decompositions of random surfaces, Geom. Funct. Anal. (2011) 21: 1069.

\bibitem{Hei01} J. Heinonen, Lectures on analysis on metric spaces, Springer, New York, 2001.

\bibitem{Hub06} J. H. Hubbard, Teichm\"{u}ller theory and applications to geometry, topology, and dynamics. Volume
1: Teichm\"{u}ller theory, Matrix Editions, Ithaca, NY, 2006. 

\bibitem{Jen57} J. A. Jenkins, On the existence of certain extremal metrics. Ann. of Math. 66 (1957). 440453. 

\bibitem{Ker80} S. P. Kerckhoff, The asymptotic geometry of Teichmüller space, Topology 19 (1980), no. 1, 23–41.

\bibitem{Kob05} S. Kobayashi, Hyperbolic manifolds and holomorphic mappings, World Scientific Publishing, MA, 2005.

\bibitem{LeG07} J.-F. Le Gall, The topological structure of scaling limits of large planar maps, Invent. Math. 169
(2007), no. 3, 621–670.

\bibitem{LV73} O. Lehto and K. I. Virtanen, Quasiconformal mappings in the plane, Second edition, Springer-Verlag, 1973.

\bibitem{LW21} M. Lipnowski and A. Wright, Towards optimal spectral gaps in large genus, 2021. 

\bibitem{Mar18} C. Marzouk. Scaling limits of random bipartite planar maps with a prescribed degree
sequence. Random Structures and Algorithms, 2018.

\bibitem{McM00} C. McMullen, The Moduli Space of Riemann Surfaces is Kahler Hyperbolic, Annals of Mathematics, 151(1), second series, 327-357, 2000.

\bibitem{Mie13} G. Miermont. The Brownian map is the scaling limit of uniform random plane quadrangulations, Acta Math., 210(2):319–401, 2013.

\bibitem{Mir07a} M. Mirzakhani, Simple geodesics and Weil-Petersson volumes of moduli spaces of
bordered Riemann surfaces, Invent. Math., 167(1):179–222, 2007.

\bibitem{Mir07b} M. Mirzakhani, Weil-Petersson volumes and intersection theory on the moduli
space of curves, J. Amer. Math. Soc., 20(1):1–23 (electronic), 2007.

\bibitem{Mir10} M. Mirzakhani, On Weil-Petersson volumes and geometry of random hyperbolic
surfaces, Proceedings of the International Congress of Mathematicians.
Volume II, Hindustan Book Agency, New Delhi, 2010, pp. 1126–1145.
MR 2827834.

\bibitem{Mir13} M. Mirzakhani, Growth of Weil-Petersson volumes and random hyperbolic surfaces of large genus. J. Differential Geom., 94(2):267–300, 2013.

\bibitem{MZ15} M. Mirzakhani and P. Zograf. Towards large genus asymptotics of intersection
numbers on moduli spaces of curves. Geom. Funct. Anal., 25(4):1258–1289, 2015.

\bibitem{MP19} M. Mirzakhani and B. Petri, Lengths of closed geodesics on
random surfaces of large genus, Comment. Math. Helv. 94 (2019), no. 4,
869–889.

\bibitem{Mon20} L. Monk, Benjamini-Schramm convergence and spectrum of random hyperbolic surfaces of high genus, (2020).

\bibitem{MT20} L. Monk and J. Thomas, The tangle-free hypothesis on random hyperbolic surfaces, (2020).

\bibitem{NWX20} X. Nie, Y. Wu, and Y. Xue, Large genus asymptotics for lengths of separating closed geodesics on random surfaces, (2020).

\bibitem{Pen92} R. C. Penner, Weil-Petersson volumes, J. Differential Geom. 35 (1992), no. 3, 559–608. MR MR1163449
(93d:32029)

\bibitem{Roy71} H. L. Royden, Automorphisms and Isometries of Teichmilller Space, Advances in the Theory of Riemann Surfaces. (AM-66), Volume 66, Princeton University Press, NJ, 1971.

\bibitem{Str66} K. Strebel, Uber quadratische differentiate mit geschlossen trajectorien und extremale quasikonforme abbildungen, in Festband zum 70 Geburstag von Rolf Nevanlinna, Springer-Verlag, Berlin (1966).

\bibitem{Vai71} J. Vaisala, Lectures on $n$-dimensional quasiconformal mappings, Springer-Verlag, 1971.

\bibitem{Wri15} A. Wright, Translation surfaces and their orbit closures: An introduction for a broad audience. EMS Surv. Math. Sci. 2 (2015), 63-108. 

\bibitem{WX18} Y. Wu and Y. Xue, Small eigenvalues of closed Riemann surfaces for large genus, (2018).

\bibitem{WX21} Y. Wu and Y. Xue, Random hyperbolic surfaces of large genus have first eigenvalues
greater than $\frac{3}{16}-\epsilon$, American Journal of Math., (2021).

\end{thebibliography}
\end{document}